\documentclass[11pt]{article} 

\usepackage[margin=1in]{geometry} 
\usepackage{amsmath,amsthm,amssymb}
\usepackage{mathrsfs}
\usepackage{amsmath}
\usepackage{amssymb}
\usepackage{enumerate}
\usepackage{mathtools}
\usepackage{bbm}
\usepackage{hyperref}
\usepackage[title]{appendix}

\usepackage{url}

\usepackage{graphicx}

\usepackage{tikz-cd}
\usetikzlibrary{arrows}
\theoremstyle{plain}
\newtheorem{thm}{Theorem}[section]
\newtheorem{lem}[thm]{Lemma}
\newtheorem{cor}[thm]{Corollary}
\newtheorem{prop}[thm]{Proposition}

\theoremstyle{definition}
\newtheorem{Def}[thm]{Definition}
\newtheorem{main def}[thm]{Main Definition}
\newtheorem{ex}[thm]{Example}

\newtheorem{notation}[thm]{Notation}

\theoremstyle{remark}
\newtheorem{rmk}[thm]{Remark}

\begin{document}
 
% --------------------------------------------------------------
%                         Start here
% --------------------------------------------------------------
 
\title{Paschke Categories, K-homology and the Riemann-Roch Transformation}%replace X with the appropriate number
\author{Khashayar Sartipi\footnote{Department of Mathematics, Statistics, and Computer Science, University of Illinois at Chicago}} %if necessary, replace with your course title
\date{}
 
\maketitle

\begin{abstract}
For a separable $C^*$-algebra $A$, we introduce an exact $C^*$-category called \emph{the Paschke Category} of $A$, which is completely functorial in $A$, and show that its K-theory groups are isomorphic to the topological K-homology groups of the $C^*$-algebra $A$. Then we use the Dolbeault complex and ideas from the classical methods in Kasparov K-theory to construct an acyclic chain complex in this category, which in turn, induces a Riemann-Roch transformation in the homotopy category of spectra, from the algebraic K-theory spectrum of a complex manifold $X$, to its topological K-homology spectrum. % We examine whether this is a natural transformation with respect to proper maps of complex manifolds.%, and how we can extend it to complex spaces.
\end{abstract}

\section*{Introduction}

The main purpose of this paper is to define a Riemann-Roch transformation from the algebraic K-theory spectrum of a complex manifold to its topological K-homology spectrum. The topological K-homology spectrum of a manifold can be defined in various ways, but in this paper, we concern ourselves with definitions that use the language of $C^*$-algebras, as they provide a natural framework for the Dolbeault complex. For a separable $C^*$-algebra $A$, the K-homology spectrum of $A$ can be defined through the K-theory spectrum of the $C^*$-algebra $\mathfrak{Q}(A)$ called the \emph{Paschke dual} of $A$ \cite{paschke1981k}. However the definition of the Paschke dual depends on the choice of a representation of $A$ and is only functorial up to homotopy. Here for any separable $C^*$-algebra $A$ we introduce the Paschke \emph{category} $(\mathfrak{D}/ \mathfrak{C})_A$ of $A$ whose objects are representations of $A$ and morphisms are the quotient of pseudo-local modulo locally compact operators. Since we are considering all the representations, this category is completely functorial in $A$. We define structure of an exact $C^*$-category on the Paschke category which in particular, makes it a topological exact category, so that by applying Waldhausen's $S_{\cdot}$-construction on the Paschke category and considering the fat geometric realization, we obtain a functor from $C^*$-algebras to the category of spectra. The K-theory groups of the Paschke category are defined to be the (stable) homotopy groups of this spectrum.

We observe that the \emph{ample} representations of $A$ form a strictly cofinal subcategory of the Paschke category and through a standard argument, show that the K-theory spectrum of the Paschke category is homotopy equivalent to the K-theory spectrum of the Paschke dual of the $C^*$-algebra $A$, which gives the K-homology spectrum of $A$. We also check that the pull-back maps of the Paschke category agree with the classically defined ones up to homotopy. This makes Paschke categories a convenient place to study K-homology of $C^*$-algebras. 

By translating the arguments into the language of categories, we can replicate the constructions in bivariant K-theory \cite{kasparov1980operator} to show that the Dolbeault complex of a complex manifold $X$ with coefficients in a holomophic vector bundle $E$ induces an exact sequence in the Paschke category, obtained by considering the $L^2$-completions and applying functional calculus with respect to a normalizing function to the Dolbeault operator. %Since the same construction we are replicating, can be applied to any Dirac operator on a spin-$\mathbb{C}$ manifold in the language of bivariant K-theory, this motivates us to investigate in future if we can also obtain an exact sequence in the Paschke category corresponding to a Dirac operator on a spin-$\mathbb{C}$ manifold.
% gives a Fredholm module (which represent classes in bivariant K-theory) The similarity of the arguments leaves us wondering if any Dirac operator on a spin-$\mathbb{C}$ manifold would also give us an exact sequence, since there is 
% We do the arguments carefully, to observe how similar they are to those used to show that the Dolbeault operator induces a Fredholm module in the bivariant K-theory. This motivates us to investigate in future if the same process applies to Dirac operators on a spin-$\mathbb{C}$ manifolds.
Since the $L^2$-completion of sections of a bundle depends on the choice of the metric, then so does the exact sequence we obtain out of the Dolbeault complex, even though there are natural isomorphisms on the relatively compact open subsets. Therefore this process only makes sense on a certain category of vector bundles with a choice of metric. We show that this process induces an exact functor from that category, to the category of acyclic chain complexes in the Paschke category $(\mathfrak{D}/ \mathfrak{C})_{C_0(X)}$.

% 
%We show that the Dolbeault complex gives a functor from the K-theory of holomorphic vector bundles on $X$ to the K-theory of the category of acyclic chain complexes in the Paschke category. 
To obtain the Riemann-Roch transformation, we need to land in the loop space of the K-theory spectrum of Paschke category. To achieve this, we first note that there is a natural construction of Grayson \cite{grayson2012algebraic} in the homotopy category of spectra, from the K-theory spectrum of the category of bounded acyclic \emph{double} chain complexes to the loop space of the K-theory spectrum of the original category. Then we generalize a construction of Higson given in \cite{higson1995c} to obtain a natural functor from the category of acyclic chain complexes in the Paschke category to the category of acyclic double chain complexes in the Paschke category. %, where Grayson \cite{grayson2012algebraic} constructs a natural map (in the homotopy category of spectra) from its K-theory spectrum to the loop space of the K-theory spectrum of the original cateogry itself. 
The composition of these maps give us a Riemann-Roch transform from the K-theory spectrum of a certain category of vector bundles with metrics, to the loop space of the K-theory spectrum of the Paschke category. When we restrict this transformation to a relatively compact open subset, then this map factors through the K-theory spectrum of the original category of vector bundles. Then we can use the descent properties of the topological K-homology to glue all the maps on the relatively compact open subsets and obtain the Riemann-Roch transformation.

A functorial Riemann-Roch transformation of complex analytic spaces was defined by Levy in \cite{levy1987riemann,levy2008riemann}. We leave to a future paper the functoriality of the Riemann-Roch transformation, with respect to proper maps of complex manifolds.

This paper is organized as follows. In section \hyperref[section C*-categories]{1}, for a $C^*$-algebra $A$, we define the Paschke category (and also a variant called the Calkin-Paschke category) of $A$, as an exact $C^*$-category, and investigate its basic properties, including properties of certain subcategories of the Paschke category. In section \hyperref[section k-theory]{2}, we replicate Waldhausen's arguments to prove cofinality for topological categories, then repeat a construction of Grayson to obtain a map (in the stable homotopy category) between certain K-theory spectra, and investigate if the same holds for topological categories. Then we generalize a construction of Higson to obtain an exact functor between certain categories of chain complexes in the Paschke category. %Both of these constructions are essential to defining the Riemann-Roch transformation later.
%of the Paschke category (as an exact category) by using methods of algebraic K-theory. 
In section \hyperref[section differential operator]{3}, we use the Dolbeault complex of a complex manifold, together with methods commonly used in the bivariant K-theory, to define an exact sequence in the corresponding Paschke category. This procedure depends on the choice of metric, and we go through a careful argument to show that all the choices induce homotopic maps of spectra.
 Finally in section \hyperref[Main results section]{4}, we show that the positive K-theory groups of the Paschke category of $A$ are equal to the shifted K-homology groups of the $C^*$-algebra $A$. We also show that the natural pull-back maps agree with the classically defined ones, and use the ingredients from the previous sections to define the Riemann-Roch transformation. Finally, for a unital $C^*$-algebra $A$, we will define a natural pairing of the category of normed right projective $A$-modules, with the Paschke category of $A$, and show that it has the expected properties.
% use the Paschke category to define Euler characteristic of a topological vector bundle through the Dolbeault complex. 
% and use the natural functors defined in the previous section to obtain a Riemann-Roch transformation. 

\subsection{Notation and Terminology}

In this paper, we are only considering \emph{separable} Hilbert spaces and $C^*$-algebras.% Also, whenever we say something \emph{essentially} happens, we mean that it happens up to compact operators, e.g. if we say two bounded operators $T_1,T_2$ essentially commute, we mean $T_1 T_2 - T_2 T_1$ is a compact operator. In addition, if $T-S$ is compact, we denote $T\sim S$.\\
We will use the letters $\mathfrak{A} , \mathfrak{A}' , \mathfrak{B} , \mathfrak{C}, \ldots$ to refer to categories "somewhat related" to categories of $C^*$-algebras, and letters $\mathcal{A} , \mathcal{B} , \mathcal{P}, \ldots$ to other categories. Also, if $A,B$ are two objects in the category $\mathfrak{A}$, then we will use the notation $\mathfrak{A}(A,B)$ to denote the set (or space) of morphisms from $A$ to $B$ in the category $\mathfrak{A}$. Also, if no confusion arises, we will write $\mathfrak{A}(A)$ instead of $\mathfrak{A}(A,A)$. We will use $\mathscr{O}, \mathscr{A}, \mathscr{C}, \ldots$ to refer to certain sheaves. 

\subsection{Acknowledgment}
First and foremost, I want to thank my advisor, Henri Gillet, for sharing his ideas, and many valuable conversations and thoughtful answers to my questions.
I would also want to thank Ulrich Bunke, who pointed out a gap in one of the arguments, and Nigel Higson, for suggesting a shortcut that simplified some of the arguments. Also I want to thank Ben Antieau, and Pete Bousfield for helpful comments.

\section{$C^*$-Categories and the Paschke Category} \label{section C*-categories}

\subsection{Definitions and Basic Properties}
Let us start with giving a brief history and basic definitions of \emph{$C^*$ categories}. Karoubi first defined \emph{Banach Categories} in \cite{karoubi1968algebres}. A good source for this material is \cite{karoubi2008k}. Later \emph{$C^*$-categories} were defined in \cite{ghez1985w}. Another good source for $C^*$-categories is \cite{mitchener2002c}. %As a reminder, we have:

\begin{Def}
The category $\mathfrak{A}$ is called a \emph{complex $*$-category} if:
\begin{itemize}\itemsep -1pt
\item[A1] For each two objects $A,B$ of $\mathfrak{A}$, $\mathfrak{A}(A,B)$ is a complex vector space and composition of arrows is bilinear.
\item[A2] There is an involution antilinear contravariant endofunctor $*$ of $\mathfrak{A}$ which preserves objects. The image of $x$ under $*$ will be denoted by $x^*$. It follows that each $\mathfrak{A}(A,A)$ is a $*$-algebra with identity.
\item[A3] For each $x\in \mathfrak{A}(A,B)$, $x^*x$ is a positive element of the $*$-algebra $\mathfrak{A}(A,A)$, i.e. $x^*x=y^*y$ for some $y\in \mathfrak{A}(A,A)$. Furthermore, $x^*x=0$ implies $x=0$.

It follows that the mapping $\mathfrak{A}(A,B) \times \mathfrak{A}(A,B) \rightarrow \mathfrak{A}(A,A)$ defined by $(x,y) \mapsto x^*y$ is a $\mathfrak{A}(A,A)$-valued inner product on the right $\mathfrak{A}(A,A)$-module $\mathfrak{A}(A,B)$, where $\mathfrak{A}(A,A)$ acts on $\mathfrak{A}(A,B)$ by composition of arrows.

A $*$-category $\mathfrak{A}$ is called a \emph{normed $*$-category} if:
\item[A4] Each $\mathfrak{A}(A,B)$ is a normed space and $\Vert xy \Vert \leq \Vert x \Vert \Vert y \Vert$.

A normed $*$-category is called a \emph{Banach $*$-category} if: 
\item[A5] Each $\mathfrak{A}(A,B)$ is a Banach space. 

A Banach $*$-category is called a \emph{$C^*$-category} if:
\item[A6] For each arrow $x$ of $\mathfrak{A}$, $\Vert x \Vert ^2  = \Vert x^* x \Vert$.
\end{itemize}
 
\end{Def}

It follows that each $\mathfrak{A}(A,A)$ is a $C^*$-algebra with identity. A6 shows that the norm on a $C^*$-category is uniquely determined by the norms on the $C^*$-algebra $\mathfrak{A}(A,A)$. In fact, we can say more: Let $\mathfrak{A}$ be a $*$-category where each $\mathfrak{A}(A,A)$ is a $C^*$-algebra, then $\mathfrak{A}$ can be made into a normed $*$-category satisfying A6 (but not A5 in general) in a unique way by setting $\Vert x  \Vert =  \Vert x^*x \Vert ^{1/2} $. 

Of course any $C^*$-algebra with identity can be considered as a $C^*$-category with a single object. 

\begin{Def}
Let $\mathfrak{A}, \mathfrak{A}'$ be $C^*$-categories. Then a functor $F: \mathfrak{A} \rightarrow \mathfrak{A}'$ is called a \emph{$*$-functor} if it is a linear functor (i.e. $F: \mathfrak{A}(A,B) \rightarrow \mathfrak{A}'(F(A),F(B))$ is linear for all objects $A,B$ of $\mathfrak{A}$.) and also $F(x)^*=F(x^*)$ for all morphisms $x$ in $\mathfrak{A}$.
\end{Def}

\begin{Def}\cite[3.1.]{mitchener2002c}
A \emph{non-unital category}, is a category of objects and morphisms similar to a category, except that there need not exist an identity morphism $1\in Hom(A,A)$ for each object $A$. 
A \emph{non-unital functor} $F:\mathfrak{A} \rightarrow \mathfrak{B}$ between (possibly non-unital) categories $\mathfrak{A}, \mathfrak{B}$ is a transformation similar to a functor, except that there is no condition on the identity morphisms of the category $\mathfrak{A}$.
Similarly, we can define non-unital $C^*$-categories, and $*$-functors between them.
\end{Def}

\begin{Def}\cite[4.2.]{mitchener2002c}
Let $\mathfrak{A}$ be a $C^*$-category, then a \emph{$C^*$-ideal} $\mathfrak{I}$ in the category $\mathfrak{A}$ is (a probably non-unital) subcategory of $\mathfrak{A}$ so that:
\begin{itemize}\itemsep -1pt
\item The subcategory $\mathfrak{I}$ has the same objects as the category $\mathfrak{A}$.
\item Each morphism set $\mathfrak{I}(A,B)$ is a norm closed subspace of the space $\mathfrak{A}(A,B)$.
\item The composition of an arrow in the category $\mathfrak{A}$ with an arrow in the subcategory $\mathfrak{I}$ is an arrow in the subcategory $\mathfrak{I}$.
\end{itemize}
\end{Def}

As a result of the definition above we have:
\begin{prop}\cite[4.7.]{mitchener2002c}
Let $j\in \mathfrak{I}(A,B)$ be a morphism in the $C^*$-ideal $\mathfrak{I}$ of the $C^*$-category $\mathfrak{A}$. Then the adjoint morphism $j^*$ is also a morphism in the ideal $\mathfrak{I}$.

Also, we can define the quotient $\mathfrak{A}/\mathfrak{I}$ to be the category with the same objects as $\mathfrak{A}$ and with morphism sets the quotient Banach space $$(\mathfrak{A}/ \mathfrak{I})(A,B) = \frac{\mathfrak{A} (A,B) }{ \mathfrak{I} (A,B)}.$$
This is also a $C^*$-category.
\end{prop}

\begin{ex}
We will use $\mathfrak{B}$ to denote the category of Hilbert spaces with bounded operators between them, which is an additive $C^*$-category as products and coproducts of a finite number of Hilbert spaces is just their direct summand, and also the set of bounded operators between two Hilbert spaces forms an abelian group, as we can add the operators with each other and composition on both sides is linear.

We will denote the $C^*$-ideal of compact operators by $\mathfrak{K}$.
\end{ex}

\begin{ex}
Let $A$ be a $C^*$-algebra, and let $\mathfrak{A}$ denote a $C^*$-category. Let $\mathcal{R}ep_{\mathfrak{A}}(A)$ denote the category of \emph{representations of $A$}, i.e. a category whose objects are representations $\rho:A \rightarrow \mathfrak{A}(H)$, where $H$ is an object in $\mathfrak{A}$, and whose morphisms between two representations $\rho_1:A\rightarrow \mathfrak{A}(H_1)$ and $\rho_2:A \rightarrow \mathfrak{A}(H_2)$ is the Banach space $\mathfrak{A}(H_1,H_2)$. Notice that we are \emph{not} restricting our attention to unital representations, i.e. we also consider the zero representation, and other non-unital representations.

If $\mathfrak{A}$ is additive, then it is easy to check that $\mathcal{R}ep_{\mathfrak{A}}(A)$ is an additive $C^*$-category as well.
\end{ex}

\begin{Def}\label{pseudolocal-and-locallycompact Def}
Let $A$ be a $C^*$-algebra, and let $\rho_i:A \rightarrow \mathfrak{B}(H_i)$, be representations of $A$ for $i=1,2$. A bounded operator $T:H_1\to H_2$ is called \emph{pseudo-local}, if $\rho_2(a)T - T\rho_1(a) \in \mathfrak{K}(H_1,H_2), \forall a\in A$, and $T$ is \emph{locally compact}, if both $\rho_2(a)T , T\rho_1(a)$ are in $\mathfrak{K}(H_1,H_2)$ for all $a\in A$.
\end{Def}

\begin{Def}
Let $A$ be a $C^*$-algebra. Then we define the \emph{Paschke category} of $A$ to be the quotient category $(\mathfrak{D}/ \mathfrak{C})_A \coloneqq \mathfrak{D}_A/ \mathfrak{C}_A$, where $\mathfrak{D}_A$ is the category of representations $\rho:A\to \mathfrak{B}(H)$ of $A$, where the morphisms between two representations are the pseduo-local operators between them, and the $C^*$-ideal $\mathfrak{C}_A$ has the same objects, but the morphisms are locally compact operators.

We define the \emph{Calkin-Paschke category} of $A$ to be the category where the objects are representations $\rho':A\to (\mathfrak{B}/ \mathfrak{K})(H)$, and morphisms are again the quotient of pseudo-local operators modulo locally compact operators. We denote the Calkin-Paschke category by $(\mathfrak{D}/ \mathfrak{C})'_A$.
\end{Def}
\noindent
Note that there is a natural functor $(\mathfrak{D}/ \mathfrak{C})_A\to (\mathfrak{D}/ \mathfrak{C})'_A$, which sends a representation $\rho:A\to \mathfrak{B}(H)$ to $\rho': A\to \mathfrak{B}(H)\to (\mathfrak{B}/ \mathfrak{K})(H)$.
\begin{notation}
From now on, we will use letters such as $\rho,\nu$ to refer to objects of the Paschke category, and use similar letters with "primes", i.e. $\rho', \nu'$ to refer to objects of the Calkin-Paschke categories.

In case no confusion should arise, instead of writing $T\rho(a)$ is compact for all $a\in A$, we will simply say $T\rho$ is compact, and similarly for $\rho'$.
\end{notation}

% we can define $$\mathfrak{D}(\rho_1,\rho_2)= \{T \in \mathfrak{B}(H_1,H_2) \vert \rho_2(a)T - T\rho_1(a) \in \mathfrak{K}(H_1,H_2), \forall a\in A\}$$ to be the \emph{pseudo-local} operators and $$\mathfrak{C}(\rho_1,\rho_2) = \{T \in \mathfrak{B}(H_1,H_2) \vert \rho_2(a)T , T\rho_1(a) \in \mathfrak{K}(H_1,H_2), \forall a\in A\}$$ to be the \emph{locally-compact} operators.
%%(we can similarly consider $\mathfrak{D}_{\rho '}$ and $\mathfrak{C}_{\rho '}$ instead). Then it is easy to show that $(\mathfrak{D}_{\rho}, \mathfrak{C}_{\rho})$ (and also $(\mathfrak{D}_{\rho'}, \mathfrak{C}_{\rho'})$) is a Calkin pair.\\
% Similarly, we can define $\mathfrak{D}(\rho'_1, \rho'_2)$ and $\mathfrak{C}(\rho'_1,\rho'_2)$ to get a Calkin-Paschke category.
%% Note that the previous example, was a special case of this example when $A=\mathbb{C}$.
% \end{ex}

\begin{ex}[See {\cite[5.3.2.]{higson2000analytic}}] \label{ex-relative pseudo-local operators}
One can generalize the definition above, by introducing a "relative" version. Let $A$ be a $C^*$-algebra and $I\subset A$ a $C^*$-ideal. Then for representations $\rho_i:A\rightarrow \mathfrak{B}(H_i)$ for $i=1,2$, define $\mathfrak{D}_A(\rho_1, \rho_2)$ to be the same as the above example, and let $$\mathfrak{C}_{I,A}(\rho_1,\rho_2)= \{T\in \mathfrak{D}_A(\rho_1,\rho_2) \vert T\rho(a) , \rho(a)T\in \mathfrak{K}(H_1,H_2), \, \forall a\in I \}.$$
% We will denote the corresponding Paschke category by $(\mathfrak{D}/ \mathfrak{C})_{I,A}$ instead of $(\mathfrak{D} / \mathfrak{C}_{I,A})_A$. \\
Note that when $I\subset J$ then $\mathfrak{C}_A \subset \mathfrak{C}_{J,A} \subset \mathfrak{C}_{I,A} $, and if $I=A$, then we recover the definition above. All of the results on the Paschke category also holds for this relative version, however, by theorem \ref{K-homology of paschke category is k-homology} and excision for K-homology (cf. \cite[5.4.5.]{higson2000analytic}), this does not provide any new information.
%We will later use this to define relative K-homology.
\end{ex}

\begin{Def}\cite[1.1.6.7.]{karoubi2008k} \cite[Def 8.]{kandelaki2000kk}
Let $\mathfrak{A}$ be an additive category. Then $\mathfrak{A}$ is called \emph{pseudo-abelian} if for each object $H$ of $\mathfrak{A}$ and every morphism $p:H \rightarrow H$ so that $p^2=p$, the kernel of $p$ exists.

In the case when $\mathfrak{A}$ is an additive $C^*$-category, and each self-adjoint projection has a kernel, then we say $\mathfrak{A}$ is \emph{weakly pseudo-abelian}.
\end{Def}

\begin{prop}\cite[1.1.6.9.]{karoubi2008k}
Let $\mathfrak{A}$ be a (weakly) pseudo-abelian category, let $H$  be an object of $\mathfrak{A}$ and let $p:H\rightarrow H$  be such that $p^2=p$ (and also $p=p^*$). Then the object $H$ splits into the direct sum $H=\ker(p) \oplus \ker(1-p)$.
\end{prop}

\begin{prop}\cite[1.1.6.10.]{karoubi2008k} \cite[Thm 9.]{kandelaki2000kk} \label{pseudo-abelian category}
Let $\mathfrak{A}$ be an additive category. Then there exists a pseudo-abelian category $\tilde{\mathfrak{A}}$, and an additive functor $\phi: \mathfrak{A} \rightarrow \tilde{\mathfrak{A}}$ which is fully faithful and is universal among additive functors from $\mathfrak{A}$ to a pseudo-abelian category. The pair $(\phi, \tilde{\mathfrak{A}})$ is unique up to equivalence of categories.

$\tilde{\mathfrak{A}}$ is equivalent to the category where objects are pairs $(H,p)$ where $H$ is an object in $\mathfrak{A}$ and $p:H\rightarrow H$ is a projector (i.e. $p^2=p$), and morphisms between $(H_1,p_1)$ and $(H_2,p_2)$ are morphisms $f:H_1\rightarrow H_2$ in $\mathfrak{A}$ such that $fp_1 = p_2f=f$ in $\mathfrak{A}$. This category is called the \emph{pseudo-abelianization} of $\mathfrak{A}$.

The same statement is true for a $C^*$-category and its weakly pseudo-abelian counterpart. 
\end{prop}

\begin{prop}
The weak pseudo-abelianization of the $C^*$-category $(\mathfrak{B}/ \mathfrak{K})$ is naturally isomorphic to the Calkin-Paschke category $(\mathfrak{D}/\mathfrak{C})'_{\mathbb{C}}$.
\end{prop}

\begin{proof}
The objects in the Calkin-Paschke category $(\mathfrak{D}/\mathfrak{C})'_{\mathbb{C}}$ can be considered as pairs $(H,\rho' (1))$ of a Hilbert space $H$ and a self-adjoint projection $p=\rho'(1)\in (\mathfrak{B}/ \mathfrak{K})(H)$, and morphisms $\rho'_1\to \rho'_2$ in $(\mathfrak{D}/ \mathfrak{C})'_{\mathbb{C}}$ are the pseudo-local operators modulo locally compact ones, i.e. the operators $F\in (\mathfrak{B}/ \mathfrak{K})(H_1,H_2)$ so that $F p_1 = p_2 F$, modulo the ones that $F p_1=0 = p_2 F$. In other words since $F(1-p_1) $ is locally compact, hence $F= F p_1=p_2 F$ in the Calkin-Paschke category $(\mathfrak{D}/ \mathfrak{C})'_{\mathbb{C}}$. 
Therefore we have a natural functor $(\mathfrak{D}/ \mathfrak{C})'_{\mathbb{C}}\to \widetilde{(\mathfrak{B}/\mathfrak{K})} $. 

This functor is faithful, because $F=F p_1 = p_2 F$ are all zero in the category $(\mathfrak{B}/ \mathfrak{K})$ iff $F$ is locally compact in the Calkin-Paschke category. The functor is also full, because any $F\in (\mathfrak{B}/ \mathfrak{K})(H_1,H_2)$ that satisfies $F p_1 = p_2 F$ is pseudo-local.
% Therefore we have a faithful embedding of $(\mathfrak{D}/\mathfrak{C})'_{\mathbb{C}}$ in $\widetilde{(\mathfrak{B}/\mathfrak{K})}$. This embedding is full, because if the morphism $F$ from $(H_1,p_1)$ to $(H_2,p_2)$ is locally compact, then $F\sim F p_1 \sim 0$, i.e. $F$ corresponds to a compact operator, which is equivalent to zero in $\widetilde{(\mathfrak{B}/\mathfrak{K})}$.
% the only equivalence class of operators between $(H_1, p_1)$ and $(H_2, p_2)$ in the pseudo-abelianization of $\mathfrak{B} /\mathfrak{K}$ which corresponds to locally compact operators in $(\mathfrak{D}/\mathfrak{C})'_{\mathbb{C}}$ is equivalence class of compact operators.
\end{proof}

\begin{prop}\label{Calkin-Paschke category is pseudo abelian-proposition}
The Calkin-Paschke category $(\mathfrak{D}/\mathfrak{C})'_A$ is  a weakly pseudo-abelian category. 
\end{prop}

We need the following two lemmas to prove the proposition above.
%The main ingredients for proving the proposition above are The following two lemmas.

\begin{lem} \label{projection has an actual representative lemma}
% Let $(\mathfrak{D}/ \mathfrak{C})'_A$ be the ..... Pascke category .
Each self-adjoint projection in $(\mathfrak{D}/ \mathfrak{C})'_A$ has a representative in $\mathfrak{D}_A$ which is a self-adjoint projection. 
\end{lem}

\begin{proof}
% First, we prove the assertion when $(\mathfrak{D}/\mathfrak{C})'_A$ is the Paschke category of example \ref{ex-pseudolocal-and-locallycompact}. 
Let $P\in \mathfrak{D}_A(\rho')$ be a representative for a self-adjoint  projection in $(\mathfrak{D}/\mathfrak{C})'_A$. Hence $\rho'(P-P^*)$ and $(P - P^*)\rho'$ % , \rho_2(P_1 - P_1 ^2) , (P_1 - P_1 ^2)\rho_1$
 are compact operators. Set $P'=(P + P^*)/2 $. Then $P'$ is a self-adjoint operator, hence by Weyl-Von Neumann Theorem \cite[2.2.5.]{higson2000analytic} there is a diagonal compact perturbation of $P'$, i.e. there exists an operator $P_1$ with an orthonormal basis $\{ e_i\}_{i=1}^{\infty}$ of eigenvectors of $P_1$ for $H$ with complex numbers $\lambda_i$ as eigenvalues so that and $P_1-P'$ is a compact operator. Therefore $P-P_1$ is in $\mathfrak{C}_A(\rho')$. 
 
Let $I\subset \mathbb{N}$ be the set of indices $i$ such that $ \vert \lambda_i \vert < 1/2$. Now define the bounded operator $Q$ by $Q(e_i) = e_i $ if $i \notin I$, and set $Q(e_i)=0$ otherwise. Evidently, $Q$ is a self-adjoint projection in the category $\mathfrak{B}$. We want to show that $P_1-Q$ is in $\mathfrak{C}_A(\rho')$. 

\noindent
Define $D(e_i) = \frac{1}{1-\lambda_i} e_i$ if $i\in I$, and $D(e_i)= \frac{1}{\lambda_i} e_i $ otherwise. Notice that $D$ is a bounded diagonal operator (of norm at most $2$). Also $(P_1-Q)e_i = \lambda_i e_i$ when $i\in I$, and $(P_1 - Q)e_i = (\lambda_i - 1)e_i$ otherwise. Furthermore $(P_1-P_1 ^2)e_i =(\lambda_i - \lambda_i ^2)e_i $ for all $i\in \mathbb{N}$. Therefore $(P_1-Q)(e_i) = D(P_1-P_1 ^2)(e_i) = (P_1-P_1 ^2)D(e_i) $ for all $i$, hence $P_1-Q=D(P_1-P_1 ^2)=(P_1-P_1 ^2)D$. But since $P_1-P_1 ^2\in \mathfrak{C}_A(\rho')$, then $(P_1-P_1 ^2)\rho' , \rho'(P_1-P_1 ^2)$ are compact operators. Therefore $(P_1 -Q)\rho', \rho'(P_1 -Q)$ are also compact, which proves that $P_1  -Q \in \mathfrak{C}_A$ and hence $Q\in \mathfrak{D}_A(\rho')$. % \\ The proof in the case of Example \ref{easy example paschke category} is exactly similar.
\end{proof}

\begin{lem}\label{inclusion and projection pseudo-local}
Let $T\in \mathfrak{D}_A(\rho'_1, \rho'_2 )$ be a pseudo-local operator with closed image $V_2 \subset H_2$. Let $V_1 \subset H_1$ be the orthogonal complement of $\ker(T)$. Then for $i=1,2$, the projections $\pi_i:H_2 \rightarrow V_i$ and the inclusions $\iota _i :V_i \rightarrow H_i$ are pseudo-local operators. 
\end{lem}

\begin{proof}
Since $T$ has a closed image, then it induces an isomorphism of Hilbert spaces from $V_1$ to $V_2$. To simplify the notation, denote $T'= \pi_2 T \iota _1:V_1 \rightarrow V_2$. Let $S'\in \mathfrak{B}(V_2, V_1)$ be the inverse to $T'$ and 
% so that $S' T' = Id_{V_1}$ and $T' S' = Id_{V_2}$. Now 
set $S= S' \oplus 0 : V_2 \oplus V_2 ^{\perp} = H_2 \rightarrow V_1 \oplus V_1 ^{\perp} = H_1$. We have $ST=\iota_1 \pi_1$ and $TS =\iota_2 \pi_2$. First we show that $S$ is also pseudo-local. This would show that $\iota_i, \pi_i$ are pseudo-local. 

Let $ \rho'_i = \left( \begin{smallmatrix} \rho'^{11}_i & \rho'^{12}_i \\ \rho'^{21}_i & \rho'^{22}_i \end{smallmatrix}\right) \in (\mathfrak{B}/ \mathfrak{K})(V_i \oplus V_i^{\perp})$ 
for $i=1,2$. Since $T$ is pseudo-local, we have $T'\rho'^{11}_1 - \rho'^{11}_2 T'$ and $T' \rho'^{12}_2 $ and  $\rho'^{21}_2 T'$ are compact \footnote{Where in here, we are denoting the operator induced by $T'$ in $ (\mathfrak{B}/ \mathfrak{K})(V_1,V_2)$ again by $T'$. We abuse notation in a similar way for $S',S,T, \pi_i, \iota_i$.}. Therefore $\rho'^{12}_1 = S'T'\rho'^{12}_1 = 0 $ and $\rho'^{21}_2 = \rho'^{21}_2 T'S' = 0$. Also, since $\rho'^{12}_i(a)^* = \rho'^{21}_i(a^*) $, then we can say that $\rho'^{12}_i, \rho'^{21}_i$ are zero, for $i=1,2$. Therefore $\rho'_1 S - S\rho'_2 = (\rho'^{11}_1 S'-S'\rho'^{11}_2) \oplus 0 $. But $(\rho'^{11}_1 S'-S'\rho'^{11}_2 )= S'T'(\rho'^{11}_1 S'-S'\rho'^{11}_2) = S'\rho'^{11}_2 T'S' - S' \rho'^{11}_2 =0$. Hence $S$ is pseudo-local. 

We have $\iota_i \rho'^{11}_i - \rho'_i \iota_i = (\rho'^{11}_i ,0) - (\rho'^{11}_i, \rho'^{12}_i ) = 0 $. This proves that $\iota_i$ is pseudo-local, for $i=1,2$. It only remains to show that $\rho'^{11}_i=\pi _i \rho'_i \iota _i:A \rightarrow (\mathfrak{B}/ \mathfrak{K})(V_i)$ is an object of the Calkin-Paschke category $(\mathfrak{D}/ \mathfrak{C})'_A$. 
% gives a representation for $i=1,2$, 
This follows from adjointness of $\iota_i$ and $\pi_i$ and \footnote{Note that this part of the argument only works in the Calkin-Paschke category. In fact, this is the only part of the proof of pseudo-abelianness of the Calkin-Paschke category that does not apply to the Paschke category $(\mathfrak{D}/ \mathfrak{C})_A$.} 
\begin{center}
$\rho'^{11}_i (ab^*)= \pi_i \rho'_i(a) \rho'_i(b^*) \iota_i = \pi_i \iota _i \pi_i \rho'_i(a) \rho'_i(b^*) \iota _i   = \pi _i \rho'_i(a) \iota _i (\pi_i \rho'_i(b) \iota_i )^*.$
\end{center} 
\end{proof}
% \pi_i \rho'_i(a) \pi^* _i \pi_i \rho'_i(b)^* \pi^* _i = 

\begin{proof}[Proof of Proposition]
Let $P\in \mathfrak{D}_A(\rho')$ be a representative for a self-adjoint  projection in $(\mathfrak{D}/\mathfrak{C})'_A$, and let $Q$ be the projection as in proof of lemma \ref{projection has an actual representative lemma} (we will only use the fact that $Q$ has a closed image). Let $\iota$ be the inclusion of $\ker(Q)$ in $H$, and let $\pi$ be the projection onto the kernel. Then we want to show that $\iota$ is a kernel for $Q$. We clearly have $Q \iota =0$. Also since $Q$ is bounded from below on the orthogonal complement of its kernel, then it has closed image. It follows from lemma \ref{inclusion and projection pseudo-local} that $\iota$ is pseudo-local. Let $Q' $ denote "inverse" of $Q$ restricted to orthogonal complement of $\ker(Q)$, i.e. $Q'$ sends image of the projection $Q$ isometrically to the orthogonal complement of $\ker(Q)$.

Now let $F\in \mathfrak{D}_A(\rho'_0, \rho')$ be an operator so that $QF \in \mathfrak{C}_A(\rho'_0, \rho')$. Then we want to show that $F$ factors through $\iota$ up to locally compact operators. We have $ \iota \pi F = F$ modulo compact operators because $(Id_{H}-\iota \pi)F = Q'QF =0 $. Also, if we have $\rho'(\iota G - F) = 0 $, then $\rho' \iota (G -\pi F)= 0$. Therefore $\rho'(G -\pi F) = 0$. This completes the proof.
\end{proof}

\begin{Def}
Let $\mathfrak{A}$ be an additive category. Then we say that a chain complex $$ \ldots \xrightarrow{T_{i-1}} \rho_i \xrightarrow{T_i} \rho_{i+1} \xrightarrow{T_{i+1}} \rho_{i+2} \xrightarrow{T_{i+2}} \ldots $$ is \emph{exact} if there is a \emph{contracting homotopy}, i.e. if there are morphisms $S_i$ in $\mathfrak{A}$ from $\rho_{i+1}$ to $\rho_i$ so that $T_{i-1} S_{i-1} + S_i T_i = Id_{\rho_i}$ in $\mathfrak{A}$.

As a result of this definition, every short exact sequence in $\mathfrak{A}$ is split, hence $\mathfrak{A}$ is an \emph{exact category} in the sence of Quillen \cite[Sec 2.]{quillen1973higher}. Note that this does not mean all exact sequences are split.

\end{Def}

In particular, (using the definition above) the Paschke category $(\mathfrak{D}/ \mathfrak{C})_A$, the Calkin-Paschke category $(\mathfrak{D}/ \mathfrak{C})'_A$, and also $\mathfrak{D}_A, \mathfrak{B}, (\mathfrak{B}/ \mathfrak{K})$ are all exact $C^*$-categories.

Notice that a map $f:A\rightarrow B$ of $C^*$-algebras, induces pull-back maps $f^*: (\mathfrak{D}/\mathfrak{C})_B \rightarrow (\mathfrak{D}/\mathfrak{C})_A$ and also $f^*: (\mathfrak{D}/\mathfrak{C})'_B \rightarrow (\mathfrak{D}/\mathfrak{C})'_A$ of categories, by precomposing with the representation. This map preserves exact sequences, hence the pull-back functor is exact, and this process is functorial.
%This map sends an exact sequence to an exact sequence, hence it is an exact functor, and this process is functorial.

\subsection{Subcategories of the Paschke Category}

We start this subsection by giving a definition similar to \cite{waldhausen1985algebraic} \footnote{This was originally defined for Waldhausen categories by considering coproduct instead of the direct sum. In this paper, we will only apply the definition to certain Waldhausen categories.}.
 
\begin{Def}
Let $\mathcal{B}$ be an additive category. Then a full additive subcategory $\mathcal{A}$ is called \emph{cofinal} if for every object $B$ of the category $\mathcal{B}$, there is an object $B'$ in $\mathcal{B}$ so that $B \oplus B'$ is isomorphic to an object in $\mathcal{A}$. If we can always take $B'$ to be an object in $\mathcal{A}$, then $\mathcal{A}$ is called \emph{strictly cofinal}.

In case the category $\mathcal{B}$ is exact, we require the subcategory $\mathcal{A}$ to be exact as well. 
\end{Def}

Let us recall some definitions and useful properties of representations.

\begin{Def}
A representation $\rho:A \rightarrow \mathfrak{B}(H)$ of a $C^*$-algebra is called \emph{non-degenerate} if $\rho(A)H$ is a dense subset of $H$ (or equivalently, it is the whole $H$, cf. \cite[1.9.17.]{higson2000analytic}.). Another equivalent definition is that for each $h\in H, h\neq 0$, there exists an $a\in A$ so that $\rho(a)h \neq 0$.

A representation $\rho:A\to \mathfrak{B}(H)$ is called \emph{ample} if it is non-degenerate, and also for each $a\in A, a \neq 0$, $\rho(a)$ is not a compact operator.
\end{Def}

\begin{prop} \label{ample reps cofinal subcat- prop}
Let $\mathfrak{Q}_A$ denote the full subcategory of $(\mathfrak{D}/ \mathfrak{C})_A$ whose objects are ample representations, together with the zero representation $A\to 0$. This is an exact strictly cofinal subcategory of $(\mathfrak{D}/ \mathfrak{C})_A$.%, which is also strictly cofinal.

% Therefore, by .......[find citation] the K-theory spaces $K((\mathfrak{D}/ \mathfrak{C})_A)$ and $K(\mathfrak{Q}_A)$ are homotopy equivalent.
% In fact, it satisfies the hypothesis of proposition \ref{Waldhausen cofinality proposition with stronger assumption}.
Let $A$ be a unital $C^*$-algebra and let $\mathfrak{Q}'_A$ denote the full subcategory of $(\mathfrak{D}/ \mathfrak{C})'_A$ whose objects are unital injective representations, together with the zero representation $A\to 0$. This is an exact strictly cofinal subcategory of $(\mathfrak{D}/ \mathfrak{C})'_A$. 
\end{prop}

\begin{proof}
Note that direct sum of two non-degenerate representations is non-degenerate, and direct sum of a non-degenerate and an ample representation is ample.
Given some representation $\rho:A \rightarrow \mathfrak{B}(H)$, let $H_1 = \overline{\rho(A)H} \subset H$, $\pi: H \rightarrow H_1$ be the orthogonal projection onto the closed subspace, and let $\iota : H_1 \rightarrow H$ be the inclusion. Then we can define $\rho_1 : A \rightarrow \mathfrak{B}(H_1)$ by $\rho_1 = \pi \rho \iota$. Since $\pi, \iota$ are adjoints to each other, and $\iota \pi \rho = \rho = \rho \iota \pi$ then $\rho_1$ is indeed a representation. Also these two representations are isomorphic as objects of $(\mathfrak{D}/ \mathfrak{C})_{A}$, as $\pi, \iota $ are pseudo-local and induce the isomorphism. %, $\pi  \iota = Id_{H_1}$, and that $\rho(a)(Id_H - \iota \pi) = 0 = (Id_{H} - \iota \pi )\rho(a)$. 
But $\rho_1$ is a non-degenerate representation. Hence for any object $\rho$ of $(\mathfrak{D}/\mathfrak{C})_A$, and any ample representation $\rho_0$ of $A$, $\rho \oplus \rho_0$ is isomorphic to an ample representation in $(\mathfrak{D}/\mathfrak{C})_{A}$. 

If the $C^*$-algebra $A$ is unital and $\rho':A\to (\mathfrak{B}/ \mathfrak{K})(H)$ is an object of $(\mathfrak{D}/\mathfrak{C})'_A$, then by lemma \ref{projection has an actual representative lemma} $\rho'(1)$ has a representative $\pi\in \mathfrak{B}(H)$ which is a self-adjoint projection. By repeating the argument above, $\rho'$ is isomorphic to the unital representation $\rho'_1:A\to (\mathfrak{B}/ \mathfrak{K})(H_1)$, where $H_1\subset H$ is image of $\pi$. Also, direct sum of an injective representation $\rho'_0$ with any representation $\rho'$ is injective.

\end{proof}

Another important property of ample representations, is the following corollary of Voiculescu's theorem \cite{voiculescu1976non}, which we mention similar to as stated in \cite[3.4.2.]{higson2000analytic}.

\begin{thm}[Voiculescu]
Let $A$ be a unital $C^*$-algebra, let $\rho:A\to \mathfrak{B}(H)$ be a non-degenerate representation, and let $\nu':A\to (\mathfrak{B}/ \mathfrak{K})(H')$ be an object in $(\mathfrak{D}/ \mathfrak{C})'_A$. Assume that for each $a\in A$ with $\rho(a)\in \mathfrak{K}(H)$, we have $\nu'(a)=0$. Then there exists an isometry $V:H'\to H$ so that $V^* \rho(a) V - \nu'(a) = 0$ for all $a\in A$ \footnote{We are abusing the notation for the class of $V^* \rho(a)V$ in $(\mathfrak{B}/ \mathfrak{K})(H')$.}.
% Let $A, B$ be unital $C^*$-algebras, $\rho_A :A \rightarrow \mathfrak{B}(H_A)$ be an ample representation, and $\rho_B :B \rightarrow \mathfrak{B}(H_B)$ be a non-degenerate representation , and let $\alpha:A \rightarrow B$ be a unital $*$-homomorphism. Then there exists an isometry $V:H_B \rightarrow H_A$ so that $V^* \rho_A(a) V - \rho_B(\alpha(a))$ is compact for all $a\in A$. In this case, we say $V$ \emph{covers} $\alpha$.
%Let $A$ be a unital $C^*$-algebra, and let $\sigma:A \rightarrow \mathfrak{B}(H') $ be a completely positive map, and $\rho:A\rightarrow \mathfrak{B}(H)$ a non-degenerate representation. If whenever $\rho(a) \in \mathfrak{K}(H)$, then we have $\sigma(a)=0$, then there exists an isometry $V:H' \rightarrow H$ so that $\sigma(a) - V^*\rho(a)V$ is always compact.
\end{thm}
\noindent
An important corollary is the following:

\begin{cor}\label{Voiculescu's theorem ample reps-cor}
Let $\rho_1, \rho_2$ be two ample representations of the $C^*$-algebra $A$. Then there is a unitary operator $U:H_1 \rightarrow H_2$ so that $U \rho_1(a) U^* - \rho_2(a)  $ is compact for all $a$. 
\end{cor}
\noindent
In other words, any two ample representations in the Paschke category are isomorphic. Hence we can denote the isomorphism class of automorphisms of an ample representation $\rho$ of $A$ by $\mathfrak{Q}(A)$, which is also known as the \emph{Paschke dual}. %(we will denote the representation by a subscript.$\mathfrak{Q}_{\rho}(A)$, and isomorphism class of it by .

\begin{rmk}\label{paschke cats cofinal comparison-remark}
The natural map $(\mathfrak{D}/ \mathfrak{C})_A\to (\mathfrak{D}/ \mathfrak{C})'_A$ is fully faithful, and by Voiculescu's theorem, each object $\nu'$ of $(\mathfrak{D}/ \mathfrak{C})'_A$ has an admissible monomorphism to an object which lifts to a non-degenerate representation of $A$. Since $(\mathfrak{D}/ \mathfrak{C})'_A$ is weakly pseudo-abelian,
therefore by Voiculescu's theorem, the full subcategory of the Calkin-Paschke category $(\mathfrak{D}/ \mathfrak{C})'_A$ consisting of objects which lift to the Paschke category $(\mathfrak{D}/\mathfrak{C})_A$ is cofinal.
%Since every exact sequence in $(\mathfrak{D}/ \mathfrak{C})'_A$ splits, then this means by Waldhausen cofinality that the $i$'th K-groups of both Paschke category and Calkin-Paschke category are isomorphic when $i\geq 1$. 
\end{rmk}

% Let $A$ be a unital $C^*$-algebra and let $\rho':A\to (\mathfrak{B}/ \mathfrak{K})(V), \nu':A\to (\mathfrak{B}/ \mathfrak{K})(W)$ be two objects of $\mathfrak{Q}'_A$ which are isomorphic. Therefore there exists morphisms of the Paschke category $F\in (\mathfrak{B}/ \mathfrak{K})(V,W), G\in (\mathfrak{B}/ \mathfrak{K})(W,V)$ so that $(GF-Id_V)\rho'(a)=0, (FG-Id_{W})\nu'(a)=0$. But since both $\rho',\nu'$ are unital then $F,G$ are Fredholm operators and hence inverse to each other (modulo compact operators). Hence we have $F\rho'(ab^*) G = \nu'(ab^*) = \nu'(a)\nu'(b)^*= F\rho'(a)GG^* \rho'(b)^*F^*$ for all $a,b\in A$.

\section{K-Theory} \label{section k-theory}
% In this section, we compute K-theory of the Paschke category, and how it behaves with respect to push-forward maps.

\subsection{Waldhausen's Cofinality}
In this subsection, we recall some standard facts about K-theory and fix our notation for the rest of the section. Aside from Waldhausen's original paper \cite{waldhausen1985algebraic}, a good source for more information is \cite{weibel2013k}.

\begin{Def} \cite[A.]{segal1974categories}
Let $X_{\cdot}$ be a simplicial space. Then define the topological space $\Vert X_{\cdot} \Vert$ called the \emph{fat geometric realization} of $X_{\cdot}$, as the quotient 
\begin{equation*}
\coprod_n X_n \times \Delta^n_{top} / \sim_{+}
\end{equation*}
where the relation $\sim_{+}$ is generated by $(x,f_*p)\sim_{+}(f^*x, p)$ for $x\in X_n, p\in \Delta^m_{top}$, whenever $f:[m]\to [n]$ in the simplex category $\Delta$ is a \emph{face} (injective) map.

Let $X_{\cdot}$ be a simplicial set (or a \emph{discrete} simplicial space). Then define the topological space $\vert X_{\cdot} \vert$ called the \emph{geometric realization} of $X_{\cdot}$, as the quotient 
\begin{equation*}
\coprod_n X_n \times \Delta^n_{top} / \sim
\end{equation*}
where the relation $\sim$ is generated by $(x,f_*p)\sim_{+}(f^*x, p)$ for $x\in X_n, p\in \Delta^m_{top}$, for \emph{any} morphism $f:[m]\to [n]$ in the simplex category $\Delta$.
\end{Def}

\begin{rmk}\label{comparison geom real rmk}
The two definitions of the geometric realization above are equivalent for discrete simplicial spaces.

Also, for a simplicial space $X_{\cdot}$, there is a natural quotient map $\Vert X_{\cdot} \Vert \to \vert X_{\cdot} \vert$.
\end{rmk}

% Recall from \cite{segal1974categories} the definition of \emph{fat (thick) geometric realization} of a simplicial topological space $X_{\cdot}$, which we will denote by $\Vert X_{\cdot} \Vert $ to avoid confusion with the usual geometric realization. 
The notion of fat geometric realization is better suited for simplicial topological spaces than the usual notion, as it takes the topological structure into account. In particular we have the proposition below.

\begin{prop}:\cite[A.1.]{segal1974categories}
Let $X_{\cdot}, Y_{\cdot}$ be simplicial topological spaces.
\begin{enumerate}\itemsep -1pt 
\item If each $X_n$ has the homotopy type of a CW-complex, then so does $\Vert X_{\cdot}\Vert$.
\item If $X_{\cdot} \to Y_{\cdot}$ is a simplicial map such that $X_n \to Y_n$ is a weak homotopy equivalence for each $n$, then $\Vert X_{\cdot}\Vert \to \Vert Y_{\cdot}\Vert$ is also a weak homotopy equivalence.
\item $\Vert X_{\cdot} \times Y_{\cdot}\Vert$ is weakly homotopy equivalent to $\Vert X_{\cdot}\Vert \times \Vert Y_{\cdot} \Vert$.
\end{enumerate}
\end{prop}
%%%%%%%%%%%%%%%%%%%%%%%%%%%%

Let us recall the general process of defining the algebraic K-theory spectrum of a small \emph{Waldhausen category} $(\mathcal{A},w)$ (for more details, see\cite{waldhausen1985algebraic}). 
\begin{Def}
Let $\mathcal{A}$ be a Waldhausen category. Define the simplicial category $ S_{\cdot} \mathcal{A}$ as follows. First, consider the category of ordered pairs of integers $(j,k)$ with $0\leq j \leq k \leq n$ that has a unique morphism from $(j,k)$ to $(j',k')$ iff $j\leq j'$ and $k\leq k'$. Then the objects in $S_n \mathcal{A}$ are the functors $A$ from this category of pairs to the category $\mathcal{A}$, so that $A(j,j)=0$ and $A(j,k)\rightarrowtail A(j,l) \twoheadrightarrow A(k,l)$ is a cofibration sequence in $\mathcal{A}$ whenever $0\leq j\leq k \leq l \leq n$. The morphisms in $S_{n}\mathcal{A}$ are the natural transformations $A\to A'$, and the weak equivalences are the morphisms that $A(j,k)\to A'(j,k)$ are all weak equivalences in $\mathcal{A}$. The cofibrations are the morphisms that $A(j,k)\to A'(j,k)$ are all cofibrations, and $A(j,l)\coprod_{A(j,k)} A'(j,k)\to A'(j,l)$ are also cofibrations in $\mathcal{A}$ whenever $0 \leq j\leq k \leq l\leq n$. Note that a morphism $f:[n]\to [m]$ in the opposite simplex category $\Delta^{op}$ induces a functor $S_{n}\mathcal{A}\to S_m \mathcal{A}$, which sends the object $(j,k)\mapsto A(j,k)$ of $S_n \mathcal{A}$ to the object $(r,s)\mapsto A(f(r), f(s))$ in $S_m\mathcal{A}$. This defines a simplicial structure on $S_{\cdot}\mathcal{A}$ \footnote{This description is taken from \cite[1.5.1.]{thomason1990higher}.}.

% sequences of cofibrations of length $n$, with choices of quotients, and morphisms are the collections of morphisms in $\mathcal{A}$ which make all the diagrams commute. The face and the degeneracy maps are obtained by repeating a cofibration (together with the corresponding choices of quotients) or deleting one. $S_n \mathcal{A}$ also comes with the structure of a Waldhausen category, where the weak equivalences are object-wise weak equivalences in $\mathcal{A}$, and the cofibrations are object-wise cofibrations where a certain collection of induced maps are also cofibrations in $\mathcal{A}$. 
% by considering whether a certain collection of maps are cofibrations in $\mathcal{A}$. 
Let $w S_{\cdot}\mathcal{A}$ be the simplicial category obtained by only considering the weak equivalences in $S_{\cdot}\mathcal{A}$, and form the nerves in each degree, which yeilds a bisimplicial set $N_{\cdot} w S_{\cdot} \mathcal{A}$. 
Define the \emph{algebraic K-theory spectrum} $K^{alg}(\mathcal{A})$ of the discrete Waldhausen category $\mathcal{A}$ as the spectrum whose $n$'th space is the goemetric realization $\vert N_{\cdot} w \underbrace{ S_{\cdot} S_{\cdot}\ldots S_{\cdot}}_\text{n \textit{times}}  \mathcal{A} \vert$. (Or we could have defined the algebraic K-theory \emph{space} to be the loop space $\Omega \vert N_{\cdot} w S_{\cdot} \mathcal{A} \vert$. In fact, they have the same (stable) homotopy groups, hence we may sometimes use the space instead of the spectrum.)

By \cite{mitchener2001symmetric}, we can define the \emph{toplogical K-theory spectrum} $K^{top}(\mathcal{A})$ of the topological Waldhausen category $\mathcal{A}$ similar as above, i.e. as a spectrum whose $n$'th space is the fat geometric realization $\Vert N_{\cdot} w \underbrace{ S_{\cdot} S_{\cdot}\ldots S_{\cdot}}_\text{n \textit{times}}  \mathcal{A} \Vert$.
% similar to the process above, except for the last part, where we need to take the fat geometric realizations of the simplicial topological categories instead.% of the usual ones. 
\end{Def}
%%%%%%%%%%%%% find reference %%%%%%%%%%%%%%%%%%%%%%%%%
% \begin{notation}\label{comparison map notation}
% Let $\mathcal{A}$ be a topological Waldhausen category. To distinguish between the two versions of K-theory, denote the algebraic (obtained through forgetting the topological structure and using the usual geometric realization) and topological K-theory (obtained by keeping the topological structure and using the fat geometric realization) spectrum of $\mathcal{A}$ by $K^{alg}(\mathcal{A}), K^{top}(\mathcal{A})$, respectively. Unless mentioned otherwise, when $\mathcal{A}$ is a discrete category, then $K(\mathcal{A})$ refers to $K^{alg}(\mathcal{A})$, and when $\mathcal{A}$ is a topological category (in particular a $C^*$-category) then $K(\mathcal{A})$ refers to $K^{top}(\mathcal{A})$.
% There is a natural \emph{comparison} map $c: K^{alg}(\mathcal{A})\to K^{top}(\mathcal{A})$.
% \end{notation}

\begin{Def}\label{comparison map def}
Let $\mathcal{A}$ be a topological Waldhausen category. Let $\mathcal{A}^{\delta}$ denote the discrete Waldhausen category obtained by forgetting the topological structure. There is a natural exact functor $\mathcal{A}^{\delta}\to \mathcal{A}$ %The category $\mathcal{A}^{\delta}$ can be considered as a topological category with the trivial topology.
which induces a natural map $K^{top}(\mathcal{A}^{\delta})\to K^{top}(\mathcal{A})$. By remark \ref{comparison geom real rmk}, there is a natural equivalence of K-theory spectra $K^{alg}(\mathcal{A}^{\delta}) \cong K^{top}(\mathcal{A}^{\delta})$. 
% Let $X_{\cdot}$ be a simplicial space. Recall from definition of fat goemetric realization that there is a natural map $\Vert X_{\cdot} \Vert \to \vert X_{\cdot} \vert $ (since in the fat geometric realization, we are quotienting out fewer relations.). % Also recall from \cite[A.4.]{segal1974categories} that the simplicial space $X_{\cdot}$ is called \emph{good} if all the inclusions $X_{n,i}\to X_n$ is a closed cofibration (i.e. is a cofibration with closed image), where $X_{n,i}$ is the copy of $X_{n-1}$ in $X_n$ that misses $i$. If $X_{\cdot}$ is a discrete simplicial space (also known as simplicial set) then it satisfies the condition above 
% If $X_{\cdot}$ is a good simplicial space, then the map $\Vert X_{\cdot} \Vert \to \vert X_{\cdot} \vert $ is a homotopy equivalence \cite[A.1.4.]{segal1974categories}.  
% But for a discrete simplicial space (otherwise known as simplicial set) these two definitions are equivalent. In particular, there is a homotopy equivalence $K^{top}(\mathcal{A}^{\delta}) \cong K^{alg}(\mathcal{A}^{\delta}) $. 

% For a simplicial topological space $X_{\cdot}$, let $X^{\delta}_{\cdot}$ denote the simplicial set

Hence there is a natural \emph{comparison map} $c: K^{alg}(\mathcal{A}^{\delta})\to K^{top}(\mathcal{A})$.
\end{Def}

\begin{notation}
If $\mathcal{A}$ is a topological Waldhausen category, we will simply write $K^{alg}(\mathcal{A})$ instead of $K^{alg}(\mathcal{A}^{\delta})$.  
\end{notation}

Recall from \cite[1.3.]{waldhausen1985algebraic}:

\begin{Def}
Let $\mathcal{A}, \mathcal{B}$ be Waldhausen categories. Then we say that a sequence $F_0 \rightarrow F_1 \rightarrow F_2$ of exact functors from $\mathcal{A}$ to $\mathcal{B}$ and natural transformations between them is a \emph{short exact sequence of functors}, if for each object $A$ of $\mathcal{A}$ we have a cofibration sequence $F_0(A) \rightarrowtail F_1(A) \twoheadrightarrow F_2(A)$ in $\mathcal{B}$.
\end{Def}

\begin{thm}[Additivity Theorem] \cite[Sec 3.]{quillen1973higher} \cite[1.3.2.]{waldhausen1985algebraic}
Let $\mathcal{A}, \mathcal{B}$ be Waldhausen categories, and let $ F_0 \rightarrowtail F_1 \twoheadrightarrow F_2$ be a short exact sequence of functors from $\mathcal{A}$ to $\mathcal{B}$. Then $F_{1 *}, F_{0 *}+ F_{2 *}: K(\mathcal{A}) \rightarrow K(\mathcal{B})$ are homotopic to each other. By \cite[4.2.]{mitchener2001symmetric}, the same holds for the topological Waldhausen categories. 
\end{thm}

\begin{Def}[Relative K-theory Space] \cite[1.5.]{waldhausen1985algebraic} \label{Relative k-theory space}
Let $\mathcal{A}, \mathcal{B}$ be Waldhausen categories, and let $F: \mathcal{A} \rightarrow \mathcal{B}$ be an exact functor. Then define the category $[\mathcal{A} \xrightarrow{F} \mathcal{B}]_{\cdot}$ by $[\mathcal{A} \xrightarrow{F} \mathcal{B}]_n = S_{n}\mathcal{A} \times_{S_n\mathcal{B}} S_{n+1}\mathcal{B}$. There is a natural simplicial structure on $[\mathcal{A} \xrightarrow{F} \mathcal{B}]_{\cdot}$, and the Waldhausen category structures of $S_{\cdot} \mathcal{A}, S_{\cdot}\mathcal{B}$ induce one on $[\mathcal{A} \xrightarrow{F} \mathcal{B}]$ in a natural way.
% One can define structure of a Waldhausen category on $[\mathcal{A}\xrightarrow{F} \mathcal{B}]$ as well, and 
\end{Def}

\begin{prop}\cite[1.5.5.]{waldhausen1985algebraic}
There are natural functors of Waldhausen categories $\mathcal{B} \rightarrow [\mathcal{A} \xrightarrow{F} \mathcal{B}]_{\cdot} \rightarrow S_{\cdot}\mathcal{A}$ which in turn induce the homotopy fibration sequence
\begin{equation*}
w S_{\cdot}\mathcal{B} \rightarrow w S_{\cdot}[\mathcal{A} \xrightarrow{F} \mathcal{B}]_{\cdot} \rightarrow w S_{\cdot} S_{\cdot} \mathcal{A}.
\end{equation*}  
By \cite[4.4.]{mitchener2001symmetric}, the same holds for topological Waldhausen categories. 
\end{prop}

\begin{Def} \cite[\text{After }1.5.3.]{waldhausen1985algebraic} 
Let $\mathcal{A}, \mathcal{B}, \mathcal{C}$ be Waldhausen categories, then a functor $F:\mathcal{A}\times \mathcal{B}\to \mathcal{C}$ is \emph{biexact} if for each object $A$ of $\mathcal{A}$ and $B$ of $\mathcal{B}$, the functors $F(A,-)$ and $F(-,B)$ are exact, and also for each cofibration $A \rightarrowtail A'$ in $\mathcal{A}$ and $B \rightarrowtail B'$ in $\mathcal{B}$, the map below is a cofibration in $\mathcal{C}$.
$$F(A,B')\cup_{F(A,B)} F(A',B)\to F(A',B').$$
\end{Def}

\begin{prop}\cite[\text{After }1.5.3.]{waldhausen1985algebraic} \label{biexact product waldhausen k theory prop}
A biexact functor $F:\mathcal{A}\times \mathcal{B}\to \mathcal{C}$ of Waldhausen categories, induces a map of bisimplicial categories $wS_{\cdot}\mathcal{A}\wedge wS_{\cdot} \mathcal{B}\to wwS_{\cdot}S_{\cdot}\mathcal{C}$ which in turn induces a map of K-theory spectra 
\begin{equation*}
K(\mathcal{A})\wedge K(\mathcal{B})\to K(\mathcal{C}).
\end{equation*} 
The same holds for the topological categories \cite[2.8.]{mitchener2001symmetric}.
\end{prop}

%%%%%%%%%%%%%%%%%%%%%%%%%%%%%

% \begin{Def}
% Let $\mathcal{A}$ be a Waldhausen category. Then a Waldhausen subcategory $\mathcal{A}$ is called \emph{cofinal} if for every object $B$ of the category $\mathcal{B}$, there is an object $B'$ in $\mathcal{B}$ so that $B \amalg B'$ is isomorphic to an object in $\mathcal{A}$. If we can always take $B'$ to be an object in $\mathcal{A}$, then $\mathcal{A}$ is called \emph{strictly cofinal}.
% \end{Def}

\begin{Def}
We say that a (topological) Waldhausen subcategory $\mathcal{A}$ of $\mathcal{B}$ is \emph{closed under extensions} if for each cofibration sequence in $\mathcal{B}$ where the source, and the quotient are in $\mathcal{A}$, then the target is isomorphic to an object in $\mathcal{A}$.
\end{Def}

\begin{prop}\cite[1.5.9.]{waldhausen1985algebraic} \label{cofinality waldhausen k-theory prop}
If $\mathcal{A}$ is a strictly cofinal (topological) Waldhausen subcategory of $\mathcal{B}$, then the natural map $K(\mathcal{A})\to K(\mathcal{B})$ is a homotopy equivalence.
% Waldhausen K-theory spectra of $\mathcal{A}$ and $\mathcal{B}$ are homotopy equivalent. 
If $\mathcal{A}$ is a only a cofinal (topological) Waldhausen subcategory of $\mathcal{B}$ which is also closed under extensions, then the natural map of K-theory spectra $K(\mathcal{A})\to K(\mathcal{B})$ induces an isomorphism on the $i$'th homotopy group when $i \geq 1$.
\end{prop}

The above statement was originally proved for discrete categories, however, in here we will need to apply it to a certain cofinal subcategory of the Paschke category. The proof goes through for topological categories with no change, but for the sake of completeness, we repeat the argument here.

%we outline the proof of the three assertions Waldhausen proves.\\  % Waldhausen's proof relies on three assertions, the first two of which, can be proved for topological categories with no change. For the sake of completeness, we outline the proof of the first two assertions here.\\ 
\begin{notation} \label{space of objects wald cat-notation}
Let $\mathcal{A}, \mathcal{B}$ denote topological Waldhausen categories and let $F:\mathcal{A}\to \mathcal{B}$ be an exact functor. Then we denote the space of objects in $S_{\cdot}\mathcal{A}$ by $s_{\cdot}\mathcal{A}$, and denote the space of objects in $[\mathcal{A}\xrightarrow{F} \mathcal{B}]$ by $[s(\mathcal{A}\xrightarrow{F} \mathcal{B})]_{\cdot}$. Beware that the second notation is not standard.
\end{notation}
% The next lemma, is essentially the result of property number $2$ of fat geometric realization.
% Note that properties $2$ and $3$ of fat geometric realization mean that a natural transformation $\eta: f \Rightarrow g$ of exact functors $f,g: \mathcal{C} \rightarrow \mathcal{C}'$ between topological Waldhausen categories $\mathcal{C}, \mathcal{C}'$, induces a homotopy between the corresponding maps of spectra from $K(\mathcal{C})$ to $K(\mathcal{C}')$(using \cite[Page 6]{mitchener2001symmetric}). Hence, if all the weak equivalences in $\mathcal{C}$ are isomorphisms, then there is a homotopy equivalence $s_{\cdot}\mathcal{C} \rightarrow wS_{\cdot} \mathcal{C}$ induced by the face and degeneracy maps in the $w$-direction.(Compare \cite[1.4.1.]{waldhausen1985algebraic} and the corollary following it.). \\
\begin{lem}\cite[1.4.1.]{waldhausen1985algebraic}
Let $F:\mathcal{A}\to \mathcal{B}$ be an exact functor of topological Waldhausen categories. Then there is an induced map $s_{\cdot}F: s_{\cdot}\mathcal{A}\to s_{\cdot}\mathcal{B}$. An isomorphism between two such functors $F_0,F_1$ induces a homotopy between $s_{\cdot}F_0, s_{\cdot}F_1$.
\end{lem}

\begin{proof}
The first statement is clear (cf. \cite[Page 6.]{mitchener2001symmetric}). To prove the second part, we will explicitly write down a simplicial homotopy. Simplicial objects in a category $\mathcal{C}$ can be considered as functors $X:\Delta^{op}\to \mathcal{C}$, and maps of simplicial objects are natural transformations of such functors. Simplicial homotopies can be described similarly; namely let $\Delta/[1]$ be denote the category of objects over $[1]$ in the simplex category, i.e. objects are maps $[n]\to [1]$. For any $X:\Delta^{op}\to \mathcal{C}$, let $X^*$ denote the composited functor
$$ (\Delta/ [1])^{op} \to \Delta^{op} \xrightarrow{X} \mathcal{C} $$
$$ ([n]\to [1])\mapsto [n] \mapsto X[n] $$
Then a simplicial homotopy of maps may be identified with a natural transformation $X^*\to Y^*$.

Now, suppose there is a functor isomorphism from $F_0$
to $F_1$ given by $F: \mathcal{A}\times [1]\to \mathcal{B}$. The required simplicial homotopy then is a map from $([n]\to [1]) \mapsto s_{n}\mathcal{A}$ to $([n]\to [1]) \mapsto s_{n}\mathcal{B}$ given by
$$ (a:[n]\to [1]) \mapsto \left( (A:Ar[n]\to \mathcal{A}) \mapsto (B:Ar[n]\to \mathcal{B}) \right) $$
where $B$ is defined as the composition 
$$ Ar[n] \xrightarrow{(A,a_*)} \mathcal{A}\times Ar[1] \xrightarrow{Id \times p} \mathcal{A}\times [1] \xrightarrow{F} \mathcal{B}  $$
and $p:Ar[1]\to [1]$ is given by $(0,0)\mapsto 0 ,\, (0,1) \mapsto 1 ,\, (1,1) \mapsto 1$.
\end{proof}

\begin{cor}\cite{waldhausen1985algebraic}
An equivalence of Waldhausen topological categories $\mathcal{A}\to \mathcal{B}$ induces a homotopy equivalence $s_{\cdot} \mathcal{A}\to s_{\cdot}\mathcal{B}$. Therefore if weak equivalences of $\mathcal{A}$ are the isomorphisms, denoted by $i$, then $s_{\cdot}\mathcal{A}\to iS_{\cdot}\mathcal{A}$ is a homotopy equivalence.
\end{cor}

The first part of this corollary is clear consequence of the lemma. The second part is a result of considering the simplicial object $[m]\to i_m S_{\cdot}\mathcal{A}$, the nerve of $iS_{\cdot}\mathcal{A}$ in the $i$-direction, and noting that $i_0 S_{\cdot} \mathcal{A}= s_{\cdot}\mathcal{A}$ and that face and degenaracy maps are homotopy equivalences by the first part of the corollary.  

\begin{proof}[Proof of Propositoin \ref{cofinality waldhausen k-theory prop}]
To prove that a strictly cofinal topological Waldhausen subcategory $\mathcal{A}$ of $\mathcal{B}$ and $\mathcal{B}$ have homotopy equivalent K-theory spaces (and similarly spectra), it suffices to show that the relative K-theory category $wS_{\cdot}[\mathcal{A} \hookrightarrow \mathcal{B}]_{\cdot}$ is contractible. By property $2$ of fat geometric realization, it suffices to show that each $wS_n [\mathcal{A} \hookrightarrow \mathcal{B}]_{\cdot}$ is contractible. Consider the inclusions $S_n\mathcal{A} \hookrightarrow S_n{\mathcal{B}}$. Then $wS_n [\mathcal{A} \hookrightarrow \mathcal{B}]_{\cdot}$ is equivalent to $w [S_n\mathcal{A} \hookrightarrow S_n\mathcal{B}]$. But it is easy to show that $S_n\mathcal{A}$ is a strictly cofinal subcategory of $S_n\mathcal{B}$: take an object $\{B_{jk}\}_{0\leq j<k\leq n}$ in $S_n\mathcal{B}$. Then for each $B_{jk}$ in $\mathcal{B}$, there exists an object $A_{jk}$ in $\mathcal{A}$ so that $B_{jk} \amalg A_{jk}$ is isomorphic to an object in $\mathcal{A}$. Hence if $A= \amalg_{j,k}A_{jk}$, then for all $j,k$, $B_{jk} \amalg(\amalg_{l=1}^{k-j} A)$ is isomorphic to an object in $\mathcal{A}$, as $\mathcal{A}$ is closed under finite coproduct. Therefore $\{B_{jk} \amalg(\amalg_{l=1}^{k-j} A)\}_{0\leq j < k \leq n}$ is isomorphic to an object in $S_n\mathcal{A}$. 

To show $w [\mathcal{A} \hookrightarrow \mathcal{B}]_{\cdot}$ is contractible, again, by property $2$ of fat geometric realization, it suffices to show that $w_m [\mathcal{A} \hookrightarrow \mathcal{B}]_{\cdot}$ (the $m$-degree part in the $w$-direction) is contractible for all $m$. Let $\mathcal{A}(m,w)$ denote a sequence $A_0 \xrightarrow{\simeq} A_1 \xrightarrow{\simeq} \ldots \xrightarrow{\simeq} A_m$ of $m$-weak equivalences in $\mathcal{A}$, and similarly define $\mathcal{B}(m,w)$, and consider the inclusion $\mathcal{A}(m,w) \hookrightarrow \mathcal{B}(m,w)$. %%%%%%%%%%%%%%% $f_{(m,w)}$ denote the inclusion. 
Similar to before, $\mathcal{A}(m,w)$ is strictly cofinal in $\mathcal{B}(m,w)$. It is easy to see that $w_m [\mathcal{A} \hookrightarrow \mathcal{B}]_{\cdot} \simeq [s(\mathcal{A}(m,w) \hookrightarrow \mathcal{B}(m,w))]_{\cdot}$ %%%%%%%%%%%%%%%%%%s_{\cdot}_{(m,w)}$
% where the latter is the simplicial space of objects in $[\mathcal{A}(m,w) \hookrightarrow \mathcal{B}(m,w)]_{\cdot}$. %%%%%%%%%%%%%%%%%%%%%%%$S_{\cdot}f_{(m,w)}$. 
Hence it suffices to prove that when $\mathcal{A}\hookrightarrow \mathcal{B}$ is an inclusion of a strictly cofinal topological Waldhausen subcategory, %and let $[s(\mathcal{A} \hookrightarrow  \mathcal{B})]_{\cdot}$ denote the simplicial space of objects of $[\mathcal{A}\hookrightarrow \mathcal{B}]_{\cdot}$, 
then $[s(\mathcal{A} \hookrightarrow  \mathcal{B})]_{\cdot}$ is contractible (cf. \ref{space of objects wald cat-notation}).

% We show the answer is positive if we were to start with the following stronger condition instead of cofinality:

% \begin{prop}\label{Waldhausen cofinality proposition with stronger assumption} Let $\mathcal{A}$ be a topological Waldhausen subcategory of $\mathcal{B}$, and assume that there exists an object $A'$ in $\mathcal{A}$ so that $A' \amalg B$ is isomorphic to an object of $\mathcal{A}$, for any object $B$ of $\mathcal{B}$. Then the topological Waldhausen K-theory spectrum of $\mathcal{A}$ and $\mathcal{B}$ are homotopy equivalent. \end{prop}

%By the argument above, it suffices to show that $s_{\cdot}f$ is contractible, where $f$ is the inclusion map. 
First note that the simplicial space $[s(\mathcal{A}\xrightarrow{Id} \mathcal{A})]_{\cdot}$ is nerve of the (topological) category of cofibrations in $\mathcal{A}$ which has an initial object, and hence is contractible \footnote{The point and the $1$-simplex $[1]$ are both \emph{good} simplicial spaces \cite[A.4.]{segal1974categories}, hence their fat geometric realizations are homotopy equivalent to their (usual) geometric realizations \cite[A.5.]{segal1974categories} which are both contractible. Hence, by property 3 of the fat geometric realization, the argument in \cite[2.1.]{segal1968classifying} goes through to show that a natural transformation between two functors between topological categories induces a homotopy between the induced maps on the fat geometric realizations. If the topological category $\mathcal{C}$ has an initial object, then there is an induced homotopy between the fat geometric realization of the nerve of the category $\mathcal{C}$, and the fat geometric realization of a point, which is contractible.}. Now we want to show that the inclusion $[s(\mathcal{A} \xrightarrow{Id}  \mathcal{A})]_{\cdot} \rightarrow [s(\mathcal{A} \hookrightarrow  \mathcal{B})]_{\cdot} $ is a homotopy equivalence. Consider the category of cofibrations in $\mathcal{B}$. Then $[s(\mathcal{A} \hookrightarrow  \mathcal{B})]_{\cdot}$ is homotopic to a simplicial subset of the nerve of this category (through forgetting the choices of quotients $B_{jk}\simeq B_{0k}/B_{0j}$), and taking a pushout with a fixed object is a natural transformation of the identity functor on this category to the pushout functor. In other words, there is a homotopy from the identity functor on $[s(\mathcal{A} \hookrightarrow  \mathcal{B})]_{\cdot}$ to the functor $\kappa_A$, where $\kappa_A (A_{jk}, B_{jk})=(A_{jk}, B'_{jk})$, where $B'_{jk}=B_{jk} \amalg A$ when $j=0$ and $B'_{jk}= B_{jk}$ otherwise. But for any simplicial set $L$ in $[s(\mathcal{A} \hookrightarrow  \mathcal{B})]_{\cdot}$ with \emph{finitely} many non-degenerate simplicies, as we argued before there exists an object $A$ in $\mathcal{A}$ so that $\kappa_A$ applied to $L$, would send each of the non-degenarate simplicies to simplicies (weakly equivalent to simplicies) in $[s(\mathcal{A}\xrightarrow{Id} \mathcal{A})]_{\cdot}$. But then $\kappa_A$ sends all of $L$ to simplicies (weakly equivalent to simplicies) in $[s(\mathcal{A}\xrightarrow{Id} \mathcal{A})]_{\cdot}$. Therefore there is a homotopy from the inclusion of $L$ in $[s(\mathcal{A} \hookrightarrow  \mathcal{B})]_{\cdot}$ to a map from $L$ to $[s(\mathcal{A} \xrightarrow{Id}  \mathcal{A})]_{\cdot}$.
% the latter functor is homotopic to a functor from $s_{\cdot}f$ to $s_{\cdot}Id_{\mathcal{A}}$, as $B'_{ij}$ is isomorphic to an object in $\mathcal{A}$ for all $i,j$. 

The proof for when $\mathcal{A}$ in $\mathcal{B}$ is only cofinal, goes through similarly. To show that the connected component of the zero object in $[s(\mathcal{A} \hookrightarrow  \mathcal{B})]_{\cdot}$ is contractible, one needs to use the assumption that $\mathcal{A}$ is closed under extension, which in turn shows that the object $A$ (that was obtained by using the cofinality assumption applied to the objects $B_{jk}$) used in the paragraph above, is in fact isomorphic to an object of $\mathcal{A}$.

%is used to show that the inclusion of the connected component of the zero object in $\Vert \Omega w S_{\cdot} \mathcal{A}\Vert \to \Vert wS_{\cdot}\mathcal{B} \Vert$ is a homotopy equivalence, because 
\end{proof}

% \begin{cor}\label{cofinality for topological categories-cor}
% As a result of the proof above, Waldhausen cofinality theorem (proposition \ref{cofinality waldhausen k-theory prop}) also holds for topological categories.
% \end{cor}

\begin{rmk}
We will only use cofinality in the case when there exists an object $A_0$ of $\mathcal{A}$ so that for each object $B$ of $\mathcal{B}$, $A_0\oplus B$ is isomorphic to an object in $\mathcal{A}$. The proof of lemma above for this special case is slightly simpler and easier to understand. 
\end{rmk}

%%%%%%%%%%%%%%%%%%%%%%%%%%%%%%%%%%%%%%%%%%%%%%%%%%%%%
%%%%%%%%%%%%%%%%%%%%%%%%%%%%%%%%%%%%%%%%%%%%%%%%%%%%%

\subsection{Grayson's Map}

We start this subsection by going through an unfortunately rather long list of notations and definitions, and then we will use a construction of Grayson to give a natural map in the homotopy category of spectra \footnote{The \emph{stable homotopy category} (cf. \cite{adams1995stable}) can be considered as the localization of the category of spectra at the weak homotopy equivalences. In particular, all homotopic maps are equivalent to each other in the homotopy category, and homotopy equivalences are invertible.}, from the K-theory spectrum of the category of acyclic \emph{binary} chain complexes in an exact category $\mathcal{A}$, to the loop space of the K-theory spectrum of $\mathcal{A}$. We will closely follow the construction to see the extend of which it can be applied to topological exact categories.

\begin{notation}\label{postnikov tower notaiton}
Let $X_{\cdot}$ be a spectrum. For $n\in \mathbb{Z}^{\geq 0}$, let $V^n X_{\cdot}$ denote the $n$'th stage of the postnikov filtration of $X_{\cdot}$, obtained by killing the stable homotopy groups $\pi_m (X_{\cdot})$ for $m< n$. In particular, $V^0 X_{\cdot}$ is the connective part of $X_{\cdot}$.
\end{notation}

Following \cite[Sec 1.6.]{waldhausen1985algebraic}, let $(\mathcal{A},w)$ be a (topological) Waldhausen category with the subcategory $w\mathcal{A}$ (sometimes abbreviated to just $w$) of weak equivalences. If $(\mathcal{A},v)$ is also a (topological) Waldhausen category with weak equivalences so that $w$ is a subcategory of $v$, then let $\mathcal{A}^v$ denote the full (topological) Waldhausen subcategory of $(\mathcal{A},w)$ whose objects $A$ are the ones with the property that $*\to A$ is in $v$. Recall that for a (topological) category $\mathcal{A}$ with cofibrations, if $i\mathcal{A}$ denotes the subcategory of isomorphisms, then $(\mathcal{A},i)$ is a (topological) Waldhausen category. 

For a (topological) exact category $\mathcal{A}$, let $C\mathcal{A}$ denote the category of \emph{chain complexes} in $\mathcal{A}$ and let $Ch\mathcal{A}$ be the category of \emph{acyclic chain complexes} in $\mathcal{A}$, both of which have chain maps as morphisms. The categories $C\mathcal{A}, Ch\mathcal{A}$ have a natural (topological) exact structure; a sequence is called exact iff it is exact degreewise. This means the cofibrations are the morphisms which are degree-wise cofibrations (admissible monomorphisms) and the weak equivalences $i$ are the degree-wise \emph{isomorphisms} \footnote{We will abuse notation and denote the class of isomorphisms of different categories by $i$.}. We introduce a different structure of a (topological) Waldhausen category on $C\mathcal{A}$ by defining the cofibrations to be the degree-wise cofibrations again, and define the weak equivalences to be the \emph{quasi-isomorphisms}, which we denote by $q$. Note that quasi-isomorphisms are considered with respect to embedding the exact category $\mathcal{A}$ into an abelian category. This definition does not depend on the choice of the embedding if $\mathcal{A}$ either \emph{supports long exact sequences} \cite[1.4.]{grayson2012algebraic}, or if $\mathcal{A}$ satisfies the condition in \cite[1.11.3.]{thomason1990higher}. These conditions are both satisfied if $\mathcal{A}$ is a (topological) pseudo-abelian category, cf \cite[1.11.5.]{thomason1990higher} and \cite[4.]{grayson2012algebraic}. 
% We introduce an exact structure on $C\mathcal{A}$ by saying a sequence of chain complexes in $C\mathcal{A}$ is exact iff it is exact degree-wise. Instead we could have defined the weak equivalences in $C\mathcal{A}$ to be the \emph{quasi-isomorphisms} \footnote{Quasi-isomorphisms are considered with respect to embedding the exact category $\mathcal{A}$ into an abelian category. This definition does not depend on the choice of the embedding if $\mathcal{A}$ either \emph{supports long exact sequences} \cite[1.4.]{grayson2012algebraic}, or if $\mathcal{A}$ satisfies the condition in \cite[1.11.3.]{thomason1990higher}. These conditions are both satisfied if $\mathcal{A}$ is a pseudo-abelian category, cf \cite[1.11.5.]{thomason1990higher} and \cite[4.]{grayson2012algebraic}.}, which we denote by $q$, and are exactly the morphisms whose mapping cone is acyclic. The cofibrations are morphisms which are cofibrations degree-wise. 
Evidently, $(C\mathcal{A},q)$ is a (topological) Waldhausen category, and $Ch(\mathcal{A})$ is equal to $(C\mathcal{A})^q$. Furthermore, denote the full subcategories of \emph{bounded chain complexes} and \emph{bounded acyclic chain complexes} by $C^b\mathcal{A}$ and $Ch^b\mathcal{A}$ respectively. Again we have $(C^b\mathcal{A})^q = Ch^b\mathcal{A}$. % We can define structure of a Waldhausen category on $Ch\mathcal{A}$ by considering the subcategory $i$ of weak equivalences to be the morphisms which are object-wise isomorphisms.
%morphisms whose mapping cone is \emph{split-exact}, and cofibrations to be degree-wise cofibrations again. %%%%%%%%%%%%%%%%%%%%%%%%%%%%%%%%%%%%%%%%%%%%%%%%%%%%%%%%%%%%%%%%%%%%%%%%%% ARE THE WEAK EQUIVALENCES CORRECT? %%%%%%%%%%%%%%%%%%%%%%%%%%%%%%%%%%%%%%%%%%%%%%%%%%%%%%%%%%%%%

\begin{Def} \cite[3.1.]{grayson2012algebraic}
Let $\mathcal{A}$ be a (topological) exact category. We define a \emph{binary chain complex} in $\mathcal{A}$ to be a chain complex in $\mathcal{A}$ with \emph{two} differentials instead of one, i.e. a pair of chain complexes with the same objects but possibly different differentials, called the \emph{top differential} of the top chain complex, and the \emph{bottom differential} of the bottom chain complex. A binary chain complex is \emph{acyclic} if both the top and the bottom chain complexes are acyclic. If we denote a binary chain complex by $(A^{\cdot}, d_1^{\cdot},d_2^{\cdot})$, then $A^{\cdot}$ are the objects of the complex, $d_1^{\cdot}$ are the top differentials, and $d_2^{\cdot}$ are the bottom differentials. 
Let $B\mathcal{A}$ and $Bi \mathcal{A}$ be the (topological) category of binary chain complexes in $\mathcal{A}$ and acyclic binary chain complexes in $\mathcal{A}$.
Also denote the (topological) category of \emph{bounded binary chain complexes} and the category of \emph{bounded acyclic binary chain complexes} in $\mathcal{A}$ by $B^b \mathcal{A}$ and $Bi^b\mathcal{A}$, respectively.
A morphism between two (respectively, acyclic) binary chain complexes is a map between the underlying objects which is a chain map with respect to both chain complexes; in other words, a chain map when we consider only the top chain complexes, and also a chain map when we consider only the bottom chain complexes.

Similar to before, the categories $B\mathcal{A}, Bi \mathcal{A}, B^b\mathcal{A}, Bi^b\mathcal{A}$ have a natural (topological) exact structure given by exactness at each degree. This means the cofibrations are degree-wise cofibrations, and weak equivalences $i$ are the degree-wise isomorphisms. We can define another structure of a (topological) Waldhausen category on $B\mathcal{A}, B^b\mathcal{A}$ with the cofibrations being the degree-wise cofibrations and the weak equivalences being the \emph{quasi-isomorphisms} which we again denote by $q$. This again may depend on the choice of embedding $\mathcal{A}$ in an abelian category, but does not depend on that choice if $\mathcal{A}$ is a (topological) pseudo-abelian category. 
% the cofibrations in $B \mathcal{A}$ and $Bi \mathcal{A}$ (and also $B^b\mathcal{A}, Bi^b\mathcal{A}$) to be the degree-wise cofibrations with respect to both chain complexes, and we can define weak equivalences in $B\mathcal{A}$ to be \emph{quasi-isomorphisms}, denoted by $q$ (by abuse of notation of course), or equivalently those morphisms whose both the top and the bottom mapping cones are acyclic chain complexes. 
Hence we have a (topological) Waldhausen category $(B\mathcal{A},q)$, and again $(B\mathcal{A})^q, (B^b\mathcal{A})^q$ are the categories $Bi\mathcal{A}, Bi^b\mathcal{A}$, respectively.
% We can define structure of a (topological) Waldhausen category on $Bi\mathcal{A}$ (and also $Bi^b\mathcal{A}$) by considering the subcategory $i$ of weak equivalences (again, by abuse of notation) to be the morphisms whose both the top and the bottom mapping cones are \emph{split-exact} chain complexes.
\end{Def}

Let us denote the morphism that sends a chain complex $(A^{\cdot}, d^{\cdot})$ to the binary chain complex $(A^{\cdot}, d^{\cdot}, d^{\cdot})$ by $\Delta:C\mathcal{A}\to B\mathcal{A}$,
and  denote the morphisms that send a binary chain complex to respectively the top and the bottom chain complex by $\top, \bot: B\mathcal{A} \to C\mathcal{A}$. These are exact functors, and we use the same notation for their restriction to $C^b\mathcal{A} \to B^b\mathcal{A}$, $Ch\mathcal{A} \to Bi\mathcal{A}$, $B^b\mathcal{A}\to C^b\mathcal{A}$ and $Bi^b\mathcal{A} \to Ch^b\mathcal{A}$. Let $\tau, \tau^b$ denote the category of maps $f$ in $B\mathcal{A}$ and $B^b\mathcal{A}$ respectively, such that $\top f$ is in $qC \mathcal{A}, qC^b \mathcal{A}$, respectively, and let $\beta, \beta^b$ denote the category of maps $f$ in $B\mathcal{A}$ and $B^b\mathcal{A}$, such that $\bot f$ is in $qC \mathcal{A}, qC^b \mathcal{A}$, respectively. 
Define $F: (C\mathcal{A},q) \to (B\mathcal{A},\tau)$ by $F(A^{\cdot}, d^{\cdot}) = (A^{\cdot}, d^{\cdot}, 0)$. Then the composition $\top \circ F$ is the identity functor on $C\mathcal{A}$, and $F\circ \top$ is an exact endofunctor of $(B\mathcal{A},\tau)$.

Recall from definition \ref{Relative k-theory space} that for an exact functor $F:\mathcal{A} \to \mathcal{B}$ between (topological) Waldhausen categories, we have the relative K-theory space denoted by $[\mathcal{A} \xrightarrow{F} \mathcal{B}]$. We have the following proposition from Grayson \cite[Sec 7.]{grayson2012algebraic}:

\begin{prop}
Let $\mathcal{A}$ be a discrete exact category. Then there is a natural homotopy equivalence of spectra 
$$ K[ (Ch^b\mathcal{A},i) \xrightarrow{\Delta}  (Bi^b\mathcal{A},i)] \simeq V^0 \Omega K(\mathcal{A})$$
In particular, there is a natural isomorphism of K-theory groups $ K_{n-1}[ (Ch^b\mathcal{A},i) \xrightarrow{\Delta}  (Bi^b\mathcal{A},i) ]\cong K_n(\mathcal{A})$ when $n\geq 1$, and there is a natural map in the homotopy category of spectra 
\begin{equation} \label{Grayson's map equation}
\tau^G_{\mathcal{A}}: K(Bi^b \mathcal{A},i) \to \Omega K(\mathcal{A}). 
\end{equation}
\end{prop}

The proposition above uses ingredients such as Waldhausen's fibration and approximation theorems \cite[1.6.4, 1.6.7.]{waldhausen1985algebraic}, the Gillet-Waldhausen theorem \cite[6.2.]{gillet1981riemann}, and Thomason's cofinality theorem \cite[1.10.1.]{thomason1990higher}, which we will check for topological categories in a future paper.
%have checked for topological categories.
 
%% To show that there is still a natural map from the topological K-theory spectrum of $Bi^b(\mathcal{A})$ to the loop space of topological K-theory spectrum of $\mathcal{A}$, we go through the proof in Grayson's paper \cite{grayson2012algebraic}. 
% We go through Grayson's proof to show that there is still a natural map from the topological K-theory spectrum of $Bi^b(\mathcal{A})$ to the loop space of topological K-theory spectrum of $\mathcal{A}$.

Let $\mathcal{A}$ be (a topological) an exact category. Recall that by definition of the relative K-theory space, there is an exact functor $Bi^b(\mathcal{A}) \to [Ch^b \mathcal{A} \xrightarrow{\Delta} Bi^b \mathcal{A}]$ for a (topological) Waldhausen category $\mathcal{A}$. 
%%%%%%%%%%%%%%%%%%%%%%%%%%% FIGURE THIS OUT %%%%%%%%%%%%%%%%%%%%%%%%%%%%%%%%%%%%%%%%
Now, assuming that the (topological) category $\mathcal{A}$ %satisfies an extra technical condition, so called 
"supports long exact sequences", we give a series of maps in the homotopy category of spectra as follows. %\footnote{The Paschke category $(\mathfrak{D}/ \mathfrak{C})_A$ does not seem to satisfy this condition. However, since the Calkin-Paschke category $(\mathfrak{D}/ \mathfrak{C})'_A$ is pseudo-abelian by proposition \ref{Calkin-Paschke category is pseudo abelian-proposition}, hence it satisfies this extra condition by the argument following \cite[1.5.]{grayson2012algebraic}. }.
%%%%%%%%%%%%%%%%%%%%%%%%%%%%%%%%%%%%%%%%%%%%%%%%%%%%%%%%%%%%%%%%%%%%%%%%%%%%%%%%%%%
Following the proof of \cite[4.3.]{grayson2012algebraic}, we first give a map in the homotopy category of spectra 
\begin{equation*}
G_1: K[ (Ch^b \mathcal{A},i) \xrightarrow{\Delta} ( Bi^b \mathcal{A},i) ] \to \Omega K[(C\mathcal{A},i) \xrightarrow{\Delta} (B\mathcal{A},i)].
\end{equation*} 
We have the following commutative diagrams:
\begin{center}
\begin{tikzcd}
K(Ch^b \mathcal{A},i) \ar[r] \ar[d] 
& K(C^b \mathcal{A}, i) \ar[d] 
& K(Bi ^b \mathcal{A}, i) \ar[r] \ar[d]
& K(B^b \mathcal{A},i) \ar[d] \\
K(Ch^b \mathcal{A},q) \ar[r]
& K(C^b \mathcal{A},q) 
& K(Bi ^b \mathcal{A}, q) \ar[r]
& K(B^b \mathcal{A} , q) 
\end{tikzcd}
\end{center}
By Waldhausen's fibration theorem \cite[1.6.4.]{waldhausen1985algebraic}, the squares above are cartesian when the category $\mathcal{A}$ is discrete. Therefore we get the cartesian square below. 
% Since we have not checked Waldhausen's fibration theorem \cite[1.6.4.]{waldhausen1985algebraic}, we can not say whether the two squares above are cartesian. Nevertheless we still get the (no longer cartesian) commutative square below:
\begin{center}
\begin{tikzcd}
K[(Ch^b \mathcal{A},i) \xrightarrow{\Delta} (Bi^b \mathcal{A} ,i)] \ar[r] \ar[d] &
K[(C^b\mathcal{A}, i) \xrightarrow{\Delta} (B^b \mathcal{A} ,i)] \ar[d] \\
K[(Ch^b \mathcal{A},q) \xrightarrow{\Delta} (Bi^b \mathcal{A}, q)] \ar[r]&
K[(C^b\mathcal{A}, q) \xrightarrow{\Delta} (B^b \mathcal{A} ,q)]
\end{tikzcd}
\end{center}
In the case of topological exact categories, the square above is still commutative (but not necessarily cartesian). Also the argument below works for topological exact categories as well.

The lower left hand corner of the diagram is contractible as each of the two categories in the relative K-theory space are contractible. % Now, if $\mathcal{A}$ has the technical condition of "supporting exact sequences" \cite[1.4.]{grayson2012algebraic},%%%%%%%%%%%%%%%%%%%%%%%%%%%%%%%% 
The map $K(Ch^b \mathcal{A},i) \xrightarrow{K\Delta} K(Bi^b \mathcal{A}, i)$ is a homotopy equivalence,  because the functor $P^j:(C^b \mathcal{A},i)\to \mathcal{A} $ which sends a chain complex to the term in degree $j$ is exact, and by induction and the additivity theorem, induces an isomorphism $K(C^b\mathcal{A},i) \cong K(\coprod_{\mathbb{Z}} \mathcal{A})$ (cf. \cite[6.2.]{gillet1981riemann}). Similarly we can say the same for $K(B^b\mathcal{A},i)$, and note that $\Delta$ commutes with these isomorphisms and the identity map on $K(\coprod_{\mathbb{Z}} \mathcal{A})$. 
%by the same argument as Grayson uses (let note that additivity theorem also holds for topological categories, and we need the assumption of "supporting long exact sequences" for this step.). 
Thus the upper right hand corner of the diagram above is also contractible. Therefore we have the following sequence of natural maps:
% \begin{align*}
\begin{center} $
K [(Ch^b \mathcal{A},i) \xrightarrow{\Delta} (Bi^b \mathcal{A} ,i) ] \cong K \left[ [(Ch^b \mathcal{A},i) \xrightarrow{\Delta} (Bi^b \mathcal{A} ,i) ] \to 0 \right] \xleftarrow{\sim}  
K \left[ [(Ch^b \mathcal{A},i) \xrightarrow{\Delta} (Bi^b \mathcal{A} ,i) ] \to [(Ch^b \mathcal{A},q) \xrightarrow{\Delta} (Bi^b \mathcal{A} ,q)] \right] \xrightarrow{*} 
\Omega K\left[ [ (C^b \mathcal{A},i) \xrightarrow{\Delta} (B^b \mathcal{A} ,i)] \to [ (C^b \mathcal{A},q) \xrightarrow{\Delta} (B^b \mathcal{A} ,q)] \right] \xleftarrow{\sim} \Omega K[ (C^b \mathcal{A},q) \xrightarrow{\Delta} (B^b \mathcal{A} ,q) ] .$
\end{center}
% \end{align*}
When $\mathcal{A}$ is a discrete category, all of the maps above are homotopy equivalences, hence $G_1$ is a homotopy equivalence. Note that if the fibration theorem holds for topological exact categories, then $*$ (and therefore $G_1$) is a homotopy equivalence for topological categories as well.

The next step is to define the homotopy equivalence
\begin{equation*}
G_2: \Omega K[(B^b\mathcal{A}, q) \xrightarrow{\top} (C^b\mathcal{A} , q)] \to K[(C^b \mathcal{A},q) \xrightarrow{\Delta} (B^b \mathcal{A} ,q)]. 
\end{equation*}
This map is induced by the commutative diagram of \cite[4.5.]{grayson2012algebraic}. To be more precise, $G_2$ is the composition of the following sequence of maps:
\begin{center}
$\Omega K \left[(B^b\mathcal{A}, q) 
\xrightarrow{\top} (C^b\mathcal{A} , q) \right] \cong \Omega K \left[ \begin{tikzcd} 0 \ar[r] \ar[d] & 0 \ar[d] \\ (B^b\mathcal{A}, q) \ar[r, "\top"] & (C^b\mathcal{A},q) \end{tikzcd} \right] 
\xrightarrow{\sim} \Omega K \left[\begin{tikzcd} (C^b\mathcal{A},q) \ar[r, "1"] \ar[d, "\Delta"] & (C^b\mathcal{A},q) \ar[d, "1"] \\ (B^b\mathcal{A}, q) \ar[r, "\top"] & (C^b\mathcal{A},q) \end{tikzcd} \right]  
\xrightarrow{\sim} \Omega K \left[ \begin{tikzcd} (C^b\mathcal{A},q) \ar[r] \ar[d, "\Delta"] & 0 \ar[d] \\ (B^b\mathcal{A}, q) \ar[r] & 0 \end{tikzcd} \right] 
\cong K \left[ (C^b\mathcal{A},q) \xrightarrow{\Delta} (B^b\mathcal{A}, q) \right].$
\end{center}
Where we used the fact that $\top \circ \Delta = 1$ and K-theory of squares is a generalization of relative K-theory which was defined in \cite[Sec 4.]{grayson1992adams}\footnote{The proofs only rely on the additivity theorem, which holds for topological categories.}.

For the next step \cite[4.9.]{grayson2012algebraic}, Grayson defines a map $K\left((B^b\mathcal{A})^{\tau},q \right) \to \Omega K[ (B^b\mathcal{A},q) \xrightarrow{\top} (C^b\mathcal{A},q) ]$, which is a homotopy equivalence in the case of discrete categories. Instead we define the homotopy equivalence below for the (topological) exact category $\mathcal{A}$.
%But defining this map needs Waldhausen's fibration theorem, which we have not checked for topological categories. Hence we instead define a homotopy equivalence 
\begin{equation*}
G_3: K[ (B^b\mathcal{A},q) \xrightarrow{\top} (C^b\mathcal{A},q) ] \to  K[ (B^b\mathcal{A},q) \to (B^b\mathcal{A}, \tau) ]. 
\end{equation*}
Notice that we have the following commutative diagram, where each row is a cofiber sequence \footnote{Recall that fibration and cofiber sequences are the same in the category of spectra.}, and $F:(C^b\mathcal{A},q) \to (B^b\mathcal{A},\tau)$ is defined by $F(A^{\cdot}, d^{\cdot}) = (A^{\cdot}, d^{\cdot}, 0)$.
\begin{center}
\begin{tikzcd}
K(B^b\mathcal{A},q) \ar[r, "K\top"] \ar[d, "1"] &
K(C^b\mathcal{A},q) \ar[r] \ar[d, "KF"] &
K[ (B^b\mathcal{A},q) \xrightarrow{\top} (C^b\mathcal{A},q) ] \ar[d, "\exists G_3" ]  \\
K(B^b\mathcal{A},q) \ar[r] &
K(B^b\mathcal{A},\tau) \ar[r] &
K\left[ (B^b\mathcal{A},q) \to (B^b\mathcal{A}, \tau) \right]
\end{tikzcd}
\end{center}
This induces the desired map $G_3$. Similar to \cite[4.8.]{grayson2012algebraic}, we can argue that for the (topological) exact category $\mathcal{A}$, the maps $KF, K\top$ are inverses to each other up to homotopy. 
%In the case of discrete categories, $KF, K\top$ are inverses to each other by \cite[4.8.]{grayson2012algebraic}. 
(The argument relies on the fact that \emph{weak equivalence} between two functors induces a homotopy between the corresponding maps of K-theory \cite[1.3.1.]{waldhausen1985algebraic}, which also holds for topological categories; see \cite[2.1.]{segal1968classifying}.) Hence $KF$ is a homotopy equivalence for the (topological) exact category $\mathcal{A}$, which means that $G_3$ is also a homotopy equivalence. %%%% why did I include this? %%%%
% Note that the inclusion $\left( S_{\cdot}(B^b \mathcal{A})^{\tau} ,q \right) \to (S_{\cdot} B^b\mathcal{A},q)$ factors through the category $\left[ (B^b\mathcal{A},q) \to (B^b\mathcal{A}, \tau) \right]$.

Let $(C^b \mathcal{A})^x$ be the subcategory of chain complexes $(A^{\cdot} , d^{\cdot})$ in $C^b\mathcal{A}$, whose euler characteristic $\chi(A^{\cdot}) = \sum _n (-1)^n A^n$ is equal to zero.
When $\mathcal{A}$ is a discrete category, according to \cite[5.8.]{grayson2012algebraic}, as a corollary of Thomason's cofinality theorem \cite[1.10.1.]{thomason1990higher}, we have a homotopy equivalence $K\left( (C^b \mathcal{A})^x,q \right) \xrightarrow{\sim} V^1 K(C^b\mathcal{A},q)$. Then for the discrete exact category $\mathcal{A}$, one has the following sequence of homotopy equivalences:
\begin{equation} \label{Grayson's sequence of isomorphisms}
\begin{array}{c}
 \Omega K\left( (B^b \mathcal{A})^{\tau} ,q \right) \cong \Omega K\left( (B^b \mathcal{A})^{\beta} ,q \right) = \Omega K\left( (B^b \mathcal{A})^{\beta} ,\tau \right) \xrightarrow{\sim} \Omega K\left( (C^b \mathcal{A})^{x} ,q \right) \\
\xrightarrow{\sim} \Omega V^1 K(C^b\mathcal{A},q) \xleftarrow{\sim}  \Omega V^1 K(\mathcal{A}) \cong V^0 \Omega K(\mathcal{A})
\end{array}
\end{equation}
where the first homotopy equivalence is given by interchanging the top and the bottom differentials; the second is done by observing that $((B^b\mathcal{A})^{\beta},q) = ((B^b \mathcal{A})^{\beta} ,\tau)$; the third map is induced by the functor $\top:((B^b\mathcal{A})^{\beta}, \tau) \to ((C^b\mathcal{A})^x ,q)$ (note that this is well-defined since $(A^{\cdot}, d_1^{\cdot} , d_2^{\cdot})$ in $((B^b\mathcal{A})^{\beta} , \tau)$ is sent to $(A^{\cdot}, d_1^{\cdot})$, but acyclicity of $(A^{\cdot}, d_2^{\cdot})$ shows that the euler characteristic is zero.), which by theorem \cite[5.9.]{grayson2012algebraic} is a homotopy equivalence for discrete categories \footnote{The proof relies on Waldhausen's approximation theorem \cite[1.6.7.]{waldhausen1985algebraic} whose proof is quite long!}; the fourth one by \cite[5.8.]{grayson2012algebraic}, is a corollary of Thomason's cofinality theorem \cite[1.10.1.]{thomason1990higher}; the last one is true for any spectrum; and finally, the fifth map is induced by the inclusion $\mathcal{A}\to C^b\mathcal{A}$ as the chain complex concentrated in degree zero, which is a homotopy equivalence is by \cite[6.2.]{gillet1981riemann} (Also see \cite[1.11.7.]{thomason1990higher})\footnote{We need the extra assumption \cite[1.11.3.]{thomason1990higher} for this theorem to hold, however we \emph{are} assuming that $\mathcal{A}$ "supports long exact sequences", which ensures that there is no problem.}. However, since the map is going in the opposite direction, the homotopy equivalence does not necessarily induce a map for topological exact categories.
%Notice that there is an inclusion $\mathcal{A} \to C^b \mathcal{A}$ given by considering an object as a chain complex concentrated in degree zero. There is also the Euler characteristic $\chi: C^b\mathcal{A} \to \mathcal{A} $, and for discrete categories these two functors induce inverse maps on the level of K-theory spectra.

The proof of \cite[6.2.]{gillet1981riemann} relies on Waldhausen's fibration theorem as well. The author will check whether this holds for topological exact categories in a future work.

\begin{lem}  
Let $\mathcal{A}$ be a topological exact category, that "supports long exact sequences", and assume that $K(\mathcal{A})\to K(C^b\mathcal{A},q)$ is a homotopy equivalence, where the map is induced by inclusion as the chain complex concentrated in degree zero. Then the sequence of maps in \ref{Grayson's sequence of isomorphisms}, composed with the natural map $V^0\Omega K(\mathcal{A})\to \Omega K(\mathcal{A})$ (given by definition of Postnikov tower) factors through $K[(B^b\mathcal{A},q)\to (B^b\mathcal{A},\tau)]$.

% Therefore, there is a natural map in the homotopy category of spectra $$ K(Bi^b\mathcal{A},i)\to \Omega K(\mathcal{A}).$$
\end{lem}  

Before proving the lemma above, let us summarize \cite[Sec 6.]{grayson2012algebraic} on what happens when the (topological) exact category $\mathcal{A}$ does not "support long exact sequences".

The pseudo-abelianization (cf. proposition \ref{pseudo-abelian category}) $\tilde{\mathcal{A}}$ of the (topological) exact category $\mathcal{A}$ inherits (both the topological and) the exact structure of $\mathcal{A}$. Also $\mathcal{A}$ embeds in $\tilde{\mathcal{A}}$ as a cofinal subcategory. Hence $K_n(\mathcal{A})\to K_n(\tilde{\mathcal{A}})$ is an isomorphism when $n>0$ and is injective for $n=0$. Similar to \cite[6.3.]{grayson2012algebraic}, The induced inclusions $Ch^b(\mathcal{A}) \hookrightarrow Ch^b(\tilde{\mathcal{A}})$ and $Bi^b(\mathcal{A})\hookrightarrow Bi^b(\tilde{\mathcal{A}})$ are also cofinal, and by repeating the argument in \cite[6.2.]{grayson2012algebraic}, the natural map from the cofiber of $K(Ch^b(\mathcal{A}))\to K(Bi^b(\mathcal{A}))$ to the cofiber of $K(Ch^b(\tilde{\mathcal{A}}))\to K(Bi^b(\tilde{\mathcal{A}}))$ is a homotopy equivalence \footnote{The reason why the $n$'th homotopy groups are isomorphic follows from cofinality when $n>0$. But for $n=0$ an extra argument is needed.}. Again by cofinality, $V^0 \Omega K(\mathcal{A})\to V^0 \Omega K(\tilde{\mathcal{A}})$ is a homotopy equivalence.

Since $\tilde{\mathcal{A}}$ is pseudo-abelian hence as explained before, $\tilde{\mathcal{A}}$ "supports exact sequences", and when $\mathcal{A}$ is a discrete category, then there is an induced natural map in the homotopy category of spectra 
\begin{equation*}
\tau^G_{\mathcal{A}}: \quad K(Bi^b(\mathcal{A}))\to K(Bi^b(\tilde{\mathcal{A}}))\xrightarrow{\tau^G_{\tilde{\mathcal{A}}}} V^0 \Omega K(\tilde{\mathcal{A}}) \xleftarrow{\sim} V^0\Omega K(\mathcal{A}).
\end{equation*}

\begin{proof}[Proof of Lemma]
% First assume that $\mathcal{A}$ "supports long exact sequences".
Let $(A^{\cdot}, d_1^{\cdot}, d_2^{\cdot})$ be an object of $(B^b \mathcal{A},q)^{\tau}$. By definition, the top chain complex $(A^{\cdot}, d_1^{\cdot})$ is acyclic. This goes to $(A^{\cdot}, d_2^{\cdot}, d_1^{\cdot})$ through the first map in the sequence \ref{Grayson's sequence of isomorphisms}, and the second map is the identity. The third map sends it to the top chain complex $(A^{\cdot}, d_2^{\cdot})$ in $((C^b \mathcal{A})^x,q)$. 
% The fourth map $K((C^b\mathcal{A},q)^x) \to V^1 K((C^b\mathcal{A},q)$ factors as $K((C^b \mathcal{A},q)^x) \to K(C^b\mathcal{A},q) \to V^1 K(C^b \mathcal{A},q)$, where the first one is induced by inclusion of categories, and the second one is given by definition of Postnikov tower. Hence the object $(A^{\cdot}, d_2^{\cdot})$ in $(C^b \mathcal{A},q)^x$ is sent to $(A^{\cdot}, d_2^{\cdot})$ in $(C^b \mathcal{A},q)$. %Then this maps by identity to $\Omega V^1 K(C^b\mathcal{A},q)$, and so on. but we can instead, just map it directly to $\Omega K(C^b \mathcal{A},q) \to \Omega K(\mathcal{A})$.\\
The composition 
$$\Omega K\left( (C^b \mathcal{A})^{x} ,q \right) \sim \Omega V^1 K(C^b\mathcal{A},q) \sim  \Omega V^1 K(\mathcal{A}) \sim V^0 \Omega K(\mathcal{A}) \to \Omega K(\mathcal{A})$$ 
is equal to the composition 
$$G_5: \quad \Omega K\left( (C^b \mathcal{A})^{x} ,q\right)  \to \Omega K(C^b\mathcal{A},q)\sim \Omega K(\mathcal{A}),$$ 
where the first map is induced by inclusion of categories, and the second is given by the hypothesis of the lemma.
%which sends the chain complex $(A^{\cdot}, d_2^{\cdot})$ in $(C^b \mathcal{A},q)^x$, first to the same chain complex in $(C^b\mathcal{A},q)$ which is then sent to its Euler characteristic.

The natural map $\Omega K((B^b\mathcal{A})^{\tau},q) \to K[ (B^b\mathcal{A},q) \to (B^b\mathcal{A}, \tau) ]$ is induced by inclusion. This sends the object $(A_{jk}^{\cdot}, d_{1,jk}^{\cdot}, d_{2,jk}^{\cdot})_{0\leq j \leq k \leq n}$ of $S_n(B^b \mathcal{A})^{\tau}$ to the pair $\left( (A_{jk}^{\cdot}, d_{1,jk}^{\cdot}, d_{2,jk}^{\cdot})_{0\leq j \leq k \leq n}, (0)_{0\leq j \leq k \leq n+1}\right)$.
Now, define the map $G_4: K[ (B^b\mathcal{A},q) \to (B^b\mathcal{A}, \tau) ] \to \Omega K(C^b,q)$ by %sending the object (L_{jk}^{\cdot}, S_{1,jk}^{\cdot}, S_{2,jk}^{\cdot})_{0\leq j \leq k \leq n+1}
\begin{equation*}
G_4 \left( (A_{jk}^{\cdot}, d_{1,jk}^{\cdot}, d_{2,jk}^{\cdot})_{0\leq j \leq k \leq n}, (A',d'_1,d'_2)_{n+1} \right) = (A_{jk}^{\cdot}, d_{2,jk}^{\cdot})_{0\leq j \leq k \leq n},
\end{equation*}
where $(A',d'_1,d'_2)_{n+1}$ is an object of $S_{n+1}(B^b\mathcal{A},\tau)$, the first term is an object of $\left[ (B^b\mathcal{A},q) \to (B^b\mathcal{A}, \tau) \right]$  %which we temporarily abbreviate to $(H,T;G,S)_n$. We claim that $(H_{ij}^{\cdot}, T_{2,ij}^{\cdot})_{0\leq i \leq j \leq n}$, abbreviated to $(H,T)_n$ is an object in $S_n((C^b \mathcal{A},q)^x)$. This follows from
and the second term is an object of $S_n(C^b \mathcal{A},q)$. Then use the natural homotopy equivalence $\Vert wS_{\cdot}S_{\cdot} \mathcal{E} \Vert \cong \Omega \Vert w S_{\cdot} \mathcal{E}\Vert$ for the topological Waldhausen category $\mathcal{E}=(C^b\mathcal{A},q)$. 

% Now, for a general topological exact category $\mathcal{A}$, following the argument for discrete categories, we have the maps: \begin{equation*}
% K(Bi^b(\mathcal{A}))\to K(Bi^b(\tilde{\mathcal{A}}))\to  \Omega K(\tilde{\mathcal{A}}) \leftarrow \Omega K(\mathcal{A}). \end{equation*}
% By hypothesis of the lemma, the last map admits a section. This finishes the proof.
  
\end{proof}

\subsection{Higson's Functor}

Generalizing a construction given by Higson in \cite[Page 6.]{higson1995c}, for a $C^*$-algebra $A$ we define the functor $\tau_A^H : C(\mathfrak{D} / \mathfrak{C})_A \to B(\mathfrak{D} / \mathfrak{C})_A $ below.
\begin{Def} \label{The large diagram functor tau2 definition}
Let $(\rho^{\cdot} , T^{\cdot})$ be a chain complex in $C(\mathfrak{D}/ \mathfrak{C})_A$. Define $\tau_A^H(\rho^{\cdot} , T^{\cdot})$ to be the binary chain complex whose $n$'th term is the graded object $\nu ^n = \left( \oplus_{i= -\infty}^{n-1} (\rho^{n-1} \oplus \rho^n) \right) \oplus \rho^n $ in $(\mathfrak{D}/ \mathfrak{C})_A$, where the last piece is of degree $n$. The top differential (temporarily denoted by) $\top^n$ from $\nu^n$ to $\nu^{n+1}$ is a degree $1$ map, where its $i$'th degree piece from $\rho^{n-1} \oplus \rho^n$ (of degree $i$) to $\rho^n \oplus \rho^{n+1}$ (of degree $i+1$) is the trivial one (i.e. is identity on $\rho^n$ and zero on $\rho^{n-1}$.) for $i\leq n-1$, and its $n$'th degree piece is equal to $\rho^n \xrightarrow{T^n} \rho^{n+1}$.  
The bottom differential (temporarily denoted by) $\bot^n$ from $\nu^n$ to $\nu^{n+1}$ is a degree $0$ map, where its $i$'th degree piece from $\rho^{n-1} \oplus \rho^n$ to $\rho^n \oplus \rho^{n+1}$ is again the trivial one (i.e. is identity on $\rho^n$ and zero on $\rho^{n-1}$.) for $i\leq n-1$, and its $n$'th degree piece is the trivial inclusion $\rho^n \xrightarrow{(Id,0)} \rho^n \oplus \rho^{n+1}$.
\end{Def}

\begin{center}
\scalebox{0.8}{ 
\begin{tikzpicture}%[shorten <=10pt,shorten >=10pt,>=stealth,ampersand replacement=\&]
[ampersand replacement=\&]
  \matrix(m)[matrix of math nodes,nodes in empty cells, row sep=1em,column sep=0em,minimum width=0em]
{
	\& \& \& \vdots 
	\& \& \& \& \ddots \\
	\ldots \& \oplus \&
	(\rho^{n-1} \& \oplus \& 
	\rho^{n}) \& \oplus \& 
	(\rho^{n-1} \& \oplus \&
	\rho^{n}) \& \oplus \&
	\& \rho^{n} \& \& \& \& \& \& \& \& \\
	\ldots \& \oplus \&
	(\rho^{n} \& \oplus \& 
	\rho^{n+1}) \& \oplus \&  
	(\rho^{n} \& \oplus \&
	\rho^{n+1}) \& \oplus \& 
	(\rho^{n} \& \oplus \&
	\rho^{n+1}) \& \oplus \& 
	\&  \rho^{n+1} \& \& \& \& \\
	\ldots \& \oplus \&
	(\rho^{n+1} \& \oplus \& 
	\rho^{n+2}) \& \oplus \&  
	(\rho^{n+1} \& \oplus \&
	\rho^{n+2}) \& \oplus \& 
	(\rho^{n+1} \& \oplus \&
	\rho^{n+2}) \& \oplus \&
	(\rho^{n+1} \& \oplus \&
	\rho^{n+2}) \& \oplus \& 
	\& \rho^{n+2} \\
	\& \& \& \vdots 
	\& \& \& \& \vdots 
	\& \& \& \& \vdots
	\& \& \& \& \vdots
	\& \& \& |[white]| \rho^{n+1} \& \vdots \\
    };
	\path[shorten <=10pt ,shorten >=15pt ,>=stealth, ->] 
	(m-1-4) edge [blue] node [] {} (m-2-8)
	(m-1-8) edge [blue] node [midway, above] {$T^{n-1}$} (m-2-12)
	(m-2-4) edge [blue] node [] {} (m-3-8)
	(m-2-8) edge [blue] node [] {} (m-3-12)
	(m-2-12) edge [blue] node [midway,above] {$T^{n}$} (m-3-16)
	(m-3-4) edge [blue] node [] {} (m-4-8)
	(m-3-8) edge [blue] node [] {} (m-4-12)
	(m-3-12) edge [blue] node [] {} (m-4-16)
	(m-4-4) edge [blue] node [] {} (m-5-8)
	(m-4-8) edge [blue] node [] {} (m-5-12)
	(m-4-12) edge [blue] node [] {} (m-5-16)
	(m-3-16) edge [blue] node [midway,above] {$T^{n+1}$} 	(m-4-20)
	(m-4-16) edge [blue] node [] {} (m-5-20) 
;	\path[->]
	(m-1-4) edge [red] node [] {} (m-2-4)
	(m-1-8) edge [red] node [] {} (m-2-8)
	(m-2-4) edge [red] node [] {} (m-3-4)
	(m-2-8) edge [red] node [] {} (m-3-8)
	(m-2-12) edge [red] node [] {} (m-3-12)
	(m-3-4) edge [red] node [] {} (m-4-4)
	(m-3-8) edge [red] node [] {} (m-4-8)
	(m-3-12) edge [red] node [] {} (m-4-12)
	(m-3-16) edge [red] node [] {} (m-4-16)
	(m-4-4) edge [red] node [] {} (m-5-4)
	(m-4-8) edge [red] node [] {} (m-5-8)
	(m-4-12) edge [red] node [] {} (m-5-12)
	(m-4-16) edge [red] node [] {} (m-5-16)
	(m-4-20) edge [red] node [] {} (m-5-20)   
    ;
\end{tikzpicture}
 }
\end{center}

It is easy to see that the bottom chain complex is split exact, and the top chain complex is exact iff the original chain complex $(\rho^{\cdot}, T^{\cdot})$ is exact. This process is functorial with respect to chain maps in a trivial way. Finally, note that if we start with a chain complex of length $n$, then we will get a binary chain complex of length $n+1$. Hence we also have the natural functor $\tau^H_{A}: Ch^b (\mathfrak{D}/ \mathfrak{C})_{A}\to Bi^b(\mathfrak{D}/ \mathfrak{C})_A$. This functor is not exact; however we can tweak the structures of the categories to obtain an exact functor.
%for example if $S$ is an isomorphism of chain complexes, then $\tau^H_A(S)$ does not have to be an isomorphism of binary chain complexes, as infinite direct sum of compact operators no longer has to be compact. Before giving the following definition, recall that a chain complex in $\mathfrak{D}_A $ is exact if there is a contracting homotopy.

\begin{Def}	\label{cat of bdd exact chain complex with different structure-def}
Let $A$ be a $C^*$-algebra and let $Ch'(\mathfrak{D}/ \mathfrak{C})_A, Bi'(\mathfrak{D}/ \mathfrak{C})_A$ denote the categories with the same objects as $Ch^b(\mathfrak{D}/\mathfrak{C})_A, Bi^b(\mathfrak{D}/\mathfrak{C})_A$ respectively, but with morphisms and exact structure coming from the category $\mathfrak{D}_A$. 
To be precise, a morphism in $Ch'(\mathfrak{D}/\mathfrak{C})_A$ from the chain complex $(\rho^{\cdot}, T^{\cdot})$ to $(\nu^{\cdot}, S^{\cdot})$ is given by a chain map $f^n:\rho^n\to \nu^n$ in the category $\mathfrak{D}_A$, and morphisms in $Bi'(\mathfrak{D}/\mathfrak{C})_A$ are defined similarly as a chain map in $\mathfrak{D}_A$ with respect to both the top and the bottom chain complex. We say a sequence of chain complexes in $Ch'(\mathfrak{D}/\mathfrak{C})_A$ is exact, iff the sequence is exact at each degree in $\mathfrak{D}_A$, and similarly define the exact structure on $Bi'(\mathfrak{D}/\mathfrak{C})_A$.

There are natural functors $Ch'(\mathfrak{D}/ \mathfrak{C})_A\to Ch^b(\mathfrak{D}/\mathfrak{C})_A$ and $Bi'(\mathfrak{D}/\mathfrak{C})_A\to Bi^b(\mathfrak{D}/\mathfrak{C})_A$. These functors are exact, since exactness in $\mathfrak{D}_A$ guarantees exactness in $(\mathfrak{D}/\mathfrak{C})_A$.
\end{Def}

\begin{lem}
The functor $\tau^H_A$ defined in \ref{The large diagram functor tau2 definition} induces an exact functor % from $Ch'(\mathfrak{D}/\mathfrak{C})_A$ to $Bi'(\mathfrak{D}/\mathfrak{C})_A$.
\begin{equation}\label{the large diagram functor tau 2 equation}
\tau^H_A: Ch'(\mathfrak{D}/\mathfrak{C})_A \to Bi'(\mathfrak{D}/ \mathfrak{C})_A.
\end{equation}
Therefore, we have a natural map of K-theory spectra
\begin{equation*}
\tau^H_A: K( Ch'(\mathfrak{D}/\mathfrak{C})_A ) \to K (Bi'(\mathfrak{D}/ \mathfrak{C})_A) . 
\end{equation*}
\end{lem}

This is proved by observing that infinite direct sum of identities is equal to identity in $\mathfrak{D}_A$. Note that this is \emph{not} true in the paschke category $(\mathfrak{D}/ \mathfrak{C})_A$.

\begin{proof}
Let $(\rho^{\cdot}_i, T^{\cdot}_i)$ denote objects in $Ch'(\mathfrak{D}/\mathfrak{C})_A $ for $i\in \mathbb{Z}$, and let $f^{\cdot}_i:(\rho^{\cdot}_i, T^{\cdot}_i) \to (\rho^{\cdot}_{i+1}, T^{\cdot}_{i+1})$ be morphisms that give an exact sequence in $Ch'(\mathfrak{D}/\mathfrak{C})_A $, with the degree-wise contracting homotopy given by $g_i^{\cdot}:(\rho^{\cdot}_{i+1}, T^{\cdot}_{i+1}) \to(\rho^{\cdot}_i, T^{\cdot}_i)$.
Then we need to show that $\tau^H_A(g^{\cdot}_i) \tau^H_A(f^{\cdot}_i) + \tau^H_A(f^{\cdot}_{i-1}) \tau^H_A(g^{\cdot}_{i-1})$ is equal to identity at each degree of $\tau^H_A(\rho^{\cdot}_i, T^{\cdot}_i)$. This is true since at degree $n$ this is given by infinite direct sum $g^n_i f^n_i + f^n_{i-1}g^n_{i-1}$ and $g^{n-1}_i f^{n-1}_i + f^{n-1}_{i-1}g^{n-1}_{i-1}$. But by assumption, each term is equal to identity in $\mathfrak{D}_A$ , hence their infinite direct sum is also equal to identity.
\end{proof}

None of this process works in the Calkin-Paschke category $(\mathfrak{D}/ \mathfrak{C})'_A$, as infinite direct sums of $\rho':A\rightarrow (\mathfrak{B}/\mathfrak{K})(H)$ is not necessarily defined, since infinite direct sum of compact operators does not have to be compact. In fact if the infinite direct sum of $\rho'$ is well-defined, then by an \emph{Eilenberg swindle} argument we can show that the class corresponding to $\rho'$ in $K_1^{top}(A) = Ext(A)$ defined in \cite{brown1977extensions} is zero.% But $\tau_A^H$ can be defined on the subcategory $(\mathfrak{D}/ \mathfrak{C})_A''$.
% Now, we are in the position to define:
% \begin{Def}\label{The functor tau2 definition}
% Let $A$ be a separable $C^*$-algebra. Then we denote the composition of natural maps of spectra $K(Ch^b(\mathfrak{D}/\mathfrak{C})_A ,i) \xrightarrow{\tau^G_A} K(Bi^b(\mathfrak{D}/ \mathfrak{C})_A ,i) \to K(Bi^b(\mathfrak{D}/ \mathfrak{C})'_A ,i) \to  V^0 \Omega K^{top}(A)$ by $\tau^H_A$, where the last map was obtained in proposition \ref{Grayson's map for topological categories proposition}.
% \end{Def}
%%%%%%%%%%%%

\section{Complex Manifolds and the Dolbeault Complex} \label{section differential operator}

%In this section, we are planning on using the Dolbeault complex to construct an exact sequence in the Paschke category. Then we show that this process induces an exact functor from the category of locally free sheaves on a complex manifold to the category of bounded acyclic chain complexes in the corresponding Paschke category. For a complex manifold $X$, we name this functor to be $\tau^D_X: \mathcal{P}(X) \rightarrow Ch(\mathfrak{D}/ \mathfrak{C})_{C_0(X)}$. Then this process composed with the one from the Definition \ref{The functor tau2 definition} gives us a map $\tau \coloneqq \tau^D_X \circ \tau_{C_0(X)}^2$ from the algebraic K-theory of a complex manifold to its topological K-homology.

This section will contain a great deal of computations, and to ease the readability, we will fix some of our notations.

\begin{notation}
Fix $\chi(t)= \frac{t}{\sqrt{1+t^2}}$. The functions $\chi, \phi$ will be used for functional calculus. The letter $X$ denotes manifolds, $U,V$ are used for open subsets of the manifold, $\lambda$ for a partition of unity, and $\gamma$ for cutoff functions. The letters $D,d$ will be used for differential operators, and $\bar{\partial}$ will denote the Dolbeault operator.

We will use $E$ for vector bundles, $g,h$ will be reserved for a metric on the manifold, and on the bundle respectively. The letters $\alpha,\beta$ will be isomorphisms of vector bundles, $\varphi, \sigma, \psi$ will be maps of vector bundles. 

The letter $I$ will be used as a map of Hilbert spaces induced by identity map on a bundle (with different choices of metrics), $\pi$ will refer to projection onto the $L^2$-integrable functions on an open subset, and $\iota$ will denote extension by zero of $L^2$ sections on an open subset to the whole space. 
\end{notation}

\subsection{The Dolbeault Functor}

For the definition and basic properties of functional calculus, see the appendix \hyperref[functional calc appendix]{B}. One could apply functional calculus to an essentially self-adjoint operator, and in certain cases we get interesting properties.
% if we apply it to a sufficiently well behaved operator, then 

\begin{lem}\label{basic props of func calc on mfld lem}
Let $X$ be a differentiable manifold, $E$ a differentiable vector bundle over $X$, and let $D\in \text{Diff}_1(E,E)$ be a differential operator of order $1$. Consider the representation $\rho:C_0(X)\to \mathfrak{B}(L^2(X,E))$.
\begin{enumerate}\setlength\itemsep{0em}
\item Let $\phi$ be a bounded Borel function on $\mathbb{R}$, whose Fourier transform is compactly supported. Then $\phi(D)$ is a well-defined bounded operator acting on $L^2(X,E)$, which is in fact, pseudo-local. \cite[10.3.5, 10.6.3.]{higson2000analytic} 
\item Assume in addition that $D$ is an elliptic operator. Let $\phi\in C_0(\mathbb{R})$, then $\phi(D):L^2(X,E)\to L^2(X,E)$ is  a locally compact operator.  \cite[10.5.2.]{higson2000analytic}
\end{enumerate}
\end{lem}

Now we are ready to define the functor $\hat{\tau}^D_X$.

\begin{Def} \label{completed Dolbeault complex def}
Let $X$ be a complex manifold of dimension $n$, and let $E$ be a holomorphic vector bundle on $X$.
%Let $X$ be a \emph{compact} complex manifold, $E$ a holomorphic vector bundle on $X$, and let $A=C(X)$ be the $C^*$-algebra of complex valued continuous functions on $X$. 
We will use the Dolbeault complex \ref{Dolbeault complex} to define an exact sequence in the Paschke category of $C_0(X)$.

Fix some hermitian metric $g$ on $X$ and a Hermitian metric $h$ on $E$ and let $H^i$ be the space of $L^2$-integrable sections of the bundle $\wedge^{0,i} T^*X\otimes E$ over $X$. There are natural representations $\rho^i: C_0(X) \to \mathfrak{B}(H^i) $ given by point-wise multiplication of a function on $X$ with the $L^2$-section. Let $\bar{\partial}^*_E $ be the formal adjoint of the Dolbeault operator $\bar{\partial}_E$ (with respect to the metrics $g,h$.), and consider the essentially self-adjoint differential operator $D_E= \bar{\partial}_E + \bar{\partial}_E^*$ of order $1$ \cite[11.8.1.]{higson2000analytic}. %Since $X$ is compact, then $D_E$ is essentially self-adjoint, hence
Therefore we can apply functional calculus to $D_E$ with respect to the function $\chi(t)= \frac{t}{\sqrt{1+t^2}}$, to obtain a bounded operator $ \frac{D_E}{\sqrt{1+D_E^2}} = \chi(D_E) \in \mathfrak{B}(\oplus_i H^i)$. By lemma \ref{basic props of func calc on mfld lem}, this is a pseudo-local operator with respect to the $\rho^i$'s, so if $\chi_i(D_E)= \frac{\bar{\partial}_i}{\sqrt{1+(D_E)^2}}$ denotes the restriction of $\chi(D_E)$ to $\mathfrak{B}(H^i,H^{i+1})$, then we have the following chain complex in the Paschke category $(\mathfrak{D}/ \mathfrak{C})_{C_0(X)}$.%, which we will abbreviate to $\hat{\tau}^D_{X,g}(E,h)$.
\begin{equation} \label{completed-Dolbeault-complex}
\hat{\tau}^D_{X,g}(E,h): \quad 0\to \rho^0 \xrightarrow{\chi_0(D_E) } \rho^1 \xrightarrow{\chi_1(D_E) } \ldots \xrightarrow{\chi_{n-1}(D_E) } \rho^n \to 0.
\end{equation}
To show that this is in fact an exact sequence in the Paschke category, we need to find pseudo-local operators $P_i:H^{i+1}\to H^i $ which give a contracting homotopy \footnote{The operators $P_i$ are also called the \emph{parametrices}}, i.e. $P_i \chi_i(D_E) + \chi_{i-1}(D_E) P_{i-1} - Id_{H^i}$ is a locally compact operator. It is easy to see if $P_i= \frac{\bar{\partial}^*_i}{\sqrt{1+D^2_E}}:H^{i+1}\to H^i$, then $P_i \chi_i(D_E) + \chi_{i-1}(D_E) P_{i-1} - Id_{H^i} = \frac{1}{1+D^2_E}$, which is locally compact by lemma \ref{basic props of func calc on mfld lem}. This shows that \ref{completed-Dolbeault-complex} is an exact sequence. %\footnote{We could have also used the more general argument of lemma \ref{parametrices for a complex atiyah lemma}, of course with no use of the word "uniform".} 
\end{Def}

\begin{prop} \label{completed dolbeault cx cmpt mnfld q.i. to dolb cx prop}
Let $X$ be a compact complex manifold and $E$ a holomorphic vector bundle. Then the chain complex \ref{completed-Dolbeault-complex} considered as a complex of Hilbert spaces and bounded operators, is quasi-isomorphic to the Dolbeault complex \ref{Dolbeault complex} with coefficients in $E$.
\end{prop}
\begin{proof}
It is easy to see that the diagram below commutes:
\begin{center}
\begin{tikzcd}
0 \ar[r] 
& \mathscr{A}_X^{0,0}(E) \ar[r, "\bar{\partial}_0 "] \ar[dd, hookrightarrow] 
& \mathscr{A}_X^{0,1}(E) \ar[r, "\bar{\partial}_1 "] \ar[dd, hookrightarrow, " \left( 1+D_E ^2\right)^{-1/2}" ]
& \ldots \ar[r, "\bar{\partial}_{n-1}"]
& \mathscr{A}_X^{0,n}(E) \ar[r] \ar[dd, hookrightarrow, "\left( 1+D_E^2\right)^{-n/2}"] 
& 0 \\
& & & & & \\
0 \ar[r]
& H^0 \ar[r, "\chi_0(D_E)"]
& H^1 \ar[r, "\chi_1(D_E)"]
& \ldots \ar[r, "\chi_{n-1}(D_E)"]
& H^n \ar[r]
& 0
\end{tikzcd}
\end{center}
Since the image of the vertical maps are dense, then by the Hodge-decomposition \ref{hodge-decomposition theorem}, we can see this sends Harmonic forms isomorphically to the cohomology of the complex below.
%the map above induces a quasi-isomorphism of the two chain complexes above.
\end{proof}

The definition \ref{completed Dolbeault complex def}, is not very easy to work with %over non-compact manifolds, or 
when we restrict to open subsets, because restriction of an essentially self-adjoint operator to an open subset is not necessarily essentially self-adjoint. We will give an equivalent definition in \ref{general functional calculus open cover definition} for any symmetric elliptic operator. 

% We would like to generalize the process of example \ref{completed Dolbeault complex compact example} above for a general complex manifold as well. However, there are a few obstacles. The most significant one is that over a non-compact space, the symmetric differential operator $D_E=\bar{\partial}_E + \bar{\partial}_E^*$ no longer needs to be essentially self-adjoint, hence we can't simply apply functional calculus on it to obtain a bounded operator. One also should note that the Hodge-decomposition has no counterpart for non-compact manifolds, hence we can not show that the Dolbeault complex and the chain complex we wish to define in the Paschke category are quasi-isomorphic.

% To get around the problem with $D_E$ not being essentially self-adjoint, there are two remedies. The first is based on lemma \ref{compactly supported symm diff op is self adj lem} that if we can obtain a differential operator from $D_E$ which is compactly supported, then we can apply functional calculus to  it. 

\begin{lem}\cite[10.8.4.]{higson2000analytic} \label{choice of cover and cutoffs don't matter for func calc-lem}
Let $X$ be a differentiable manifold, $E$ a differentiable vector bundle on it, $U\subset X$ an open subset, and $D_1,D_2 \in \text{Diff}_1(E,E)$ are order one essentially self-adjoint differential operators, so that $D_1\vert_U = D_2 \vert_U$. Then if $f\in C_0(U)$, we have $\rho(f)\chi(D_1)- \rho(f)\chi(D_2)$ is a compact operator, where $\rho:C_0(U)\to \mathfrak{B}( L^2(X,E))$ is given by pointwise multiplication, and $\chi(t)=\frac{t}{\sqrt{1+t^2}}$.  
\end{lem}

\begin{Def}\label{good cover def}
Let $X$ be a locally compact, and Hausdorf topological space. We say that the open cover $\{U_j\}_j$ is a \emph{good cover}, if it is countable, locally finite, and each open set $U_j$ is relatively compact. 
\end{Def}

\begin{Def}\cite[10.8.]{higson2000analytic} \label{general functional calculus open cover definition}
Let $X$ be a (non-compact) differentiable manifold, $E$ a differentiable vector bundle on $X$, and let $D\in \text{Diff}_1(E,E)$ be a symmetric elliptic differential operator of order $1$. Let $\{U_j\}_j$ be a good cover (definition \ref{good cover def}), and let $\{ \lambda_j\}_j$ be a partition of unity subbordinate to the cover, and let $\{\gamma_j\}_j$ be compactly supported non-negative continuous functions, so that $\gamma_j \vert_{U_j}$ is equal to (the constant function) one. Then the symmetric differential operator $D_j= \gamma_j D \gamma_j$ is supported on a compact set, hence by lemma \ref{compactly supported symm diff op is self adj lem}, is essentially self-adjoint. Therefore if the representation $\rho:C_0(X)\to \mathfrak{B}(L^2(X,E))$ (where the $L^2$-completion is of course defined with respect to a choice of a metric on $X$ and one on $E$.) is given by pointwise multiplication, then we can define 
\begin{equation}\label{general functional calculus open cover equation}
\chi_D \coloneqq \sum_j \rho( \lambda_j^{1/2}) \chi(D_j) \rho( \lambda_j^{1/2}),
\end{equation}
as the partial sums are bounded in norm and the series converges in the strong operator topology.

One can see that $\chi_D$ is self-adjoint. $\chi_D$ depends on the choice of the open cover, the partition of unity, and the cut-off functions $\gamma_j$'s, but if $f\in C_0(X)$ is compactly supported, then $\rho(f)\chi_D$ has only finitely many terms, hence by lemma \ref{choice of cover and cutoffs don't matter for func calc-lem}, if $D_1\in \text{Diff}_1(E,E)$ is any essentially self-adjoint differential operator which agrees with $D$ on support of $f$, then $\rho(f)\chi_D - \rho(f)\chi(D_1)$ is compact, and in particular if $D$ is itself essentially self-adjoint, then $\chi_D-\chi(D)$ is locally compact. Hence the choices do not matter up to locally compact operators. Therefore, we have a well-defined operator $\chi_D$ in the Paschke category $(\mathfrak{D}/ \mathfrak{C})_{C_0(X)}$.
\end{Def}

% In the previous section, we defined $\hat{\tau}^D_{X,g}$ for complex manifolds $X$ with a metric $g$. In this subsection, we plan to show that $\hat{\tau}^D_X$ can be turned into a functor in proposition \ref{functoriality of tau1 prop}. For the rest of this section, we have $\chi(t)=\frac{t}{\sqrt{1+t^2}}$.

\begin{Def} \label{category of bundles with choice of metric def}
Let $X$ be a complex manifold, then denote the category of holomorphic vector bundles on $X$ by $\mathcal{P}(X)$. This is an exact category.

It is straightforward to show that $\mathcal{P}(X)$ has a small skeletal subcategory. For each vector bundle on $X$, there is a \emph{set} of metrics, hence if we denote the category of holomorphic vector bundles with a choice of metric by $\mathcal{P}_{\text{m}}(X)$, (i.e. objects are pairs $(E,h)$ of a 
holomorphic vector bundle with a hermitian metric, and morphisms are bundle maps) then without loss of generality, we can assume that this is a small category. This inherits an exact structure from the category $\mathcal{P}(X)$. 

Let $g$ be a hermitian metric on $X$, then denote the \emph{bounded} category of holomorphic vector bundles with a choice of metric by $\mathcal{P}_{\text{m, b}}(X,g)$, where objects again are pairs $(E,h)$ of a holomorphic vector bundle with a choice of a hermitian metric, and a morphism from $(E_1,h_1)$ to $(E_2,h_2)$ is a map of bundles $E_1\to E_2$, so that the induced map $L^2(X,\wedge^{0,*}T^*X\otimes E_1) \to L^2(X,\wedge^{0,*}T^*X\otimes E_2) $ is a bounded map of Hilbert spaces. We say a sequence $\ldots \to (E_{i-1},h_{i-1})\to (E_i,h_{i}) \to (E_{i+1},h_{i+1}) \to \ldots$ is exact in $\mathcal{P}_{\text{m, b}}(X)$, if there are smooth maps of bundles $\sigma_i:E_{i+1}\to E_i$ which are a contracting homotopy, and also the induced maps $L^2(X,\wedge^{0,*}T^*X\otimes E_{i+1})\to L^2(X,\wedge^{0,*}T^*X\otimes E_i) $ are bounded. \footnote{Notice that a smooth contracting homotopy always exists \cite[1.4.11.]{atiyah1967k}. The only condition here is boundedness of $\sigma_i$'s. } This is again a small category.

Note that there $\mathcal{P}_{\text{m, b}}(X)$ is a subcategory of $\mathcal{P}_{\text{m}}(X)$, and there is a forgetful map $\mathcal{P}_{\text{m}}(X)\to \mathcal{P}(X)$. Both of these functors are exact.
\end{Def}

\begin{prop}\label{functoriality of tau1 prop}
Let $X$ be a complex manifold, and let $g$ be a hermitian metric of bounded geometry on $X$. Recall from definition \ref{cat of bdd exact chain complex with different structure-def} that $Ch'(\mathfrak{D}/\mathfrak{C})_{C_0(X)}$ denotes the category of bounded acyclic chain complexes in $(\mathfrak{D}/\mathfrak{C})_{C_0(X)}$ where the exact structure is induced by that of $\mathfrak{D}_{C_0(X)}$.
% for an exact category $\mathcal{A}$, the category $Ch'(\mathcal{A})$ denotes the exact category of bounded acyclic chain complexes in $\mathcal{A}$.

The map $\hat{\tau}^D_{X,g}$ defined in \ref{completed Dolbeault complex def} induces an exact functor from $\mathcal{P}_{\text{m,b}}(X)$ to $Ch'(\mathfrak{D}/\mathfrak{C})_{C_0(X)}$. 
\end{prop}

The proof of proposition above, will take the entirety of the next subsection, as we will need to prove a series of technical lemmas (which are probably known to the experts). We include them with great details for readers who may not have a background in the topic.

Throughout this section, we had fixed a single metric on the manifold $X$ and do the rest of our computations. % As was mentioned in remark \ref{change metric on manifold changes complex in pashcke cat -remark}, we may not expect to define $\tau_{X,g}$ in the Paschke category independent of the metric. 
Let us investigate effect of the choice of the % However we will show that the choice of the 
metric $g$ %does not affect the homotopy class of the map induced by
on $\hat{\tau}^D_{X,g}$.
%, by the following lemma, making such a choice does not affect us very much.

\begin{lem}
Let $X$ be a differentiable manifold, and let $E$ be a differentiable vector bundle on $X$. Let $d \in \text{Diff}_1(E,E)$ be a differential operator, so that for each metric $g$ on $X$ and $h$ on $E$, $D = d + d^*$ is an essentially self-adjoint elliptic differential operator. Let $g_0,g_1$ be two metrics on $X$, and let $h_0,h_1$ be metrics on $E$. Denote $d+ d^*$ with respect to $(g_0,h_0),(g_1,h_1)$ by $D_0,D_1$ respectively. Then there is a unitary isomorphism $L^2(X,E; g_0,h_0)\to L^2(X,E; g_1,h_1)$ %(where the $L^2$-completions are with respect to $(g_0,h_0),(g_1,h_1)$, respectively.) 
that commutes with $\chi(D)$ up to locally compact operators. 
\end{lem}

\begin{proof}
Let $g_t= (1-t)g_0 + tg_1, h_t=(1-t)h_0 + t h_1$ for $0 \leq  t \leq 1$. Then both $g_t, h_t$ are metrics (and in case both $g_0,g_1$ are hermitian, then so are $g_t$ and etc.). Denote the Hilbert space of $L^2$-sections of $E$ with respect to the metric $g_t,h_t$ by $H_t$. Let $\nu_t:X \to \mathbb{R}^{\geq 0}$ be the "square root of the measure" given by the Radon Nikodym theorem so that $d\mu_t(Z)= \int _Z \nu_t^2 d\mu_0$ for each measurable subset $Z$ of $X$ and each $t$.
Let $S_t:E\to E$ be the square root of the positive definite map $E\xrightarrow{h_0}E^*\xrightarrow{h^*_t}E$. \footnote{Notice that by \cite[Page 150.]{lax2007linear} this exists and varries continuously.}, then $T_t(x)= \eta(x) S_t(x)$ acts fiberwise, hence it is pseudo-local. Also for $L^2$ sections $\eta, \zeta$ in $H_0$, 
\begin{small}
$$\langle T_t \eta , T_t \zeta \rangle_t = \int _X (h_t)(T_t \eta) (T_t \zeta) d\mu_t = \int_X \nu^2 h_t(S_t \eta)(S_t \zeta) d\mu_t = \int_X h_t(S_t^* S_t \eta)(\zeta) d\mu_0 = \int_X h_t h_t^* h_0(\eta)(\zeta) d\mu_0 = \langle \eta, \zeta \rangle_0. $$
\end{small} 
Therefore we have unitary maps $T_t:L^2(X,E)\to L^2(X,E)$ (where the $L^2$-completions are with respect to $(g_0,h_0),(g_t,h_t)$, respectively.). Consider the the path $t\mapsto T_t^* \chi(D_t) T_t$ from $\chi(D_0)$ to $T_1^* \chi(D_1) T_1$. Since $\chi^2-1\in C_0(\mathbb{R})$, therefore $T^*_t \frac{\chi(D_t)+1}{2} T_t\in (\mathfrak{D}/ \mathfrak{C})(\rho_0)$ is a self-adjoint projection up to locally compact operators. Hence by lemma \ref{projection has an actual representative lemma}, without loss of generality we can assume $T^*_t \frac{\chi(D_t)+1}{2} T_t\in \mathfrak{B}(L^2(X,E))$ (where the $L^2$-completion is with respect to $(g_0,h_0)$.) is a self-adjoint projection, and by \cite[4.1.8.]{higson2000analytic} this path of projections induces a unitary operator $W_1 :L^2(X,E;g_0, h_0)\to L^2(X,E;g_t,h_t)$ such that $W^*_1 (T^*_1 \frac{\chi(D_1)+1}{2} T_1)W_1 = \frac{\chi(D_0)+1}{2}$. Therefore, $W_1T_1$ is the unitary isomorphism that commutes with $\chi(D)$ up to locally compact operators.
\end{proof}

%To prove the proposition above, we will mention and prove a series of technical lemmas (which are probobly known to the experts). We include them with great details for readers who may not have a background in the topic.

\subsection{Proof of Proposition \ref{functoriality of tau1 prop}}

\begin{notation} \label{metric pair definition}
Let $X$ be a differentiable manifold, and let $E_1,E_2$ be differentiable vector bundles on $X$. Choose metrics $g$ on $X$ and $h_1,h_2$ on $E_1,E_2$. To shorten the notation, we say $(X,g; E_1,h_1; E_2,h_2)$ is a \emph{metric pair}.

Let $X$ be a complex manifold, and let $E_1,E_2$ be holomorphic vector bundles on $X$. Choose hermitian metrics $g$ on $X$ and $h_1,h_2$ on $E_1,E_2$, and let set $D_{E_1}= \bar{\partial}_{E_1} + \bar{\partial}^*_{E_1}$ and $D_{E_2}= \bar{\partial}_{E_2} + \bar{\partial}^*_{E_2}$ be the corresponding Dolbeault operators. To shorten the notation, we say $(X,g; E_1,h_1, D_{E_1}; E_2,h_2,D_{E_2})$ is a \emph{hermitian pair}.
\end{notation}

\begin{Def} \label{locally bounded def}
Let $(X,g; E_1,h_1; E_2,h_2)$ be a metric pair. We say an  operator $T$ (or a family of operators) is \emph{locally bounded} with respect to $(X,g; E_1,h_1; E_2,h_2)$, if for each relatively compact open subset $U$ of $X$, there exists an induced operator $T_U: L^2(U,E_1\vert_U)\to L^2(U,E_2\vert_U)$ so that $T_U$ is a bounded linear operator, and if for each pair of relatively compact open subsets $U_1,U_2$, we have $$\pi ^{U_1}_{U_1\cap U_2} T_{U_1} \iota^{U_1}_{U_1\cap U_2} = \pi^ {U_2}_{U_1\cap U_2} T_{U_2} \iota^ {U_2}_{U_1\cap U_2} $$ 
where in here $\pi ^U_V : L^2(U,E\vert_U)\to L^2(V,E\vert_V)$ is the projection defined by multiplication by the characteristic function of $U\cap V$, and $\iota^U_V: L^2(V,E\vert_V) \to L^2(U,E\vert_U)$ is extension by zero.

Beware that this definition is not exactly the same as more well-known definitions of local boundedness. Also
note that there does not need to be a uniform bound on $\Vert T_U \Vert$. However, in case there is a uniform bound on $T_U$ (say $M$), then we can "glue" them to obtain $T:L^2(X,E_1)\to L^2(X,E_2)$, by simply choosing a relatively compact open neighborhood $U$ of $x$, and setting $T(\zeta)(x) = T_U(\pi_U \zeta \pi^*_U)(x)$. This is independent of choice of $U$ and $\Vert T(\zeta)\Vert_2 \leq M \Vert \zeta \Vert_1$, as this holds for \emph{almost every} \footnote{Recall that evaluating $\zeta\in L^2(X,E_1)$ at a point $x\in X$ only makes sense up to subsets of measure zero in $X$.} point $x$.   
\end{Def}

\begin{ex}
Let $(X,g; E_1,h_1; E_2,h_2)$ be a metric pair (definition \ref{metric pair definition}), and let $\varphi:E_1\to E_2$ be a continuous bundle map. Then $\varphi$ is locally bounded (definition \ref{locally bounded def}).
\end{ex}

\begin{ex}\label{induced by identity notation - example}
Let $(X,g;E,h_1,E,h_2)$ be a metric pair (definition \ref{metric pair definition}), and let $L^2(X,E;h_i)= L^2(X,E;g,h_i)$ denote the space of $L^2$-sections of $E$ on $X$ with respect to the metric $h_i$ on $E$ (and $g$ on $X$). Then the identity map $Id:E\to E$ induces a locally bounded map (definition \ref{locally bounded def}) from $L^2(X,E;h_1)$ to $L^2(X,E;h_2)$ which we denote by $I(h_2,h_1)$ throughout this section.
\end{ex}

% \begin{proof} Let $\langle , \rangle_i$ denote the inner product on fibers of $E_i$, for $i=1,2$. Note that at each point $x\in X$, the ratio $\frac{\langle f_1(x), f_2(x) \rangle_1 }{\langle f_1(x), f_2(x) \rangle_2 }$ for continuous sections $f_1,f_2$ of $ \wedge^{0,*}T^*X \otimes E $, is bounded from both above; the bounds of course depend on $x$, but vary continuously (as long as neither $f_1,f_2$ are zero.). Hence for the relatively compact set $U$ there is a uniform bound $ M_U$ so that $\frac{\langle f_1(x), f_2(x) \rangle_1 }{\langle f_1(x), f_2(x) \rangle_2 } \leq M_U$ for each $x\in U$. Then it is easy to see that the linear map $\alpha$ has norm at most $M_U$ and hence induces a bounded map of the corresponding Hilbert spaces. \end{proof}

\begin{lem}
Let $(X,g; E,h_1, D_{E,1}; E,h_2,D_{E,2})$ be a hermitian pair (definition \ref{metric pair definition}). Then $D_{E,1}-D_{E,2}$ is locally bounded (definition \ref{locally bounded def}).
\end{lem}

\begin{proof}
Recall from definition \ref{hodge-star operator definition} that the metric $h_i$ can be considered as a linear map of bundles from $E$ to the dual bundle $E^*$, which by abuse of notation we denote with $h_i$ again. Let $h^*_i: E^*\to E$ denote the dual maps induced by $h_i$, let $\theta$ denote the composition $E\xrightarrow{h_1} E^* \xrightarrow{h^*_2} E$, and let $\vartheta^*$ %\footnote{This may not be a good notation, as $\theta,\vartheta^*$ are \emph{not} dual to each other.} 
denote the composition $E^*\xrightarrow{h^*_1} E \xrightarrow{h_2} E^*$.
%We would like to compute $D_{E,1} - D_{E,2}$ by brute force. 

Consider $f\otimes e \in \mathscr{C}^{\infty} (X,\wedge^{0,*}T^*X \otimes E)$, we have: 
\begin{align*}
D_{E,1}(f\otimes e) - D_{E,2}(f\otimes e) 
& = \left( \bar{\partial}_E + (\star \otimes h^*_1) \bar{\partial} (\star \otimes h_1) \right)(f\otimes e) -  \left( \bar{\partial}_E + (\star \otimes h^*_2) \bar{\partial} (\star \otimes h_2) \right) (f\otimes e)\\
&=  (\star \otimes h^*_1) \bar{\partial} (\star f \otimes h_1(e)) - (\star \otimes h^*_2) \bar{\partial} (\star f \otimes h_2(e)) \\
&=  (\star \otimes h^*_1) \left(\bar{\partial} (\star f) \otimes h_1(e) + \star f \otimes \bar{\partial} h_1(e)\right) - (\star \otimes h^*_2) \left(\bar{\partial} (\star f) \otimes h_2(e) + \star f \otimes \bar{\partial} h_2(e)\right)\\
&= \left( \star \bar{\partial} \star f \otimes e + \star \star f \otimes h^*_1 \bar{\partial} h_1(e)\right) -
\left( \star \bar{\partial} \star f \otimes e + \star \star f \otimes h^*_2 \bar{\partial} h_2(e)\right) \\
&= \star\star f \otimes ( h^*_1 \bar{\partial} (e^*) -  h^*_2 \bar{\partial} (\vartheta^* e^*) )\\
&= \star\star f \otimes \left( h^*_1 \bar{\partial} (e^*) - h^*_2 (\vartheta^* \bar{\partial}(e^*)+ e^* \bar{\partial}(\vartheta^*)  \right) \\
&=  \star\star f \otimes \theta(e) \bar{\partial}(\vartheta^*)
\end{align*}
where in here, $e^*=h_1(e)$. The term above does not have any differentials of $f\otimes e$; recall $\star \star$ is $\pm 1$, and $\Vert \theta \Vert_i, \Vert \vartheta^* \Vert_i $ vary continuously with respect to $x\in X$, and $i=1,2$, hence the term $\theta(e) \bar{\partial}(\vartheta^*)$ is bounded with respect to both norms on the relatively compact set $U$.
\end{proof}

\begin{lem}\label{different metrics on v.b. give the same operator in Paschke cat - Lemma}
Let $(X,g; E,h_1; E,h_2)$ be a metric pair (definition \ref{metric pair definition}), let 
and let $D_i\in \text{Diff}_1(E,E)$ be an essentially self-adjoint differential operators with respect to the metric $h_i$ for $i=1,2$, so that $D_1-D_2$ is a locally bounded operator (definition \ref{locally bounded def}). Let $I(h_2,h_1)$ denote the locally bounded map induced by the identity map of $E$ (example \ref{induced by identity notation - example}). % with respect to the metrics $h_1,h_2$. (i.e. $I(h_2,h_1)_U:L^2(U,E\vert_U )\to L^2(U,E\vert_U)$ where the first $L^2$-space is with respect to the metric $h_1$ and the second is with respect to the metric $h_2$.) 
Then for each relatively compact open subset $U$ of $X$, we have 
$$\pi_1 \chi(D_1) \iota_1 I(h_1,h_2)_U = I(h_1,h_2)_U \pi_2 \chi(D_2) \iota_2 $$ in the Paschke category $(\mathfrak{D}/\mathfrak{C})_{C_0(U)}$, where in here, $\pi_i:L^2(X,E;h_i)\to L^2(U,E\vert_U;h_i )$ is the projection and $\iota_i$ is extension by zero, for $i=1,2$. 
\end{lem}

This proof closely follows that of \cite[10.9.5.]{higson2000analytic}.\footnote{We can not directly apply this result here, even though they look similar; the problem is that in \cite[10.9.5.]{higson2000analytic} it is required for both $D_1$, $D_2$ to be essentially self-adjoint with respect to the same given inner product, which is not the case here.}

\begin{proof}
Similar to \cite[10.3.5.]{higson2000analytic} we argue that if $u,v$ be compactly supported smooth sections of $\wedge^{0,*}T^*X\otimes E$, and $\phi$ is a \emph{Schwartz function}, then since $\phi(x)=\frac{1}{2\pi} \int e^{\sqrt{-1}sx}\hat{\phi}(s)ds$, then we can pair $\hat{\phi}$ with the smooth function $s\mapsto \langle (I(h_2,h_1) e^{\sqrt{-1}sD_1}I(h_1,h_2)) u,  v \rangle_2 =  \langle  e^{\sqrt{-1}sD_1}u,  v \rangle_2$ to obtain
\begin{equation*}
\langle  \phi(D_1) u,v\rangle_2 =
\langle I(h_2,h_1) \phi(D_1) I(h_1,h_2) u,v\rangle_2 =
 \frac{1}{2\pi} \int \langle I(h_2,h_1) e^{\sqrt{-1}s D_1} I(h_1,h_2) u,  v \rangle_2 \hat{\phi}(s) ds,
\end{equation*}
and then use the rest of the argument in \cite[10.3.5.]{higson2000analytic}, to generalize this for any bounded Borel function whose Fourier transform is compactly supported.

Let $\phi,u,v$ be as above (with the extra assumption that $s\hat{\phi}(s)$ is a smooth function, also note that $D_1,D_2$ both share the invariant domain of smooth compactly supported functions.), then we have
\begin{equation*}
\langle (I(h_2,h_1) \phi(D_1) I(h_1,h_2) - \phi(D_2)) u,v\rangle_2 = \frac{1}{2\pi} \int \langle (I(h_2,h_1) e^{\sqrt{-1}sD_1} I(h_1,h_2) - e^{\sqrt{-1}sD_2} ) u, v \rangle_2 \hat{\phi}(s) ds.
\end{equation*}
By fundamental theorem of calculus we know that
\begin{equation*}
\begin{array}{c}
 \langle (I(h_2,h_1) e^{\sqrt{-1}sD_1} I(h_1,h_2) - e^{\sqrt{-1}sD_2} ) u, v \rangle_2 =
\langle ( e^{\sqrt{-1}sD_1} - e^{\sqrt{-1}sD_2} ) u, v \rangle_2 = \\
 \sqrt{-1} \int_0^s \langle (I(h_2,h_1) e^{\sqrt{-1}t D_1} I(h_1,h_2)(D_1-D_2) e^{\sqrt{-1}(s-t) D_2}) u ,v
\rangle_2 ,
\end{array}
\end{equation*}
and by repeating the argument in \cite[10.3.6, 10.3.7.]{higson2000analytic} we obtain that there exists a constant $C_{\phi}< \infty$ (which only depends on $\phi$) so that $\Vert I(h_2,h_1) \phi(D_1) I(h_1,h_2) - \phi(D_2)\Vert_2 \leq C_{\phi} \Vert D_1 - D_2 \Vert_2$.  

Now, let $\phi$ be a \emph{normalizing function} (i.e. $\phi -\chi \in C_0(\mathbb{R})$.) that satisfies the conditions above, and let $\phi_{\epsilon}(x)= \phi(\epsilon x)$. Then $\phi_{\epsilon}$ is also a normalizing function, and hence $\phi_{\epsilon}(D_i) - \chi(D_i)$ is a locally compact operator for any $\epsilon > 0$. But as $\epsilon \to 0$, we get $$\Vert I(h_2,h_1) \phi_{\epsilon}(D_1) I(h_1,h_2) - \phi_{\epsilon}(D_2)\Vert_2 = \Vert I(h_2,h_1) \phi(\epsilon D_1) I(h_1,h_2) - \phi(\epsilon D_2) \Vert _2 \leq \epsilon C_{\phi} \Vert D_1 - D_2 \Vert_2 \to 0.$$ In other words, there are elements of equivalency class of locally compact operators equivalent to $I(h_2,h_1) \chi(D_1) I(h_1,h_2)$ and $\chi(D_2)$ respectively, which get arbitrarily close. But these are linear subspaces of pseudo-local operators, hence these subspaces have to be the same, i.e. $I(h_2,h_1) \chi(D_1) I(h_1,h_2) - \chi(D_2)$ is locally compact. This finishes the proof.

\end{proof}

\begin{cor} \label{different metrics on v.b. give the same dolbeault operator in paschke cat - cor}
Let $(X,g; E,h_1, D_{E,1}; E,h_2,D_{E,2})$ be a hermitian pair (definition \ref{metric pair definition}), let $I(h_2,h_1)$ be the locally bounded map induced by the identity map of $E$ (example \ref{induced by identity notation - example}). Let $U$ be a relatively compact open subset of $X$, and let $\pi_i:L^2(X, \wedge^{0,*}T^*X\otimes E; h_i)\to L^2(U, \wedge^{0,*}T^*X\otimes E \vert_{U}; h_i)$ be the projection and let $\iota_i$ be its adjoint. Then $\pi_1 \chi_{D_{E_1}} \iota_1 I(h_1,h_2)_U = I(h_1,h_2)_U \pi_2 \chi_{D_{E_2}} \iota_2 $ in the Paschke category $(\mathfrak{D}/\mathfrak{C}) _{C_0(U)}$.
\end{cor}

\begin{Def} \label{preserves the metrics definition}
Let $X$ be a differentiable manifold, and let $E_1,E_2$ be differentiable vector bundles on $X$. Let $\alpha:E_1\to E_2$ be a smooth bundle map. Choose metrics $g, h_1,h_2$ on $X,E_1,E_2$ respectively. We say $\alpha$ \emph{preserves the metrics}, if the dual map of bundles $\beta:E^*_2\to E_1^*$ on the dual vector bundles (defined by $\beta(e^*_2)(e_1) = e_2^*(\alpha(e_1))$.) makes the diagram below commute.
\begin{center}
\begin{tikzcd}
E_1 \ar[r,"\alpha"] \ar[d, "h_1"]
& E_2 \ar[d, "h_2"] \\
E_1^* 
& E_2^*. \ar[l, "\beta"]
\end{tikzcd}
\end{center} 

\end{Def}

\begin{lem} \label{preserving metrics preserve functional calculus lemma}
Let $(X,g; E_1,h_1; E_2,h_2)$ be a metric pair (definition \ref{metric pair definition}), and let $\alpha:E_1\to E_2$ be a smooth isomorphism of vector bundles that preserves the metrics (definition \ref{preserves the metrics definition}). Let $D\in \text{Diff}_1(E_1,E_1)$ be an essentially self-adjoint differential operator of order one. Then $\chi(D) = \alpha^{-1} \chi(\alpha D \alpha^{-1}) \alpha $.  
\end{lem} 

\begin{proof}
Since $\alpha$ preserves the metric on each fiber, then the induced map $\alpha:L^2(X,E_1) \to L^2(X,E_2)$ is a unitary map, i.e. $\alpha^{-1} = \alpha^*$. Since $D$ is symmetric, then $\alpha D \alpha^{-1}= \alpha D \alpha^*$ is also symmetric. Also if $x\in \text{Domain}(\alpha D^* \alpha^{-1})\subset L^2(X,E_2)$, then there exists a constant $M$ so that for each $y\in \text{Domain}(\alpha D \alpha^{-1})\subset L^2(X,E_2)$, we have $\vert \langle x, \alpha D^* \alpha^{-1} y \rangle \vert \leq M \Vert y \Vert$ \cite[1.8.2.]{higson2000analytic}. But if $x' = \alpha^{-1}x, y'= \alpha^{-1}y \in L^2(X,E_1)$, then this is equivalent to saying that $\vert \langle x',  D^* y' \rangle \vert \leq M \Vert y \Vert$, i.e. $x' \in \text{Domain}(D^*)$. Since $D$ is essentially self-adjoint, then $x'\in \text{Domain}(D)$, hence $x\in \text{Domain}(\alpha D \alpha^{-1})$, i.e. $\alpha D \alpha^{-1}$ is also essentially self-adjoint. Hence $\chi(\alpha D \alpha^{-1})\in \mathfrak{B}(L^2(X,E_2))$ is well-defined.

Assume that the Fourier transform of the bounded Borel function $\phi$ is compactly supported, then for small values of $s>0$, we have $e^{\sqrt{-1} s\alpha D \alpha^{-1}} = \alpha e^{\sqrt{-1} s D} \alpha^{-1}$. Hence by \cite[10.3.5.]{higson2000analytic} it is easy to argue that $\phi(\alpha D \alpha^{-1}) = \alpha \phi(D) \alpha^{-1}$. Now if $\phi$ is a normalizing function, then $\phi(D) - \chi(D), \phi(\alpha D \alpha^{-1}) - \chi(\alpha D \alpha^{-1})$ are locally compact.

\end{proof}

\begin{lem} \label{preserving metrics preserves dolbeault operator lem}
Let $(X,g; E_1,h_1, D_{E_1}; E_2,h_2,D_{E_2})$ be a hermitian pair (definition \ref{metric pair definition}), and let $\alpha:E_1\to E_2$ be a smooth isomorphism of vector bundles that preserves the metrics. Then $\alpha D_{E_1} \alpha^{-1} - D_{E_2}$ is locally bounded (definition \ref{locally bounded def}).	
\end{lem}

\begin{proof}
Recall from definition \ref{hodge-star operator definition} that the metrics induce conjugate linear smooth bundle isomorphisms $h_i:E_i\to E_i^*$ to the dual bundle, for $i=1,2$, and $h^*_i:E_i^*\to E_i$ is the inverse. Let $\beta:E_2^*\to E_1^*$ be the map of bundles dual to $\alpha$.
Since $\alpha$ is a smooth isomorphism of vector bundles, then $\alpha \bar{\partial}_{E_1}  - \bar{\partial}_{E_2} \alpha $ is locally bounded. Therefore by \ref{hodge-star operator and adjoint dolbeault operator - equation}, to prove the lemma it suffices to show that the term below is locally compact, where $f\otimes e_2$ is a smooth section of $\wedge^{0,*} T^*X\otimes E_2$, and $e_1= \alpha^{-1} e_2, e_2^* = h_2(e_2)$ 
% Since $\alpha$ is holomorphic, then $\beta$ is also holomorphic, and they commute with $\bar{\partial}$. Hence by \ref{hodge-star operator and adjoint dolbeault operator - equation}, and smooth section $f\otimes e_2$ of $\wedge^{0,*} T^*X\otimes E_2$ we have: (\alpha D_{E_1} \alpha^{-1} - D_{E_2}) (f\otimes e_2) &=
\begin{small}
\begin{align*}
 (\alpha \bar{\star}_{E_1^*} \bar{\partial} \bar{\star}_{E_1} \alpha^{-1} - \bar{\star}_{E_2^*} \bar{\partial} \bar{\star}_{E_2} ) (f\otimes e_2) 
&= \alpha \bar{\star}_{E_1^*} \bar{\partial} (\bar{\star}f \otimes h_1(e_1) ) - \bar{\star}_{E_2^*} \bar{\partial} (\bar{\star} f \otimes h_2(e_2)) \\
&= \alpha \bar{\star}_{E_1^*} \left( \bar{\partial}(\bar{\star}f) \otimes h_1(e_1) + \bar{\star}f \otimes \bar{\partial} h_1(e_1) \right) - \bar{\star}_{E_2^*} \left( \bar{\partial}(\bar{\star}f) \otimes h_2(e_2) + \bar{\star}f \otimes \bar{\partial} h_2(e_2) \right) \\
&= \left( (\bar{\star} \bar{\partial}\bar{\star})f \otimes e_2 + \bar{\star}\bar{\star}f \otimes \alpha h_1^* \bar{\partial}  h_1(e_1)\right) - \left( ( \bar{\star}\bar{\partial}\bar{\star})f \otimes e_2 + \bar{\star}\bar{\star}f \otimes h^*_2 \bar{\partial} h_2(e_2) \right) \\
&= \bar{\star}\bar{\star}f \otimes (\alpha h_1^* \bar{\partial}  h_1(e_1) -  h^*_2 \bar{\partial} h_2(e_2)) \\
&= \bar{\star}\bar{\star}f \otimes (h_2^* \beta^{-1} \bar{\partial} \beta (e^*_2) - h_2^* \bar{\partial}(e_2^*) )
\end{align*}
\end{small}
% where in here $e_1= \alpha^{-1} e_2, e_2^* = h_2(e_2)$. \\ &= \bar{\star}\bar{\star}f \otimes (h_2^* \beta ^{-1} \beta \bar{\partial}(e_2^*) - h_2^* \bar{\partial} (e_2^*)) \\ &= 0
But $\beta$ is also a smooth isomorphism of vector bundles hence $\bar{\partial} \beta - \beta \bar{\partial}$ is locally bounded. 
% \bar{\star}\bar{\star}f \otimes (h_2^* \beta \bar{\partial}(\beta^{-1})(e^*_2) + h_2^* \beta \beta^{-1} \bar{\partial}(e_2^*) - h_2^* \bar{\partial}(e_2^*)) \\ &=\bar{\star}\bar{\star}f \otimes h_2^* \beta \bar{\partial}(\beta^{-1})(e^*_2)
\end{proof}

\begin{cor}\label{s.e.s. dolbeault operator locally bounded. -cor}
Let $0\to E_1\xrightarrow{\varphi_1}E_2\xrightarrow{\varphi_2} E_3\to 0$ be a short exact sequence of holomorphic vector bundles on the complex manifold $X$. Choose a hermitian metric $g$ on $X$, and $h_1$ on $E_1$. Then we get an induced hermitian metric on the subbundle $\varphi_1(E_1)$ of $E_2$. Extend this metric to hermitian metric $h_2$ on all of $E_2$. Then %by \cite[1.4.11.]{atiyah1967k} 
there exists a smooth map of bundles $\sigma_2:E_3\to E_2$ which is an isomorphism from $E_3$ to the orthogonal complement of $\varphi_1(E_1)$ in $E_2$. Let $h_3$ be the hermitian metric induced by this isomorhpism. 

Hence $(X,g;E_1\oplus E_3, h_1\oplus h_3, D_{E_1}\oplus D_{E_3}; E_2,h_2, D_{E_2})$ is a hermitian pair (definition \ref{metric pair definition}), and we have a smooth isomorphism $(\varphi_1,\sigma_2):E_1\oplus E_3\to E_2$. By definition of the metrics, it is easy to check that this isomorphism preserves metrics. % Note that if $\sigma_2$ was holomorphic, then by lemma above, we would get that $D_{E_1}\oplus D_{E_3} - (\varphi_1,\sigma_2) D_{E_3} (\varphi_1,\sigma_2)^{-1}=0$. However, the situation isn't so bad: 
Therefore as a corollary of \ref{preserving metrics preserves dolbeault operator lem},  
\begin{equation}\label{s.e.s. dolbeault operator locally bounded. -equation}
D_{E_1}\oplus D_{E_3} - (\varphi_1,\sigma_2)^{-1} D_{E_2} (\varphi_1,\sigma_2)
\end{equation}
is locally bounded.
\end{cor}

\begin{cor} \label{isomorphism of bundle gives dolbeault op iso in the paschke cat - cor}
Let $(X,g; E_1,h_1,D_{E_1}; E_2,h_2,D_{E_2})$ be a hermitian pair (definition \ref{metric pair definition}), and let $\alpha : E_1\to E_2$ be a smooth isomorphism of vector bundles on $X$. Let $U$ be a relatively compact open subset of $X$, let $\pi_i:L^2(X,\wedge^{0,*} T^*X\otimes E_i) \to L^2(U,\wedge^{0,*} T^*X\otimes E_i \vert_U)$ be the projection and let $\iota_i$ be its adjoint. Then
\begin{equation*}
\alpha _U \pi_1 \chi(D_{E_1}) \iota_1 = \pi_2 \chi(D_{E_2}) \iota_2 \alpha_U
\end{equation*}
in the Paschke category $(\mathfrak{D}/ \mathfrak{C})_{C_0(U)}$, where by abuse of notation, we are denoting the map induced by $\alpha_U$ from $L^2(U,\wedge^{0,*} T^*X\otimes E_1 \vert_U) \to L^2(U,\wedge^{0,*} T^*X\otimes E_2 \vert_U)$ by $\alpha_U$ as well.
\end{cor}

\begin{proof}
Consider the hermitian metric $h'_2$ on $E_2$ (defined through the diagram in definition \ref{preserves the metrics definition}.) so that the bundle isomorphism $\alpha:E_1\to E_2$ preserves the metrics. Let $D'_{E_2}= \bar{\partial}_{E_2} + \bar{\partial}_{E_2}^*$ be the Dolbeault operator with respect to $h'_2$, let $\pi'_2:L^2(X,\wedge^{0,*}T^*X\otimes E_2;h'_2)\to L^2(U,\wedge^{0,*}T^*X\otimes E_2\vert_U;h'_2)$ be the projection, and let $\iota'_2$ be its adjoint. 
%, and let $\gamma$ be a smooth compactly supported function that is equal to one on $U$. Also $D_i = \gamma D_{E_i} \gamma$ for $i=1,2$ and $D'_2= \gamma D'_{E_2} \gamma$ are compactly supported symmetric functions, hence are essentially self-adjoint, and $\chi(D_i)$ agrees with $\chi_{D_{E_i}}$ and $\chi(D'_2)$ agrees with $\chi(D'_{E_2})$ in the Paschke category $(\mathfrak{D}/ \mathfrak{C})_{C_0(U)}$. 
Denote the map induced by $\alpha$ from $L^2(U,\wedge^{0,*} T^*X\otimes E_1 \vert_U)$ to $L^2(U,\wedge^{0,*} T^*X\otimes E_2 \vert_U;h'_2)$ by $\alpha'_U$, and let $I(h_2,h'_2)$ denote the locally bounded map induced by the identity of $E_2$ (example \ref{induced by identity notation - example}). 
%from $L^2(U,\wedge^{0,*} T^*X\otimes E_2 \vert_U)$ (with the metric induced by $h'_2$) to $L^2(U,\wedge^{0,*} T^*X\otimes E_2 \vert_U)$ (with the metric induced by $h_2$) by $I(h_2,h'_2)_U$. 
Therefore $\alpha_U= I(h_2,h'_2)_U \alpha'_U$ and: % Also by lemma \ref{preserving metrics preserve functional calculus lemma}, we get $\alpha' \chi(D_{E_1}) = \chi(\alpha D_{E_1} \alpha ^{-1}) \alpha' $ in the Paschke category $(\mathfrak{D}/ \mathfrak{C})_{C_0(X)}$. By lemma \ref{preserving metrics preserves dolbeault operator lem} $\alpha D_{E_1} \alpha^{-1} - D'_{E_2} $ is locally bounded, and by lemma \ref{different metrics on v.b. give the same operator in Paschke cat - Lemma}, $\pi'_2 \chi(D'_{E_2}) \pi'^*_2 \iota_U = \iota_U \pi_2 \chi(D_{E_2}) \pi_2^*$ in the Paschke category $(\mathfrak{D}/ \mathfrak{C})_{C_0(U)}$. Therefore
\begin{align*}
\alpha _U \pi_1 \chi(D_{E_1}) \iota_1 
& = I(h_2,h'_2)_U \alpha'_U \pi_1 \chi(D_{E_1}) \iota_1  && \\
& = I(h_2,h'_2)_U \pi'_2 \alpha' \chi(D_{E_1}) \iota_1 &&  \\
& = I(h_2,h'_2)_U \pi'_2 \chi(\alpha D_{E_1} \alpha^{-1}) \alpha' \iota_1 && \text{By lemma \ref{preserving metrics preserve functional calculus lemma}} \\
& = I(h_2,h'_2)_U \pi'_2 \chi(D'_{E_2}) \alpha' \iota_1 && \text{By lemma \ref{preserving metrics preserves dolbeault operator lem} and \cite[10.9.5.]{higson2000analytic} } \\ 
& = I(h_2,h'_2)_U \pi'_2 \chi(D'_{E_2}) \iota'_2 \alpha'_U && \\
& = \pi_2 \chi(D_{E_2}) \iota_2 I(h_2,h'_2)_U \alpha'_U && \text{By corollary \ref{different metrics on v.b. give the same dolbeault operator in paschke cat - cor} } \\
& = \pi_2 \chi(D_{E_2}) \iota_2 \alpha_U. &&  
\end{align*}
\end{proof}

\begin{rmk}\label{change metric on manifold changes complex in pashcke cat -remark}
One may wonder if we can change the metric $g$ on $X$ in the corollary above as well. Consider the case where $E_1=E_2$ is the trivial bundle of rank one, $\alpha$ is the identity map, and $h_1=h_2$. When $g_1,g_2$ are two different hermitian metrics on $X$, the symbols of $\bar{\partial} +\bar{\partial}^*_{g_1}, \bar{\partial} + \bar{\partial}^*_{g_2}$ are not equal to each other, and there is no indication on why after applying functional calculus, we should get the same operator in the Paschke category. However for a relatively compact open subset $U$, the operator induced by identity $I(g_2,g_1)_U : L^2(U,E_1\vert_U;g_1,h_1) \to L^2(U,E_2\vert_U;g_2,h_2)$ is the identity on the underlying vector spaces (although these Hilbert spaces are different as they have different inner products.), hence $I(g_2,g_1)$ should not induce a map between the chain complexes $\hat{\tau}^D_{X,g_1}(E_1,h_1), \hat{\tau}^D_{X,g_2}(E_2,h_2) $.
\end{rmk}

\begin{proof}[Proof of proposition \ref{functoriality of tau1 prop}]
We have already defined $\hat{\tau}^D_{X,g}$ on the objects of the category, and showed that $\hat{\tau}^D_{X,g}(E,h)$ is an exact sequence in the Paschke category $(\mathfrak{D}/ \mathfrak{C})_{C_0(X)}$. %, if $h$ is a hermitian metric of bounded geometry. 
We need to show functoriality and exactness. Before going further, let us fix some notation.

Let $\varphi:(E_1,h_1)\to (E_2,h_2)$ be a morphism in $\mathcal{P}_{\text{m, b}}(X)$. Choose a good cover (definition \ref{good cover def}) $\{U_j\}_{j}$ so that for $i=1,2$ and for each $j$, there exists an open subset $V_j$ of $X$ that contains closure of $U_j$, and that $E_i \vert_{V_j}$ and $\wedge^{0,*} T^*X\vert_{V_j}$ is isomorphic to the trivial bundle on $V_j$. In other words, there exists holomorphic isomorphisms of bundles $\alpha_j:\wedge^{0,*}T^*X \otimes E_1\vert_{V_j} \to V_j \times \mathbb{C}^k$ and $\beta_j: \wedge^{0,*}T^*X \otimes E_2\vert_{V_j} \to V_j \times \mathbb{C}^m$ where $k,m$ are ranks of the corresponding bundles. Then $\psi_j= \beta_j \varphi\vert_{V_j} \alpha_j^{-1}:V_j \to M_{m,k}(\mathbb{C})$ is a holomorphic matrix valued function. let $D_{i,j} = \gamma_j D_{E_i} \gamma_j$ for $i=1,2$.
Let $\{ \lambda_j \}_{j}$ be a partition of unity subbordinate to the cover $\{U_j\}_{j}$, and let $\gamma_j$ be smooth cutoff functions which are equal to one on $U_j$. % and equal to zero outside of $V_j$.
Also, let $\pi_{i,j}:L^2(X,\wedge^{0,*}T^*X\otimes E_i)\to L^2(U_j,\wedge^{0,*}T^*X\otimes E_i \vert_{U_j}) $ be the projection and let $\iota_{i,j}$ be its adjoint.
%let $D_{i,j} = \gamma_j D_{E_i} \gamma_j$ for $i=1,2$. $D_{i,j}$ is a symmetric compactly supported differential operator, hence essentially self-adjoint. Therefore we can define $\chi(D_{i,j})\in \mathfrak{B}(L^2(X,E_i))$, and $\pi_{i,j} \chi_{D_{E_i}}  \pi_{i,j}^* = \pi_{i,j} \chi(D_{i,j})  \pi_{i,j}^* \in \mathfrak{B}(L^2(U_j, \wedge^{0,*}T^*X\otimes E_i\vert_{U_j} ))$ in the Paschke category $(\mathfrak{D}/ \mathfrak{C})_{C_0(U_j)}$. 
For $n\in \mathbb{Z}^{>0}$, let $D^n$ denote the Dolbeault operator corresponding to the trivial rank $n$ bundle, let $\pi^n_j:L^2(X,X\times \mathbb{C}^n)\to L^2(U_j, U_j\times \mathbb{C}^n)$ be the projection, and let $\iota^n_j$ be its adjoint. Then by corollary \ref{isomorphism of bundle gives dolbeault op iso in the paschke cat - cor}, for relatively compact subset $U_j$ of $V_j$, we get that
\begin{align*}
\alpha_{j,U_j} \pi_{1,j} \chi_{D_{E_1}}  \iota_{1,j} &=   \pi^k_j \chi_{D^k_j} \iota^{k}_j  \alpha_{j,U_j} \\ \beta_{j,U_j} \pi_{2,j} \chi_{D_{E_2}}  \iota_{2,j} &=   \pi^m_j \chi_{D^m_j} \iota^{m}_j \beta_{j,U_j} 
\end{align*}
in the Paschke category $(\mathfrak{D}/ \mathfrak{C})_{C_0(U_j)}$.
Now, let $f\in C_0(X)$ be compactly supported. Then there are only finitely many of the $V_j$'s that intersect support of $f$, i.e. the sums below are all finite.
\begin{small}
\begin{align*}
(\varphi \chi_{D_{E_1}} -  \chi_{D_{E_2}} \varphi ) \rho(f) 
& = (  \sum_{j} \lambda_j^{1/2} \varphi \chi(D_{1,j}) \lambda_j^{1/2} 
&& -  \sum_j \lambda_j^{1/2} \chi(D_{2,j}) \varphi \lambda_j^{1/2}   ) \rho(f) \\
& = (\sum_{j} \lambda_j^{1/2} \pi_{1,j} \varphi \iota_{1,j} \pi_{1,j} \chi(D_{1,j}) \iota_{1,j} \lambda_j^{1/2} 
&& -  \sum_j \lambda_j^{1/2}  \pi_{2,j} \chi(D_{2,j}) \iota_{2,j} \varphi \pi_{2,j}  \lambda_j^{1/2} ) \rho(f) \\
& = (\sum_{j} \lambda_j^{1/2}  \beta_{j,U_j}^{-1}  \psi_j \alpha_{j,U_j} \pi_{1,j} \chi(D_{1,j}) \iota_{1,j} \lambda_j^{1/2} 
&& -  \sum_j \lambda_j^{1/2}  \pi_{2,j} \chi(D_{2,j}) \iota_{2,j} \beta^{-1}_{j,U_j} \psi_j \alpha_{j,U_j} \pi_{2,j} \lambda_j^{1/2} ) \rho(f) \\
& = (\sum_{j} \lambda_j^{1/2}  \beta_{j,U_j}^{-1}  \psi_j \pi^k_j \chi_{D^k_j} \iota^{k}_j  \alpha_{j,U_j} \lambda_j^{1/2} 
&& -  \sum_j \lambda_j^{1/2} \iota_{2,j} \beta^{-1}_{j,U_j} \pi^m_j \chi_{D^m_j} \iota^{m}_j \psi_j \alpha_{j,U_j} \pi_{2,j} \lambda_j^{1/2} )  \rho(f) \\
&= (\sum_{j} \iota_{2,j} \beta_{j,U_j}^{-1} \pi^m_j \chi_{D^m_j}  \psi_j \iota^{k}_j  \alpha_{j,U_j} \lambda_j 
&& -  \sum_j  \iota_{2,j} \beta^{-1}_{j,U_j} \pi^m_j \chi_{D^m_j} \iota^{m}_j \psi_j \alpha_{j,U_j} \pi_{2,j} \lambda_j )  \rho(f).
\end{align*}
\end{small}
Where the last equality holds because, in the first sum $\chi_{D^k_j}$ is pseudo-local, and hence up to compact operators, commutes with multiplication by the matrix valued continuous function $\lambda_j^{1/2} \psi_j$ that vanishes at infinity, therefore $(\lambda^{1/2}_j \psi_j) \chi_{D^k_j} - \chi_{D^m_j} (\lambda_j^{1/2} \psi_j) $ is compact for each $j$; and in the second sum $\lambda_j^{1/2} \chi_{D_j^m} - \chi_{D_j^m} \lambda_j^{1/2}$ is also compact for each $j$, and both sums are finite.

We conclude the first part of the proof by noting that $$ \psi_j \iota^{k}_j  \alpha_{j,U_j} \lambda_j  = \iota^{m}_j \psi_j \alpha_j \pi_{2,j} \lambda_j,$$
hence each term in the sum above is zero, and therefore $\varphi$ induces a map from $\hat{\tau}^D_{X,g}(E_1,h_1)$ to $\hat{\tau}^D_{X,g}(E_2,h_2)$ in the category $Ch'(\mathfrak{D}/ \mathfrak{C})_{C_0(X)}$.

It is straightforward to check that $\hat{\tau}^D_{X,g}(\varphi_1 \circ \varphi_2) = \hat{\tau}^D_{X,g}(\varphi_1) \circ \hat{\tau}^D_{X,g}(\varphi_2)$. This shows that $\hat{\tau}^D_{X,g}$ is a functor.

\begin{rmk} \label{functoriality tau^D continuous is enough -rmk}
Note that the condition on $\varphi:L^2(X, \wedge^{0,*}T^*X \otimes E_1)\to L^2(X, \wedge^{0,*}T^*X \otimes E_2)$ being bounded is not used in the proof of why $\hat{\tau}^D_{X,g}$ is functorial. Also holomorphicity of $\varphi$ was not needed in the argument above, we only needed continuity to show that multiplication by $\lambda_j \psi_j$ commutes with $\chi(D)$ modulo compact operators.
\end{rmk}

Now, to prove that $\hat{\tau}^D_{X,g}$ is an exact functor, let
\begin{center}
\begin{tikzcd}
0 \ar[r ] & (E_1,h_1) \ar[r, "\varphi_1"] & (E_2,h_2) \ar[r, "\varphi_2"] \ar[l, bend left=20, "\sigma_1"] & (E_3,h_3) \ar[l, bend left=20, "\sigma_2" ] \ar[r] & 0 
\end{tikzcd}
\end{center}
be an exact sequence in $\mathcal{P}_{\text{m,b}}(X)$. Then by definition of exactness in this category, there exists smooth sections $\sigma_2:E_3\to E_2, \sigma_1:E_2\to E_1$ so that $\sigma_1 \varphi_1=Id_{E_1}, \varphi_2 \sigma_2 = Id_{E_3} $, and similar to $\varphi_i$, $\sigma_i$ also induce a bounded map of Hilbert spaces $L^2(X,\wedge^{0,*} T^*X\otimes E_{i+1}) \to L^2(X,\wedge^{0,*} T^*X\otimes E_i)$, for $i=1,2$.

Let $h'_2$ be the hermitian metric on $E_2$ induced by $h_1,h_3$, i.e. $$h'_2 = \sigma^*_1 h_1\sigma_1 + \varphi^*_2 h_3\varphi_2:E_2\to E_2^*$$ 
where in here, $\sigma_1^*:E^*_1\to E_2^*$ and $\varphi_2^*:E^*_3\to E^*_2$ are the dual maps to $\sigma_1, \varphi_2$ respectively. Then the subbundles $\varphi_1(E_1),\sigma_2(E_3)$ of $E_2$ are orthogonal with respect to $h'_2$, and the induced metrics on these subbundles match with the metrics $h_1,h_3$ respectively, i.e. the isomorphism between $E_1\oplus E_3$ and $E_2$ preserves the metric (definition \ref{preserves the metrics definition}). Hence by corollary \ref{s.e.s. dolbeault operator locally bounded. -cor} $(\sigma_1,\varphi_2) D'_{E_2} (\varphi_1, \sigma_2) - D_{E_1} \oplus D_{E_3}$ is locally bounded, where $D'_{E_2}$ is the Dolbeault operator on $E_2$ with respect to the metric $h'_2$. Therefore by lemma \ref{preserving metrics preserve functional calculus lemma}, we get that 
$$ \chi_{D_{E_1}} \oplus \chi_{D_{E_3}} = \chi_{D_{E_1} \oplus D_{E_3}} = (\varphi_1, \sigma_2) \chi_{(\sigma_1,\varphi_2) D'_{E_2} (\varphi_1, \sigma_2) } (\sigma_1,\varphi_2). $$
By corollary \ref{different metrics on v.b. give the same dolbeault operator in paschke cat - cor}, for any relatively compact open subset $U$ of $X$, we have $I(h'_2,h_2)_U  \pi_2 \chi_{D_{E_2}} \iota_2 = \pi'_2 \chi_{D'_{E_2}} \iota'_2 I(h'_2,h_2)_U $, where $\pi_2, \pi'_2$ are the projections $L^2(X,\wedge^{0,*}T^*X\otimes E_2)\to L^2(U,\wedge^{0,*}T^*X\otimes E_2 \vert_U)$ with respect to the metrics $h_2,h'_2$ and $\iota_2, \iota'_2$ are their adjoints, respectively. Also $I(h'_2,h_2)_U:L^2(U,\wedge^{0,*}T^*X\otimes E_2\vert_U; h_2) \to L^2(U,\wedge^{0,*}T^*X\otimes E_2\vert_U; h'_2)$ is the map induced by $Id_{E_2}$ (example \ref{induced by identity notation - example}). This factors through
\begin{small} 
$$L^2(U,\wedge^{0,*}T^*X\otimes E_2 \vert_U; h_2) \xrightarrow{(\sigma_1 ,\varphi_2)} L^2(U,\wedge^{0,*}T^*X\otimes E_1\vert_U) \oplus L^2(U,\wedge^{0,*}T^*X\otimes E_3\vert_U) \xrightarrow{(\varphi_1,\sigma_2)} L^2(U,\wedge^{0,*}T^*X\otimes E_2\vert_U; h'_2)$$
\end{small}
Because $(\phi_1,\sigma_1)$ has norm one, and $(\sigma_1,  \varphi_2)$ has a bounded norm (independent of $U$), then norm of $I(h'_2,h_2)_U$ is also independent of $U$, therefore we can glue all the data to obtain  $I(h'_2,h_2)  \chi_{D_{E_2}} =  \chi_{D'_{E_2}} I(h'_2,h_2) $. This proves that $\hat{\tau}^D_{X,g}$ is exact.

\end{proof}

%\begin{rmk} Note that target of the functor $\hat{\tau}^D_{X,g}$ is $Ch'(\mathfrak{D}/\mathfrak{C})_{C_0(X)}$, where the morphisms are chain maps in $\mathfrak{D}_{C_0(X)}$ that make the diagrams commute in the Paschke category. However we have used the same notation for when we were considering morphisms when they came from $\mathfrak{D}$, or when we composed them with something in $()$. \end{rmk}

\subsection{Restriction to Open Subsets}

% By proposition \ref{functoriality of tau1 prop}, we know there exists an exact functor $\hat{\tau}^D_{X,g}: \mathcal{P}_{\text{b,d}}(X)\to Ch'(\mathfrak{D}/ \mathfrak{C})_{C_0(X)}$ (definition \ref{cat of bdd exact chain complex with different structure-def}). 
% In this subsection, we will define an exact functor $\tau^D_X: \mathcal{P}(X) \to Ch'(\mathfrak{D}/ \mathfrak{C})_{C_0(X)}$.

\begin{lem}\label{tau hat commutes with restriction to open subsets - lem}
Let $X$ be a complex manifold, and let $U$ be an open subset. Then the diagram below commutes up to homotopy.
\begin{center}
\begin{tikzcd}
K(\mathcal{P}_{\text{b,d}}(X,g)) \ar[d, "res"] \ar[r, "\hat{\tau}^D_{X,g}"] 
& K(Ch'(\mathfrak{D}/ \mathfrak{C})_{C_0(X)}) \ar[d, "res"] \\
K(\mathcal{P}_{\text{b,d}}(U,g)) \ar[r,  " \hat{\tau}^D_{U,g}"] 
& K(Ch'(\mathfrak{D}/ \mathfrak{C})_{C_0(U)})
\end{tikzcd}
\end{center}
\end{lem}

\begin{proof}
Let $(E,h)$ be an object of $\mathcal{P}_{\text{b,d}}(X)$. It suffices to show that in the diagram below (which is \emph{not} commutative on the nose), $res^X_U \hat{\tau}^D_{X,g}(E,h)$ is naturally isomorphic to $\hat{\tau}^D_{U,g} res^X_U (E,h)$.
\begin{center}
\begin{tikzcd}
\mathcal{P}_{\text{b,d}}(X,g) \ar[d, "res^X_U"] \ar[r, "\hat{\tau}^D_{X,g}"] 
& Ch'(\mathfrak{D}/ \mathfrak{C})_{C_0(X)} \ar[d, "res^X_U"] \\
\mathcal{P}_{\text{b,d}}(U,g) \ar[r,  " \hat{\tau}^D_{U,g}"] 
& Ch'(\mathfrak{D}/ \mathfrak{C})_{C_0(U)}
\end{tikzcd}
\end{center}

% Let $\iota: U \to X$ be the inclusion. Recall that restriction maps on the Paschke category are defined by precomposing the representations with $\iota_*: C_0(U)\to C_0(X)$. Denote the representations by $\rho:C_0(X)\to L^2(X,\wedge^{0,*}T^*X\otimes E)$ and $\rho_U: C_0(U)\to L^2(U,\wedge^{0,*}T^*X\otimes E\vert_U)$. 
Denote the restriction map $L^2(X,\wedge^{0,*}T^*X\otimes E)\to L^2(U,\wedge^{0,*}T^*X\otimes E\vert_U)$ given by multiplying with the characteristic function of $U$ by $\pi$.

Let $u$ be compactly supported section of $L^2(U, \wedge^{0,*}T^*X\otimes E\vert_U)$. Then by \cite[10.3.1.]{higson2000analytic} there exists $\epsilon >0$ so that for $\vert s\vert < \epsilon$, $e^{\sqrt{-1}s D_E^U}u = e^{\sqrt{-1}s D_E} \pi^* u $ are supported on $U$. Let $\phi$ be a normalizing function so that its Fourier transform is supported in the interval $[-\epsilon, \epsilon]$. Then by \cite[10.3.5.]{higson2000analytic} we get that $\phi(D_E^U)u = \pi \phi(D_E) \pi^* u$. Since $\phi-\chi\in C_0(\mathbb{R})$, then by lemma \ref{basic props of func calc on mfld lem} $\chi(D_E^U)=\pi \chi(D_E) \pi^*$ in the Paschke category $(\mathfrak{D}/ \mathfrak{C})_{C_0(U)}$. Therefore $\pi$ is a chain map from $res^X_U \hat{\tau}^D_{X,g}(E,h)$ to $\hat{\tau}^D_{U,g} res^X_U (E,h)$.

Since $\pi \pi^*=Id$ and $\pi^* \pi - Id$ is characteristic function of $X\setminus U$ which is locally compact in $(\mathfrak{D}/ \mathfrak{C})_{C_0(U)}$, then $\pi$ induces an isomorphism.

Therefore there is a natural transformation from $res^X_U \hat{\tau}^D_{X,g}$ to $\hat{\tau}^D_{U,g} res^X_U$, meaning these two functors induce homotopic maps of K-theory spectra.
\end{proof}

\begin{prop} \label{defining tau^D_V prop}
Let $X$ be a complex manifold. Then for each relatively compact open subset $V$ of $X$ there exists an exact functor $\tau^D_{V}$ that makes the square below commute up to homotopy. Furthermore, these functors are compatible with further restriction to open subsets, i.e. for an open subset $W$ of $V$, the triangle on the bottom of the diagram commutes up to homotopy as well.
\begin{equation}\label{defining tau^1_V commutative -diagram}
\begin{tikzcd}
K(\mathcal{P}_{\text{b,d}}(X,g)) \ar[d] \ar[r, "\hat{\tau}^D_{X,g}"] 
& K(Ch^b(\mathfrak{D}/ \mathfrak{C})_{C_0(X)}) \ar[d, "res^X_U"] \\
K(\mathcal{P}(X)) \ar[r, dashrightarrow, "\exists \tau^D_{U}"]  \ar[rd, dashrightarrow, swap, "\exists \tau^D_{W}"]
& K(Ch'(\mathfrak{D}/ \mathfrak{C})_{C_0(V)}) \ar[d, "res^U_V"] \\
& K(Ch'(\mathfrak{D}/ \mathfrak{C})_{C_0(W)})
\end{tikzcd}
\end{equation}

\end{prop}
\begin{proof}
% Since $U$ is relatively compact, then for a fixed choice of a hermitian metric $g$ on $X$ and any two choices of hermitian metrics $h_1,h_2$ on the bundle $E$, the map induced by the identity map of the bundle $\iota : L^2(U, \wedge^{0,*}T^*X \otimes E\vert_U)\to L^2(U, \wedge^{0,*}T^*X \otimes E\vert_U)$ (where the $L^2$-completions are with respect to $h_1,h_2$, respectively) is a bounded map of Hilbert spaces, %over a relatively compact open subset $U$ of $X$, and by lemma \ref{different metrics on v.b. give the same operator in Paschke cat - Lemma} induces an isomorphism from $res^X_U \hat{\tau}^D_{X,g}(E,h_1)\to res^X_U \hat{\tau}^D_{X,g}(E,h_2)$ in the Paschke category $(\mathfrak{D}/ \mathfrak{C})_{C_0(U)}$. 

For each object $E$ of $\mathcal{P}(X)$, \emph{choose} a hermitian metric $h(E)$ \footnote{Note that we are assuming the axiom of choice. Also, we are only working over a small skeletal subcategory of $\mathcal{P}(X)$.
%since we are looking at K-theory of the categories, then we can consider a skeletal subcategory $\mathcal{P}'(X)$ of $\mathcal{P}(X)$ which has a \emph{set} of objects, and then do the construction. 
}. Then define $\tau^D_{V,h}(E)= res^X_V \hat{\tau}^D_{X,g}(E,h(E))$. Also, for a morphism of bundles $\varphi:E_1\to E_2$, define $\tau^D_{V,h}(\varphi)$ through the composition below, where the first map is given by projection, and the last one is given by extension by zero.
\begin{equation*}
L^2(X,\wedge^{0,*}T^*X\otimes E_1) \to L^2(V,\wedge^{0,*}T^*X\otimes E_1\vert_V ) \xrightarrow{\hat{\tau}^D_{V,g}(\varphi\vert_V)}  L^2(V,\wedge^{0,*}T^*X\otimes E_2\vert_V ) \to  L^2(X,\wedge^{0,*}T^*X\otimes E_2 )
\end{equation*}
Note that $\hat{\tau}^D_{X,g}(\varphi)$ is \emph{not} necessarily defined, as $\varphi$ could induce an unbounded map of Hilbert spaces, however by restricting to the relatively compact open subset $V$, the composition above is indeed a well-defined map.

Since $\hat{\tau}^D_{V,g}$ is a functor, then $\tau^D_{V,h}$ is also a functor, i.e. for composable maps of bundles $\varphi_1,\varphi_2$, we have $\tau^D_{V,h}(\varphi_2\circ \varphi_1 ) = \tau^D_{V,h}(\varphi_2) \circ \tau^D_{V,h}(\varphi_1)$. Exactness of $\tau^D_{V,h}$ also follows from that of $\hat{\tau}^D_{V,g}$. Hence we have an induced map of spectra
\begin{equation}\label{defining tau^1 equation}
\tau^D_{V,h} : K(\mathcal{P}(X))\to K(Ch'(\mathfrak{D}/ \mathfrak{C})_{C_0(V)}).
\end{equation}

The square in the diagram \ref{defining tau^1_V commutative -diagram} commutes (up to homotopy) because for any object $(E_1,h_1)$ of $\mathcal{P}_{\text{b,d}}(X,g)$, by corollary \ref{different metrics on v.b. give the same dolbeault operator in paschke cat - cor}, the identity map of $E$ induces an isomorphism from $res^X_V \hat{\tau}^D_{X,g}(E_1,h_1)$ to $res^X_V \hat{\tau}^D_{X,g}(E_1,h(E_1)) = \tau^D_{V,h}(E) $. Also for a morphism $\varphi:(E_1,h_2)\to (E_2,h_2)$ in $\mathcal{P}_{\text{b,d}}(X,g)$, the difference $res^X_V \hat{\tau}^D_{X,g}(\varphi) - \tau^D_{V,h}(\varphi)$ is locally compact, because multiplying by characteristic function of $X\setminus V$ is locally compact in $(\mathfrak{D}/ \mathfrak{C})_{C_0(V)}$.

The functor defined $\tau^D_{-,h}$ commutes (up to homotopy) with restriction to further open subset $W\subset V$ because multiplying by characteristic function of $V\setminus W$ is locally compact in $(\mathfrak{D}/ \mathfrak{C})_{C_0(W)}$. Therefore the triangle in the diagram \ref{defining tau^1_V commutative -diagram} commutes as well.

Note that the choices of metrics $h(E)$ on $E$ do not affect the map \ref{defining tau^1 equation} up to homotopy, because again by corollary \ref{different metrics on v.b. give the same dolbeault operator in paschke cat - cor}, for any two choices $h_1,h_2$, the objects $\tau^D_{V,h_1}(E), \tau^D_{V,h_2}(E)$ are naturally isomorphic, hence all the different functors are homotopic.% which means they all induce the same map on the level of K-theory.

%%%%%%%%
% Let $\varphi:E_1\to E_2$ be a morphism in $\mathcal{P}(X)$. Then over a relatively compact open subset $U$ of $X$, for any choice of metrics $h_i$ on $E_i$, the induced map $L^2(U, \wedge^{0,*}T^*X \otimes E_1\vert_U)\to L^2(U, \wedge^{0,*}T^*X \otimes E_2\vert_U)$ is a bounded map of Hilbert spaces. 
%Also if $\sigma:E_2\to E_1$ so that $\sigma \varphi$ for any choice of metric $h_1$ on $E_1$, we can choose a metric $h_2$ on $E_2$ so that the induced map $L^2(X, \wedge^{0,*}T^*X \otimes E_1)\to L^2(X, \wedge^{0,*}T^*X \otimes E_2)$ has norm at most one, by simply making sure that the fiber-wise norm is at most one. 
%%%%%%%%%%
\end{proof}

\begin{cor}\label{tau^d commutes with restriction to open subsets - cor}
Let $X$ be a complex manifold. %, let $g$ be a hermitian metric on $X$. 
Then the functor $\tau^D$ defined in proposition \ref{defining tau^D_V prop} commutes with restriction to open subsets, i.e. for open subset $U$ of $X$ and relatively compact open subset $V$ of $X$ and open subset $W$ of $U\cap V$ which is relatively compact as an open subset of $U$, %\footnote{The reason using $W$ instead of $U\cap V$ is that if $U\cap V$ is not relatively compact in $U$, then the map $\tau^D_{U\cap V}$ is not defined.}
the diagram below commutes up to homotopy.
\begin{center}
\begin{tikzcd}
K(\mathcal{P}(X)) \ar[r, "\tau^D_V"] \ar[d, "res^X_U"]
& K(Ch'(\mathfrak{D}/\mathfrak{C})_{C_0(V)}) \ar[d, "res^V_W"] \\
K(\mathcal{P}(U)) \ar[r, "\tau^D_W"]
& K(Ch'(\mathfrak{D}/\mathfrak{C})_{C_0(W)})
\end{tikzcd}
\end{center}
\end{cor}

\begin{proof}
Consider the diagram below, where all the arrows with no labels are the natural ones. %either the forgetful maps, or given by the natural restrictions.
\begin{center}
\begin{tikzcd}[row sep=scriptsize, column sep=scriptsize] 
& K(\mathcal{P}_{\text{m,b}}(X,g)) \arrow[dl] \arrow[rr, near start, "\hat{\tau}^D_{X,g}"] \arrow[dd]
& & K(Ch'(\mathfrak{D}/\mathfrak{C})_{C_0(X)}) \arrow[dl] \arrow[dd] \\
K(\mathcal{P}_{\text{m,b}}(U,g)) \arrow[rr, crossing over, near end, "\hat{\tau}^D_{U,g}"] \arrow[dd] 
& & K(Ch'(\mathfrak{D}/\mathfrak{C})_{C_0(U)}) \\
& K(\mathcal{P}(X)) \arrow[dl] \arrow[rr, near start, "\tau^D_V"] 
& & K(Ch'(\mathfrak{D}/\mathfrak{C})_{C_0(V)}) \arrow[dl] \\
K(\mathcal{P}(U)) \arrow[rr, near end, "\tau^D_W"] 
& & K(Ch'(\mathfrak{D}/\mathfrak{C})_{C_0(W)}) \arrow[from=uu, crossing over]\\
\end{tikzcd}
\end{center} 
The squares on the left and the one on the right commute because restriction maps (and the forgetful functors $\mathcal{P}_{\text{m,b}}\to \mathcal{P}$) are natural. By lemma \ref{tau hat commutes with restriction to open subsets - lem} the square on the top commutes up to homotopy. By proposition \ref{defining tau^D_V prop} the squares in the back and on the front commute up to homotopy as well. This proves that the square on the bottom commutes up to homotopy.
\end{proof}

\begin{rmk}
All the results in this subsection hold whether we use $K^{alg}$ or $K^{top}$.
%For both the proposition and the lemma above, we could have considered either $K^{alg}(Ch'(\mathfrak{D}/ \mathfrak{C})_{C_0(X)})$ or $K^{top}(Ch'(\mathfrak{D}/ \mathfrak{C})_{C_0(X)})$.
\end{rmk}

\section{Main Results}\label{Main results section}
\subsection{K-theory of the Paschke category}
In this subsection, we will compute the K-theory groups of the Paschke category and the Calkin-Paschke category. 

Let $A$ be a $C^*$-algebra. Let $K_{\cdot}^{top}(A)$ denote the topological K-homology groups of $A$, which are contravariant functors of the $C^*$-algebra, and let $K^{\cdot}_{top}(A)$ denote the topological K-theory groups of $A$, which are covariant functors. The reason for the unusual naming is that we are primarily interested in the case when $A$ is the $C^*$-algebra $C_0(X)$ of continuous complex valued functions on the (locally compact and Hausdorf) topological space $X$ which vanish at infinity, and in this case functoriality matches the expectations.

Let $\mathcal{A}$ be a topological exact category, and recall that $K^{top}(\mathcal{A})$ denotes the K-theory spectrum of $\mathcal{A}$ with respect to the fat geometric realization. % as explained in the previous subsection. Beware that this is different from $K^{alg}(\mathcal{A})$ obtained by forgetting the topological structure of $\mathcal{A}$. We drop the superscript $K^{top}$ in case no confusion shall arise.). 
Since the additivity theorem holds for K-theory of topological categories, then this is a \emph{connective spectrum}, i.e. there are no negative K-theory groups.

%For a $C^*$-algebra $A$, we can apply the process of the previous subsection to obtain the K-theory spectrum $K^{top}((\mathfrak{D}/\mathfrak{C})_A),K^{top}((\mathfrak{D}/\mathfrak{C})'_A)$. 

%There is also a surjective map $ K_{top}^1(A)\cong Ext'(A) \to K^{top}_0(\mathfrak{Q}'_A)$ when $A$ is \emph{nuclear} \footnote{$Ext(A)$ was defined in \cite{brown1977extensions} as the semi-group of unitary equivalence classes of unital injective representations of $A$ to the Calkin algebra, cf. \cite[2.7.1.]{higson2000analytic}. In here $Ext'(A)$ denotes the invertible elements in the semi-group $Ext(A)$, and was proved to be isomorphic to the K-homology group $K_{top}^1(A)$. As a corollary of Voiculescu's theorem, $Ext(A) = Ext'(A)$ when $A$ is a nuclear unital $C^*$-algebra.}.
%%
Here is the main result of this subsection.

\begin{thm}\label{Paschke category gives K-homology theorem}
The $(1-i)$'th topological K-homology group $K_{top}^{1-i}(A)$ of a $C^*$-algebra $A$, is isomorphic to the $i$'th topological K-theory groups of the exact $C^*$-categories $(\mathfrak{D}/ \mathfrak{C})_A$ and $(\mathfrak{D}/\mathfrak{C})'_A$  for $i \geq 1$. If $A$ is unital and nuclear, then $K_{top}^1(A)=K^{top}_0((\mathfrak{D}/ \mathfrak{C})'_A)$ as well.
For a $*$-morphism $f:A\to B$, this isomorphisms commutes with respect to the pull-back maps $f^*$. 

In particular if $A=C_0(X)$ for a locally compact Hausdorf topological space $X$, % of homotopy type of a CW-complex\footnote{This condition is necessary for some proofs in \cite{mitchener2001symmetric}},
then the topological K-homology groups $K^{top}_{i-1}(X)$ are isomorphic to $K_i((\mathfrak{D}/\mathfrak{C})'_{C_0(X)})$ for $i \geq 0$. This isomorphism also commutes with the restriction maps to open subsets.
\end{thm}

\begin{proof}
Recall that the $K_0$ group of a Waldhausen category is the free abelian group generated by the weak equivalence classes of objects of the category, modulo the relations induced by the cofibration sequences. The same is true for \emph{topological} Waldhausen categories (cf. \cite[4.8.4.]{weibel2013k}). In particular since all the non-zero objects of the category $\mathfrak{Q}_A$ are isomorphic to each other by \ref{Voiculescu's theorem ample reps-cor}, hence $K_0(\mathfrak{Q}_A)=0$.

In the case when $A$ is a unital $C^*$-algebra, recall that $Ext(A)$ \cite{brown1977extensions} is defined as the semi-group of unitary equivalence classes of unital injective representations of $A$ to the Calkin algebra $(\mathfrak{B}/\mathfrak{K})(H)$ (cf. \cite[2.7.1.]{higson2000analytic}). When $A$ is nuclear, then as a corollary of Voiculescu's theorem \footnote{Note that for a nuclear $C^*$-algebra $A$ and $C^*$-algebra $B$ with a $C^*$-ideal $K$, a $*$-morphism $A\to B/K$ lifts to a \emph{completely positive} map $A\to B$ (cf. \cite[3.3.6.]{higson2000analytic}), and \emph{Stinespring's theorem} (cf. \cite[3.1.3.]{higson2000analytic}) shows that each completely positive map to $\mathfrak{B}(H)$, can be written as $V^* \rho V$, where $V:H\to H'$ is an isometry and $\rho$ is a representation to $\mathfrak{B}(H')$. Then the restriction of $\rho$ to the orthogonal complement of image of $V$ induces a representation of $A$ to the Calkin algebra (cf. \cite[3.1.6.]{higson2000analytic}).} we know that $Ext(A)$ is in fact a group, and isomorphic to the first K-homology group $K_{top}^1(A)$. 

\begin{lem} \label{K0 Calkin Pashcke cat lemma}
Let $A$ be a unital and nuclear $C^*$-algebra. Then $K_0(\mathfrak{Q}'_A)=K_{top}^1(A)$.
\end{lem}

\begin{proof}
Let $\rho'_i:A\to (\mathfrak{B}/ \mathfrak{K})(H_i)$ be non-zero objects in $\mathfrak{Q}'_A$ for $i=1,2$, which are isomorphic, i.e. there exists isomorphisms $T:\rho'_1\to \rho'_2$ and $S:\rho'_2\to \rho'_1$ which are inverses to each other. By definition of a $C^*$-category, for a positive operator $T^*T\in \mathfrak{Q}'_A(\rho'_1)$, there exists an operator $F\in \mathfrak{Q}'_A(\rho'_1)$ so that $F^*F=T^*T$. Since $S,T$ are invertible, then so are $F$ and $FS:\rho'_2\to \rho'_1$. We have $(FS)^*(FS)= S^*F^*FS=S^*T^*TS =Id$ and $((FS)(FS)^*)F= (FSS^*F^*)F= FSS^*T^*T=FST=F$ in the category $\mathfrak{Q}'_A$. Hence $FS$ is a unitary isomorphism in this category. Choose representatives for $S,F$ in the category $\mathfrak{D}_A$. Because $\rho'_1,\rho'_2$ are unital, then it means $FS$ is also a Fredholm operator, and in particular has closed image and finite dimensional kernel and cokernel. Hence there exists closed subspaces $H'_i\subset H_i$ of finite codimension so that $\pi_1 FS \iota_2:H'_2\to H'_1$ is a isomorphism of Hilbert spaces, where $\iota_i:H'_i\to H_i$ is the inclusion and $\pi_i:H_i\to H'_i$ is the projection for $i=1,2$. Since $\ker(FS)= coker(S^*F^*), \ker(S^*F^*)=coker(FS)$ then $\pi_1 FS\iota_2$ is a unitary map of Hilbert spaces. Let $\nu'_i= \pi_i \rho'_i \iota_i:A\to (\mathfrak{B}/ \mathfrak{K})(H'_i)$ for $i=1,2$. Then we just showed that $\nu'_1, \nu'_2$ are unitarily equivalent. Since the difference between $\rho'_i, \nu'_i$ is a finite dimensional Hilbert space, then they represent the same class in $Ext(A)$ \footnote{Let $H''_i$ denote the orthogonal complement of $H'_i$ in $H_i$, which is finite dimensional. Then the direct sum of the zero representation from $A$ to $(\mathfrak{B}/\mathfrak{K})(H''_i)$ and $\nu'_i$ is equal to $\rho'_i$. }.  This shows that $Ext(A)\to K_0(\mathfrak{Q}'_A)$ is injective. Surjectivity follows from the definition.
\end{proof}

By propositions \ref{cofinality waldhausen k-theory prop} and \ref{ample reps cofinal subcat- prop}, the maps of spectra induced by inclusion of subcategories $K(\mathfrak{Q}_A)\to K((\mathfrak{D}/ \mathfrak{C})_A)$ and $K(\mathfrak{Q}'_A)\to K((\mathfrak{D}/ \mathfrak{C})'_A)$ are both homotopy equivalences. By remark \ref{paschke cats cofinal comparison-remark} and proposition \ref{cofinality waldhausen k-theory prop}, the map $ K((\mathfrak{D}/ \mathfrak{C})_A)\to K((\mathfrak{D}/ \mathfrak{C})'_A)$ induces an isomorphism on the $i$'th K-groups for $i\geq 1$.

Let $A$ be a $C^*$-algebra, let $\rho$ be an ample representation of $A$, and let $\mathfrak{R}$ be a full subcategory of $\mathfrak{Q}_A$ with two objects: the zero representation $A\to 0$ and $\rho$. %(as a zero object with unique morphisms coming from and to it), and $\rho$, with $Hom_{\mathcal{R}}(\rho, \rho)= Hom_{\mathfrak{Q}_{A}}(\rho, \rho) = \mathfrak{Q}_{\rho}(A) $. 
Then this is a $C^*$-category and also a skeleton for the category $\mathfrak{Q}_{A}$, as all ample representaions are isomorphic to each other. But since every short exact sequence in $\mathfrak{Q}_{A}$ splits, then by a result of Mitchener \cite{mitchener2001symmetric}, $\Omega \Vert w S_{\cdot}\mathfrak{Q}_{A}\Vert $ is homotopy equivalent to $BGL_{\infty}(Sk(\mathfrak{Q}_{A}))$, where the latter is defined in \cite[6.1.]{mitchener2001symmetric} and $Sk(\mathfrak{Q}_{A})$ denotes the skeleton of the additive category $\mathfrak{Q}_{A}$. Hence the K-theory space of $\mathfrak{Q}_{A}$ is homotopy equivalent to $BGL(\mathfrak{R})$, which by definition is homotopy equivalent to $BGL_{\infty}(\mathfrak{Q}(A))$. 

Therefore when $i \geq 1$ we have the following sequence of isomorphisms of abelian groups, where the first isomorphism is given by Paschke duality \cite{paschke1981k}, %\footnote{This is the only part that needs the assumption $i \geq 1$.}, 
the second isomorphism is one of the equivalent definitions of topological K-theory groups, the third and the fourth one were explained above, and the last one follows from propositions \ref{cofinality waldhausen k-theory prop} and \ref{ample reps cofinal subcat- prop}.
% together with what we mentioned above, and also equivalence of group completion and Waldhausen's $S_{\cdot}$ construction for the topological category $\mathfrak{Q}_{A}$ %%(recall that all short exact sequences in the Paschke category are split.)% , and finally by corollary \ref{ample reps cofinal subcat- cor}, we have:

\begin{equation}\label{K-homology of paschke category is k-homology} 
\begin{array}{c}
K_{top}^{1-i}(A) \cong K^{top}_i(\mathfrak{Q}(A))= \pi_i(BGL(\mathfrak{Q}(A))) \cong \pi_i(BGL_{\infty}(\mathfrak{R}))   \\ 
\cong K^{top}_i(\mathfrak{Q}_{A}) \cong K^{top}_i((\mathfrak{D}/\mathfrak{C})_{A}). 
\end{array}
\end{equation}
% \footnote{If one has a version of cofinality for topological symmetric monoidal categories similar to \cite[4.4.11.b.]{weibel2013k}, it would slightly simplify the proof above, as there would be no need to consider the subcategory $\mathfrak{Q}_{I,A}$ of $(\mathfrak{D}/ \mathfrak{C})_{I,A}$. }
% In particular, if $I=A$, we get:

% Note that $K^{top}_0(\mathfrak{Q}_{A})$ is generated by the class of an ample representation $\rho$. However, since $\rho\oplus \rho$ is isomorphic to $\rho$ in $\mathfrak{Q}_{A}$, then it means we have a short exact sequence $0 \to \rho \to \rho \to \rho \to 0$, which means that $K^{top}_0(\mathfrak{Q}_{A})=0$. 
This means we have proved the first part of the theorem.

% Now we wish to compare the pull-back maps on the classical K-homology with the pull-back maps on Paschke category through the isomorphism given by the theorem above.
%% Since these two versions of K-homology i.e. the classical version, %defined in definition \ref{classical k-homology definition}, 
%% and the version defined here by the Paschke category are isomorphic, we wish to compare the pull-back maps.\\
Let $A,B$ be unital $C^*$-algebras with ample representations $\rho_A:A\rightarrow \mathfrak{B}( H_A)$ and $ \rho_B: B \rightarrow  \mathfrak{B} (H_B)$.
Let $\alpha : A\rightarrow B$ be a unital map of $C^*$-algebras, then by Voiculescu's theorem there exists an isometry $V:H_B \rightarrow H_A$ so that $V^*\rho_A(a) V - \rho_B(\alpha(a))$ is compact for all $a\in A$. Note that $VV^*\in \mathfrak{B}(H_A)$ is a projection which commutes with the representation $\rho_A$ \cite[3.1.6.]{higson2000analytic} modulo compact operators. Also note that $V:H_B\rightarrow VV^* H_A$ is an isomorphism of Hilbert spaces. Now the function $Ad_V(T)=VTV^*$ gives a map 
\begin{equation*}
Ad_V(-):(\mathfrak{D}/ \mathfrak{C})_B(\rho_B) \cong \mathfrak{Q}(B) \rightarrow \mathfrak{Q}(A) \cong (\mathfrak{D}/ \mathfrak{C})_A(\rho_A),
\end{equation*}
and hence induces a map on the K-homology groups which only depends on $\alpha$, i.e. does not depend on the choices of ample representations $\rho_A, \rho_B$ and the isometry $V$. %First, note that for the isometry $V:H_B\rightarrow H_A$, $VV^*\in \mathfrak{B}(H_A)$ is a projection which commutes with the representation $\rho_A$ \cite[3.1.6.]{higson2000analytic} modulo compact operators. Also note $V:H_B\rightarrow VV^* H_A$ is an isomorphism of Hilbert spaces. \\

Let $T\in (\mathfrak{D}/ \mathfrak{C})_B(\rho_B) \cong \mathfrak{Q}(B)$ be a unitary element. The pull-back map of K-homology groups sends $T$ to the unitary
\begin{equation*}
V(T-Id_{H_B})V^*+Id_{H_A} = VTV^* \oplus (Id -VV^*)\in (\mathfrak{D}/ \mathfrak{C})_A(\rho_A) \cong \mathfrak{Q}(A).
\end{equation*}  
On the other hand, pulling back in the Paschke category is given by precomposing with the representation. Hence $T\in (\mathfrak{D}/ \mathfrak{C})_B(\rho_B)$ is sent to $T\in (\mathfrak{D}/ \mathfrak{C})_A(\rho_B \circ \alpha)$. These two procedures give two different maps from the unitaries (or invertible elements) in $(\mathfrak{D}/ \mathfrak{C})_B(\rho_B)$ to the topological space $\Omega^2 \Vert S_{\cdot} (\mathfrak{D}/ \mathfrak{C})_A\Vert $.% (Recall that we are working with the fat geometric realization here.)

Let $S=Id_{H_B} \oplus ( VTV^* \oplus (Id_{H_A} -VV^*))\in \mathfrak{B}(H_B\oplus H_A)$. Consider the following two "prisms" in $wS_{\cdot}(\mathfrak{D}/ \mathfrak{C})_A$ where the cofibration and the quotient maps on the left diagram are the trivial ones, but the cofibrations on the right diagram are given by $(0,V):H_B \rightarrow H_B\oplus H_A$ and the quotient maps are given by $ V +  (Id_{H_A} - VV^*) :H_B \oplus H_A \rightarrow H_A$. By \cite[3.1.6.]{higson2000analytic} these maps are pseudo-local. It is also easy to check that both diagrams below commute. 
\begin{center}
\begin{tikzcd}
& \rho_A \ar[rd,  "VTV^* \oplus (Id-VV^*)"] 
& & & & \rho_A \ar[rd, "Id"] & \\
\rho_B \ar[rd, "Id"]  \ar[r, rightarrowtail] 
& \rho_B \oplus \rho_A \ar[rd, "S"]  \ar[u, twoheadrightarrow] 
& \rho_A  
& & \rho_B \ar[rd, "T"] \ar[r, rightarrowtail]
& \rho_B \oplus \rho_A  \ar[u, twoheadrightarrow] \ar[rd, "S"]
& \rho_A \\ 
& \rho_B \ar[r, rightarrowtail]   
&  \rho_B \oplus \rho_A \ar[u, twoheadrightarrow] 
& & & \rho_B \ar[r, rightarrowtail]
& \rho_B \oplus \rho_A \ar[u, twoheadrightarrow]
\end{tikzcd}
\end{center}
Since fat geometric realization of a point is the infinite dimensional ball which is contractible, then in the fat geometric realization of $wS_{\cdot}(\mathfrak{D}/ \mathfrak{C})_A$, the side of the prism $\Delta_{top}^1 \times \Delta _{top} ^2$ corresponding to the identity map is contractible. Hence we get a homotopy from $S$ to $VTV^*\oplus (Id-VV^*)$ induced by the left diagram and also from $S$ to $T$ induced by the right diagram, by "sliding" one side of the prism towards the other along the contractible side (corresponding to identity map).
\footnote{To be more precise, by the additivity theorem we know $t, s+q: \mathcal{E}(\mathfrak{D}/ \mathfrak{C})_A \rightarrow (\mathfrak{D}/ \mathfrak{C})_A$ are homotopic, where $\mathcal{E}$ is temporarily denoting the category of cofibration sequences, and $s,t,q$ refer to the first (source), second (target), and the third (quotient) object in the cofibration sequence. By applying this to the prism on the right we get that $T+Id$ is homotopic to $S$, and by applying it on the prism on the left we get that $(VTV^*\oplus (Id-VV^*))+Id$ is homotopic to $S$, but since $Id$ is contractible in the fat geometric realization, then we get that $T$ is homotopic to $VTV^* \oplus (Id-VV^*)$.} 
Therefore the two different pushforward maps are homotopic to each other and they induce the same map on the level of K-homology groups. \footnote{One may wonder why we did not simply say that $VTV^* \oplus (Id -VV^*)$ is direct sum of $VTV^*$ and $Id -VV^*$, and argue similar to above that the identity on the Hilbert space $(Id-VV^*)H_A$ corresponds to a contractible side of a prism, and then "slide" $VTV^* \oplus (Id_{H_A}-VV^*)$ directly onto $VTV^*$ which is isomorphic to $T$ in the category $(\mathfrak{D}/ \mathfrak{C})_A$, to obtain a homotopy. The reason is that the restriction of $\rho_A$ to $(Id-VV^*)H_A$ is only a representation up to compact operators, and the Hilbert space $(1-VV^*)H_A$ does not come with a representation, which means we can not simply consider $(1-VV^*)H_A$ as an object in $(\mathfrak{D}/\mathfrak{C})_A$. Note that however, we can consider the Hilbert space $VV^* H_A$ together with the representation $V(\rho_B \alpha)V^*$.}.\\

% This shows that the class in $K_1((\mathfrak{D}/\mathfrak{C})_A)$ corresponding to the automorphism $(VUV^*\oplus 1-P)$ of $\rho_A$ is equivalent to the class corresponding to the automorphism $VUV^*$ of $\rho_A^1$, as the class corresponding to the identity map of any object in $K_1$ is zero\footnote{One may wonder why we didn't simply say that $VUV^* \oplus (1-P)$ is direct sum of $VUV^*$ and $1-P$, and the latter automorphism corresponds to the zero class in $K_1$. The reason is that the restriction of $\rho_A$ to $(1-P)H_A$ is only a representation up to compact operators, hence the Hilbert space $(1-P)H_A$ does not come with a representation, which means we can't simply consider $(1-P)$ as automorphism of some object in $(\mathfrak{D}/\mathfrak{C})_A$. }.
% Since $V:H_B \rightarrow PH_A$ is a unitary, then the class of the automorphism $VUV^*$ of $\rho_A^1$ is equal to the class of the automorphism $U$ of $\rho_B \alpha$. This shows that the two push-forwards on K-homology match with each other.\\ %%%%%% SHOULD I ADD ANYTHING  ???????

%K-homology also has \emph{restriction maps}. 
In the context of Paschke categories restriction maps are defined similar to pull-back maps, i.e. by precomposing with the representation. To be more precise, let $X$ be a locally compact and Hausdorf topological space, and let $U$ be an open subset. The inclusion $j :U \hookrightarrow X$ induces an inclusion $j_*: C_0(U)\hookrightarrow C_0(X)$ of $C^*$-algebras, given by extending functions by zero. Then the restriction map sends the object $\rho:C_0(X) \rightarrow \mathfrak{B}(H)$ to $j^*(\rho) \coloneqq \rho \circ j_*:C_0(U)\to \mathfrak{B}(H)$.
%, we can consider the restriction $\rho \vert_U:C_0(U) \rightarrow \mathfrak{B}(H)$. It is easy to see if $T\in \mathfrak{D}_A(\rho_1, \rho_2)$, then it is also in $\mathfrak{D}_A(\rho_1\vert_U, \rho_2\vert_U)$, (and similarly for $\mathfrak{C}_A$) as "there are fewer conditions to check". Therefore we get a restriction functor

We follow \cite{roe2013sheaf} to recall the process of defining the (wrong-way) restriction maps on the classical topological K-homology.
%Before comparing the restriction maps, let us first recall the process in the classical settings. We follow \cite{roe2013sheaf} for the following.\\
Let $X, U,j, \rho:C_0(X) \rightarrow \mathfrak{B}(H)$ be as before. If we extend $\rho$ to the Borel functions on $X$, then $\rho(\mathbbm{1}_U)$ is a self-adjoint projection, where $\mathbbm{1}_U$ is the characteristic function of the open subset $U$ of $X$. Let $H_U$ be the image of this projection, and define the representation $\rho_U: C_0(U)\to \mathfrak{B}(H_U)$ by $\rho_U(f) = \pi_U \rho(j_*(f)) \iota_U $ where $\iota_U: H_U \to H$ is the inclusion, $\pi_U:H\to H_U$ is the projection, and $j_*:C_0(U)\to C_0(X)$ is extension by zero. The linear map $\mathfrak{B}(H) \to \mathfrak{B}(H_U) $ defined by $T\mapsto \pi_U T \iota _U$ maps $\mathfrak{D}_{C_0(X)}(\rho)$ to $\mathfrak{D}_{C_0(U)}(\rho_U)$, and $\mathfrak{C}_{C_0(X)}(\rho)$ to $\mathfrak{C}_{C_0(U)}(\rho_U)$. Hence there is an induced map $\mathfrak{Q}_{\rho}(C_0(X))\to \mathfrak{Q}_{\rho_U}(C_0(U))$, which induces the restriction map from the K-homology groups of $X$ to the K-homology groups of $U$.
%Following the notation of the appendix, note that we have two different representations to work with here, one is $\rho^U: C_0(U) \rightarrow \mathfrak{B}(P_U H)$, and the other is $\rho \vert_U : C_0(U) \rightarrow \mathfrak{B}(H)$. 

But the representations $ j^* \rho , \rho_U$ are naturally isomorphic; in fact the maps $\pi_U, \iota_U$ induce the isomorphisms. To show this, first note that $\pi_U,\iota_U$ commute with these representations since for $f\in C_0(U)$, we have $\rho_U(f) \pi_U= \rho_U(f) \rho(\mathbbm{1}_U)= \rho(j_* f) = \pi_U j^*\rho( f)$. Also $\rho_U (\pi_U \iota _U - Id_{H_U})$ and $j^* \rho (\iota_U \pi_U - Id_H)$ are both zero. Therefore these two restriction functors are homotopic to each other as a maps to the $\Omega \Vert w S_{\cdot} (\mathfrak{D}/ \mathfrak{C})_{C_0(U)} \Vert$, which means the two induced restriction maps on K-homology groups are equal to each other.
\end{proof}
\subsection{Descent and The Riemann-Roch Transformation}

\begin{Def} \label{ku and KU definition}
Let $\mathcal{V}_{\mathbb{C}}$ denote the topological category where the objects are finite dimensional complex vector spaces, and morphisms are invertible linear maps. Let $A$ be a nuclear $C^*$-algebra. Then there is a biexact functor $\mathcal{V}_{\mathbb{C}}\times (\mathfrak{D}/ \mathfrak{C})'_A \to (\mathfrak{D}/ \mathfrak{C})'_A$ 
induced by taking the tensor product of the corresponding Hilbert space with the finite dimensional vector space. This induces a map of spectra
\begin{equation}
ku\wedge K^{top}((\mathfrak{D}/ \mathfrak{C})'_A)\to K^{top}((\mathfrak{D}/ \mathfrak{C})'_A)
\end{equation}
where $ku$ is the K-theory spectrum $K^{top}(\mathcal{V}_{\mathbb{C}})$, also known as the \emph{connective complex K-theory spectrum}.

Let $KU$ denote the \emph{(non-connective) complex K-theory spectrum}. 
\end{Def}

Since we have only defined the functor $\tau^D_{X,g}$ on relatively compact open subsets, we will need a descent argument to glue them together. So far we have defined a connective K-homology spectrum for $C^*$-algebras. We need a non-connective K-homology spectrum to make the descent work. It is well known that the process below will give us the non-connective spectrum we need. We will only provide a sketch proof for this lemma.

\begin{lem}\label{smash product gives non-connective topological k homology-lem}
By definition above, we can consider the smash product of spectra $K^{top}((\mathfrak{D}/ \mathfrak{C})'_A) \wedge_{ku} KU$. This has the same homotopy groups as the \emph{non-connective} topological K-homology of $A$. 
\end{lem}

\begin{proof}[Sketch Proof:]
Recall that for a ring spectrum $X_{\cdot}$ and $b\in X_{n} $, multiplication by $b$ induces a map $X_{\cdot} \to \Sigma^{-n}X_{\cdot}$. Then we define $X_{\cdot}[b^{-1}]$
to be the homotopy colimit of the telescope $$ X_{\cdot}\xrightarrow{b} \Sigma^{-n}X_{\cdot} \xrightarrow{\Sigma^{-n}b} \Sigma^{-2n}X_{\cdot} \xrightarrow{\Sigma^{-2n}b} \Sigma^{-3n}X_{\cdot} \to \ldots .$$
Also, (stable) homotopy groups commute with this mapping telescope (cf. \cite[5.1.14.]{elmendorf2007rings}).
Let $\beta\in ku_2$ denote the \emph{bott element}. Then it is well-known that $KU$ is naturally homotopy equivalent to $ku[\beta^{-1}]$ (cf. \cite{snaith1981localized}).  

Since $K^{top}((\mathfrak{D}/\mathfrak{C})'_{A})$ is a $ku$-module, then there is a natural weak equivalence $K^{top}((\mathfrak{D}/\mathfrak{C})'_A)\wedge_{ku} ku[\beta^{-1}] \to K^{top}((\mathfrak{D}/\mathfrak{C})'_A)[\beta^{-1}]$ (cf. \cite[5.1.15.]{elmendorf2007rings}). This is easy to see that the homotopy groups of the latter is $2$-periodic, as one could disregard the first $n$-temrs in the mapping telescope, and that positive homotopy groups of $K^{top}((\mathfrak{D}/\mathfrak{C})'_A)$ are $2$-periodic. The fact that positive homotopy groups of $K^{top}((\mathfrak{D}/\mathfrak{C})'_A)\wedge_{ku} KU$ are isomorphic to the positive homotopy groups of $K^{top}((\mathfrak{D}/\mathfrak{C})'_A)$ follows from the (strongly converging) Atiyah-Hirzebruch spectral sequence (cf. \cite[4.3.7.]{elmendorf2007rings}) and the fact that positive homotopy groups of $ku$ and $KU$ agree with each other. This finishes the proof.
% Let $n\in \mathbb{Z}$, let $\mathbb{S}$ denote the sphere spectrum, and let $\mathbb{S}^n= \Sigma^n \mathbb{S}$, and for spectrums $X,Y$, denote the homotopy class of maps from $X$ to $Y$ by $[X,Y]$. Then by the definition of mapping telescope we obtain $\pi_n \left(K^{top}((\mathfrak{D}/\mathfrak{C})'_A)[\beta^{-1}]\right) =[\mathbb{S}^n, \left(K^{top}((\mathfrak{D}/\mathfrak{C})'_A)\xrightarrow{\beta} \Sigma^{-2}K^{top}((\mathfrak{D}/\mathfrak{C})'_A)\xrightarrow{\Sigma^{-2}\beta} \Sigma ^{-4}K^{top}((\mathfrak{D}/\mathfrak{C})'_A) \to \ldots \right)] \cong \lim \left( [\mathbb{S}^n, K^{top}((\mathfrak{D}/\mathfrak{C})'_A)] \xrightarrow{\beta} [\mathbb{S}^n , \Sigma^{-2}K^{top}((\mathfrak{D}/\mathfrak{C})'_A)]\xrightarrow{\Sigma^{-2}\beta} \ldots \right) \cong \lim \left(  \right)
% $ by \cite[5.1.14.]{elmendorf2007rings}.
%%%%%%%%%%%%%%%%%%
% But for $n\geq 0$, we know that $\lim_m \pi_{n+m} \left(  K^{top}((\mathfrak{D}/\mathfrak{C})'_A) \right)  = K_n^{top}(X)$ which is $2$-periodic. Hence while $l+(n+m)\geq 0$ then $\lim_m \pi_{n+m} \left( \Sigma^{-2l} K^{top}((\mathfrak{D}/\mathfrak{C})'_A)\right) =K^{top}_n(X)$.
\end{proof}

\begin{prop} \label{Mayer-Vietoris Paschke category - prop}
Let $A$ be a $C^*$-algebra, and let $I\subset A$ be an ideal, so that the projection $\pi: A \rightarrow A/I$ has a completely positive section. Then $$ K^{top}((\mathfrak{D} / \mathfrak{C})'_{A/I}) \wedge_{ku} KU \rightarrow  K^{top}((\mathfrak{D} / \mathfrak{C})'_{A}) \wedge_{ku} KU \rightarrow K^{top}((\mathfrak{D} / \mathfrak{C})'_{I}) \wedge _{ku } KU $$ is a homotopy fiber sequence.
\end{prop}

\begin{proof}
It is easy to observe that the composition of the two maps above is null-homotopic. Hence it suffices to show that the homotopy groups of the sequence above induce a long exact sequence of homotopy groups. This is a direct consequence of the six-term exact sequence of K-homology groups \cite[5.3.10.]{higson2000analytic}, lemma \ref{smash product gives non-connective topological k homology-lem} which says that the homotopy groups agree with the K-homology groups and also theorem \ref{K-homology of paschke category is k-homology}, which says that the pull-back maps of the Paschke category agree with the classical pull-backs. 
\end{proof}

We will give definition of descent with respect to \emph{hypercovers} \cite[4.2.]{dugger2004hypercovers} below. The definition essentially states when a presheaf of spectra is in fact a sheaf (up to homotopy). A hypercover over a \emph{site} $X$ is a \emph{simplicial presheaf} $U_{\cdot}$ \cite[Sec 1.]{jardine1987simplical} with an augmentation $U_{\cdot}\to X$, which satisfies certain conditions. This can be thought of as a generalization of \emph{\v{C}ech covers}, which have the form
\begin{equation*} 
\begin{tikzcd}
\ldots & \prod_{j_0,j_1,j_2} (U_{j_0}\cap U_{j_1} \cap U_{j_2} )  \ar[r, shift left] \ar[r, shift right] 
 \ar[r] & \prod _{j_0,j_1} (U_{j_0} \cap U_{j_1}) \ar[r, shift left=1.5] \ar[r, shift right=1.5]
& \prod _{j_0} (U_{j_0}) \ar[r] & X
\end{tikzcd}
\end{equation*}
% i.e. a usual open cover of a topological space together with their $n$-tuple intersections. 
In fact, \v{C}ech covers are hypercovers of \emph{height} zero. Also, for our purposes, \emph{homotopy limits} refer to limits in the homotopy category of spectra.
% DEFINITION HOMOTOPY LIMIT???? ALSO HOMOTOPY CAT OF SPECTRA

\begin{Def} \cite[4.3.]{dugger2004hypercovers}
Let $X$ be an object in the site $\mathcal{C}$ (which can be thought of as a topological space). An object-wise fibrant simplicial presheaf $\mathscr{F}$ \emph{satisfies descent} for a hypercover $U_{\cdot}\to X$ if the natural map from $\mathscr{F}(X)$ to the homotopy limit of the diagram
\begin{equation} \label{hypercover descent equation} 
\begin{tikzcd}
\prod _i \mathscr{F}(U^j_0) \ar[r, shift left] \ar[r, shift right] 
& \prod _{j} \mathscr{F}(U^j_1) \ar[r] \ar[r, shift left=1.5] \ar[r, shift right=1.5]
& \prod_{j} \mathscr{F}(U^j_2) \ldots 
\end{tikzcd}
\end{equation}
is a weak equivalence. Here the products range over the representable summands of each $U_n$. If $\mathscr{F}$ is not object-wise fibrant, we say it satisfies descent if some object-wise fibrant replacement for $\mathscr{F}$ does.
\end{Def}

The fact that topological K-homology satisfies descent is essentially a result of the \emph{Atiyah-Hirzebruch spectral sequence} (cf. \cite{brown1973abstract}\cite{atiyah1961vector}). In fact, by  \cite[4.3.]{dugger2004topological}, for a hypercover $U_{\cdot}\to X$, we have weak equivalences $\text{hocolim}U_{\cdot} \to  \vert U_{\cdot} \vert \to X$, where the second map is induced by taking geometric realization. Since taking smash product in the homotopy category of spectra preserves colimits, then by applying the smash product with an $\Omega$-spectrum $E_{\cdot}$, the colimit of the diagram \ref{hypercover descent equation} smashed with the spectrum $E_{\cdot}$ is weakly equivalent to $E_{\cdot} \wedge X$. When the $\Omega$-spectrum $E_{\cdot} = KU$ is the (non-connective) topological K-theory spectrum, this proves descent. (Also see \cite[2.2.]{antieau2014period} for the case of twisted topological K-theory of CW-complexes.)\\

% \cite{dugger2001universal}  In this subsection we finally give the main definition of this paper, and give a Riemann-Roch transformation for complex manifolds. We will check its functoriality in future works.\\

% We will need one final ingredient to give the main definition of this paper, the Riemann-Roch transformation. 

Now we are ready to define the Riemann-Roch transformation over the relatively compact open subsets of a complex manifold, which in turn induce the Riemann-Roch transformation over the manifold itself.

\begin{Def}
Let $X$ be a complex manifold, and let $V$ be a relatively compact open subset of $X$. Define the functor $\tau_{X,V}$ in the homotopy category of spectra as the composition below

\begin{align*}
K^{alg}&(\mathcal{P}(X))\\
& \big\downarrow  \tau^D_{V,g} && \text{Defined in proposition \ref{defining tau^D_V prop}.} \\
K^{alg}&(Ch'(\mathfrak{D}/ \mathfrak{C})_{C_0(V)}) \\
& \big \downarrow \tau^H_{C_0(V)} && \text{Defined in \ref{the large diagram functor tau 2 equation}. }\\
K^{alg}& (Bi'(\mathfrak{D}/ \mathfrak{C})_{C_0(V)})\\
& \big \downarrow \\
K^{alg}& (Bi^b(\mathfrak{D}/ \mathfrak{C})_{C_0(V)})\\
& \big \downarrow \tau^G_{(\mathfrak{D}/ \mathfrak{C})_{C_0(V)}} && \text{Defined in \ref{Grayson's map equation}. } \\
\Omega K^{alg}&((\mathfrak{D}/ \mathfrak{C})_{C_0(V)})\\
& \big \downarrow c && \text{Definition \ref{comparison map def}.} \\
\Omega K^{top}&((\mathfrak{D}/ \mathfrak{C})_{C_0(V)}) \\
& \big \downarrow \\
\Omega K^{top}&((\mathfrak{D}/ \mathfrak{C})'_{C_0(V)}) \wedge _{ku} KU \\
& \big \downarrow \cong && \text{By lemma \ref{smash product gives non-connective topological k homology-lem}.}\\
K^{top}&(V)
\end{align*}

Note that except for the first one, all the maps above are functorially defined. Also, we showed in proposition \ref{defining tau^1_V commutative -diagram} that in the homotopy category of spectra, $\tau^D$ is compatible with restriction to further open subsets. Therefore for a hypercover $V_{\cdot}$ so that all the open sets in $V_n$ are relatively compact subsets of the open sets in $V_{n-1}$ (and in particular, all are relatively compact in $X$), there is an induced map $\tau:K^{alg}(\mathcal{P}(X))\to \text{holim }K^{top}(V_{\cdot})$ in the homotopy category of spectra, where the latter is referring to the homotopy colimit of the diagram \ref{hypercover descent equation} for the hypercover $V_{\cdot} \to X$. Since topological K-homology satisfies descent, then there is an induced map in the homotopy category of spectra
\begin{equation}\label{riemann-roch transformation equation}
\tau_X: K^{alg}(\mathcal{P}(X))\to K^{top}(X).
\end{equation}
By taking a finer cover if necessary, one can see that the map above is independent of the choice of the hypercover $V_{\cdot}$. 
\end{Def}
% Recall from \ref{Mayer-Vietoris Paschke category - prop} that $K^{top}(-)$ satisfies the Mayer-Vietoris, hence by an inductive argument one can show that it satisfies descent with respect to \emph{hypercovers} (See \cite{dugger2004hypercovers}, \cite{lurie2009higher}). Therefore we can glue $\tau_{X,-}$ over a good cover (definition \ref{good cover def}) of $X$ to obtain the Riemann-Roch transform $\tau_X:K^{alg}(\mathcal{P}(X))\to K^{top}(X)$ in the homotopy category of spectra.  

\begin{prop}
The Riemann-Roch transformation defined above commutes with restriction to open subsets. In other words, for a complex manifold $X$ and an open subset $U$ of $X$, the diagram below commutes in the homotopy category of spectra.
\begin{center}
\begin{tikzcd}
K^{alg}(\mathcal{P}(X)) \ar[d, "\tau_X"] \ar[r] & K^{alg}(\mathcal{P}(U)) \ar[d, "\tau_U"] \\
K^{top}(X) \ar[r] & K^{top}(U) 
\end{tikzcd}
\end{center}
\end{prop}

\begin{proof}
%Recall from corollary \ref{tau^d commutes with restriction to open subsets - cor} that $\tau^D_V$ commutes with restriction to further open subsets. All the other maps in $\tau$ are natural.......
Let $V_{\cdot}\to X$ be a hypercover so that all the open sets in $V_n$ are relatively compact in $V_{n-1}$. Choose a hypercover $W_{\cdot}\to U$ with the same condition as $V_{\cdot}$ so that $W_n$ is finer than $V_n \cap U$, i.e. each open set in $W_n$ (which is a relatively compact open subset of $U$) is contained in (the intersection of $U$ with) some open set in $V_n$. 
Hence for any relatively compact open set $W^j_n$ in the hypercover $W_{\cdot}$ there exists relatively compact open set $V^j_n$ of the hypercover $V_{\cdot}$ so that $W^j_n \subset V^j_n \cap U$ is relatively compact. Hence by corollary \ref{tau^d commutes with restriction to open subsets - cor} and by naturality of all the other maps in definition of $\tau$, the diagram below commutes in the homotopy category of spectra.
\begin{center}
\begin{tikzcd}
& K^{alg}(\mathcal{P}(X)) \ar[d, "\tau^D_{V^j_n}"] \ar[ldd, swap, bend right, "\tau_X"] \ar[r] 
& K^{alg}(\mathcal{P}(U)) \ar[d, "\tau^D_{W^j_n}"] \\
& K^{alg}(Ch'(\mathfrak{D}/\mathfrak{C})_{C_0(V^j_n)}) \ar[d] \ar[r] 
& K^{alg}(Ch'(\mathfrak{D}/\mathfrak{C})_{C_0(W^j_n)}) \ar[d]\\
K^{top}(X) \ar[r] & K^{top}(V^j_n) \ar[r] & K^{top}(W^j_n)
\end{tikzcd}
\end{center}
Hence there is a unique map from $K^{top}(X)$ to the $\text{holim } K^{top}(W_{\cdot})$ that makes the diagram above commute (in the homotopy category of spectra), where the homotopy limit is taken on the diagram \ref{hypercover descent equation} for the hypercover $W_{\cdot}\to U$. But this homotopy limit is weakly equivalent to $K^{top}(U)$ because topological K-homology satisfies descent. This finishes the proof.
%The assersion follows from corollary \ref{tau^d commutes with restriction to open subsets - cor} stating that $\tau^D$ commutes with restriction to further open subsets (up to homotopy), together with naturality of all the other maps in $\tau$.
\end{proof}

We will investigate functoriality of $\tau$ with respect to proper morphisms of complex manifolds in a future work.

\subsection{Cap Product}
In the last subsection, let us emphasize on the case when the $C^*$-algebra $A$ is unital, and how one could define a pairing between the K-theory and K-homology of $A$ using the Paschke category. %potentially define a more general concept of the Euler characteristic.

When $A$ is unital, we will define the \emph{Euler characteristic} of an exact sequence in the Paschke category $(\mathfrak{D} / \mathfrak{C})_A$. %, or in the Calkin-Paschke category $(\mathfrak{D} / \mathfrak{C})'_A$.
Recall that if
\begin{center}
$\ldots \rightarrow \rho _i \xrightarrow{T'_i} \rho_{i+1} \xrightarrow{T'_{i+1}} \rho_{i+2} \rightarrow \ldots$
\end{center}
is a chain complex in $(\mathfrak{D}/ \mathfrak{C})_A$, and $T_i$ are representatives for $T'_i$ in $\mathfrak{D}_A$, then the composition $T_{i+1} \circ T_i$ may not be zero; it only has to be locally compact. If the representations are all unital, then $T_{i+1} \circ T_i$ has to be compact. This is still not enough for us to be able to take the quotient of $\ker(T_{i+1})$ by the image of $T_i$, as the image may not even be a subset of the kernel. However, we can get around this issue. Recall the following definition from \cite[Sec 1.]{segal1970fredholm} and the main result of \cite{tarkhanov2007euler}. %, where we define an exact structure on the additive $C^*$-category $(\mathfrak{B}/ \mathfrak{K})$ similar to above, i.e. call a chain complex to be exact if there is a contracting homotopy.

\begin{Def}
Let $\ldots \rightarrow V_i \xrightarrow{T_i} V_{i+1} \xrightarrow{T_{i+1}} V_{i+2} \rightarrow \ldots$ be a complex of Hilbert spaces and bounded linear operators. Then it is called a \emph{Fredholm complex} if all the $T_i$'s have closed images and the cohomology is finite dimensional at every step.

Equivalently, we may define a complex of bounded operators between Hilbert spaces as before to be a Fredholm complex if there exists bounded operators $S_i:V_{i+1}\rightarrow V_i$ so that $T_{i-1} S_{i-1}+ S_i T_i - Id_{V_i}$ is a compact operator for all $i$.

If the complex is bounded, i.e. $V_i=0$ for all but finitely many values of $i$, then define its \emph{Euler characteristic} by 
\begin{center}
$\chi(V_{\cdot})= \sum_i (-1)^i H^i(V_{\cdot})$
\end{center}
We can consider the Euler characteristic as a formal difference of two finite dimensional subspaces of $\oplus_i V_i$.
\end{Def}

\begin{prop}\label{euler char well defined up to compacts-prop}
Let $ \ldots \xrightarrow{T'_{i-1}} V_i \xrightarrow{T'_i} V_{i+1} \xrightarrow{T'_{i+1}} \ldots$ be a bounded above exact sequence in $(\mathfrak{B}/ \mathfrak{K})$. Then there are morphisms $T_i\in \mathfrak{B}(V_i, V_{i+1})$ so that $T_i$ is a representative for  $T'_i$, and also $T_{i+1} \circ T_i =0$. Hence the new complex has a well-defined cohomology. Also the sequence $\ldots \xrightarrow{T_{i-1}} V_i \xrightarrow{T_{i}} V_{i+1} \xrightarrow{T_{i+1}} \ldots $ is a Fredholm complex, and if the complex is both bounded above and below, then the Euler characteristic of the complex is independent of the choices of $T_i$'s.(In the sense that for the finite dimensional subspaces $V^+,V^-, W$ of the Hilbert space $H$, we consider the formal differences $V^+ - V^-$ and $V^+ \oplus W - V^- \oplus W$ to be equivalent.) 
\end{prop}

Let $A$ be a unital $C^*$-algebra, and let
$$\ldots \rightarrow \rho _i \xrightarrow{T'_i} \rho_{i+1} \xrightarrow{T'_{i+1}} \rho_{i+2} \rightarrow \ldots$$
be a bounded exact sequence (i.e. there are only finitely many non-zero objects) in the Paschke category $(\mathfrak{D}/ \mathfrak{C})_A$. Then by the argument proving proposition \ref{ample reps cofinal subcat- prop}, we know there exists natural choices of unital representations $\hat{\rho}_i$ which are isomorphic to $\rho_i$. This induces a new exact sequence $$\ldots \rightarrow \hat{\rho} _i \xrightarrow{\hat{T}_i} \hat{\rho}_{i+1} \xrightarrow{\hat{T}_{i+1}} \hat{\rho}_{i+2} \rightarrow \ldots$$
where all the representations are unital and hence $\hat{T}_{i+1} \hat{T}_i$ is compact for all $i$, and this induces an exact sequence in $(\mathfrak{B}/ \mathfrak{K})$. Therefore by proposition \ref{euler char well defined up to compacts-prop}, the exact sequence above has a well-defined euler characteristic. %\\A precisely similar argument works for the Calkin-Paschke category $(\mathfrak{D}/ \mathfrak{C})'_A$.
Note that this process can not be replicated for the Calkin-Paschke category, as the choice of the projection $\pi\in \mathfrak{B}(H)$ corresponding to $\rho'(1)$ may affect the index. To sum it all up:
\begin{cor} \label{euler char for exact seq in pashcke cat-cor}
Let $A$ be a unital $C^*$-algebra, and let $(\rho_{\cdot}, T_{\cdot})$ be an exact sequence in the Paschke category $(\mathfrak{D}/\mathfrak{C})_A$ with finitely many non-zero objects. Then the procedure above defines the Euler characteristic of this complex. The Euler characteristic defined is additive with respect to exact sequences in the category $Ch'(\mathfrak{D}/\mathfrak{C})_A$. (This is \emph{not} true for the category $Ch^b(\mathfrak{D}/\mathfrak{C})_A$.)%, i.e. for an exact sequence $(\rho^i_{\cdot}, T^i_{\cdot})$ in $Ch'(\mathfrak{D}/\mathfrak{C})_A$ with finitely many objects, we have 
% Then one can define the Euler-characteristic of this chain complex.
\end{cor}

\begin{rmk}
When $A$ is not unital, the argument in \cite{tarkhanov2007euler} does not work anymore; as it relies on the fact that if $Id_H-\Delta \in \mathfrak{K}(H)$, then $\Delta$ has a closed image. This is no longer the case if we replace $\mathfrak{K}(H)$ by locally compact operators.
\end{rmk}

%It is tempting to try to replicate the argument for the Calkin-Paschke category, but even when $A=\mathbb{C}$ it does not necessarily work. The problem arises from the fact that there is no canonical choice for the self-adjoint projection $\pi\in \mathfrak{B}(H)$ corresponding to $\rho'(1)$ coming from lemma \ref{projection has an actual representative lemma}.  

% \begin{ex}
% Let $A=\mathbb{C}$, and consider Hilbert spaces $H_1\subset H$, and let $\pi:H\to H$ be the projection onto the subspace $H_1$. Define $\rho':A\to (\mathfrak{B}/ \mathfrak{K})(H)$ by multiplication composed with $[\pi]\in (\mathfrak{B}/ \mathfrak{K})(H)$. Consider the exact sequence $0\to \rho' \xrightarrow{[Id]} \rho' \to 0$. Then as long as $\mathfrak{C}_A(\rho)$ is not the compact operators, we can lift $[Id]\in (\mathfrak{D}/ \mathfrak{C})(\rho)$ to an operator which is not even Fredholm. If we try to switch $\rho'$ with another representation, then the choices still matter. For example, we can   
% \end{ex}

In a slightly different direction, let us define a natural pairing between projective modules and representations of a $C^*$-algebra.

\begin{Def}
Let $R$ be a ring. Denote the exact category of finitely generated projective right modules on $R$ by $\mathcal{P}^r(R)$. When $R$ is commutative, we drop the superscript $r$. Note that for any (right) projective $R$-module, there exists an integer $n$, and an inclusion $\iota:P\to R^n$ of (right) $R$-modules.

We are interested in the particular case when $R=A$ is a unital $C^*$-algebra. Let $\mathcal{P}^r_m(A)$ denote the category of finitely generated projective right $A$-modules with an inner-product structure. One can consider the inclusion $\iota:P\to A^n$ of right $A$-modules to be norm preserving.
% Let $\mathcal{P}^r_m(A)$ denote the category whose objects are pairs $(P, \iota)$, where $P$ is in $\mathcal{P}^r(A)$ and $\iota:P\to A^n$ is an inclusion in $\mathcal{P}^r(A)$. %, and two inclusions $\iota:P\to A^n, \iota':P\to A^{n'}$ are equivalent if the composition $P\to A^n\to A^{n'}$ is equal to $\iota'$, where the latter morphism is inclusion as the \emph{first} $n$-coordinates.
%pairs $(P,\iota), (P',\iota')$ are called equivalent if there is an isomorphisim of right $A$-modules $P\to P'$ so that the composition $P\to P'\to A^{n'}$ is equal to $P\to A^n \to A^{n'}$ where the last morphism is given by inclusion in the \emph{first} $n$ variables. 
Morphisms in $\mathcal{P}^r_m(A)$ are the (not necessarily norm-preserving) morphisms between the projective modules. % There is a forgetful exact functor $\mathcal{P}^r_m(A)\to \mathcal{P}^r(A)$.  

Let $X$ be a compact Hausdorf space, in particular a compact manifold. Then denote the exact category of topological (complex) vector bundles on $X$ by $\mathcal{P}^t(X)$. Recall the category $\mathcal{P}_m(X)$ from definition \ref{category of bundles with choice of metric def}. Note that by the Serre-Swan theorem \cite[Thm 2.]{swan1962vector}, $\mathcal{P}^t(X)= \mathcal{P}(C(X))$, hence $\mathcal{P}^t_m(X)=\mathcal{P}_m(C(X))$.
\end{Def}

\begin{Def}\label{pairing proj mod with rep def}
Let $A$ be a unital $C^*$-algebra and let $P\in \mathcal{P}^m_r(A)$. Let $\rho$ be an object in $ (\mathfrak{D}/\mathfrak{C})_A$. Then we define the representation $P\otimes \rho:A\to \mathfrak{B}(P\otimes_A H)$, where we are considering $H$ as a left $A$-module through the representation $\rho$.
\end{Def}

We follow \cite{atiyah1970global} to show that $P\otimes_A H$ is in fact a Hilbert space, and hence the definition above makes sense. %Choose a norm preserving inclusion $\iota :P\to A^n$ of right $A$-modules.

% For the projective module $P$, there exists $n\in \mathbb{Z}$ with an $A$-linear projection $\pi:A^{\oplus n}\to P$ together with an $A$-linear section $\iota:P\to A^{\oplus n}$.
% so that $P\oplus Q \cong A^{\oplus n}$. We denote the inclusion by $\iota:P\to A^{\oplus n}$ and the projection by $\pi:A^{\oplus n}\to P$. 
% Following \cite{atiyah1970global}, we will define the representation $P\otimes \rho$. %:A\to \mathfrak{B}(H^{\oplus n}) \cong \mathfrak{B}(A^{\oplus  n}\otimes_A  H)$.%%%%%%%%%%%%%%%%%%%%%%%%%%%%%%%%%%%%%%%%%%%%%%%%% WHY ?????   %%%%%%%%%%%%%%%%%%%%%%%%%%%%%%%%%%%%%%%%%%%%%%%%%%%%%%%%%%%%%%%%%%%%%%
Since $P$ is a finitely generated projective (right) module, there exists a norm preserving (right) $A$-module surjection $\pi:A^n\to P$, with a norm preserving (right) $A$-module section $\iota:P\to A^n$. Without loss of generality, we can assume that $\iota \pi$ is a self-adjoint projection on $A^{\oplus n}$. Now let $\iota \pi (e_i)=\sum_j e_j a^j_i$ for $ 1\leq i\leq n,$ 
where $e_1,\ldots,e_{n}$ are the standard basis for $A^{ n}$, and $a^j_i\in A$. 
Define the linear operator $\hat{\pi}$ on $H^{\oplus n} \cong A^{ n} \otimes_A H$ by 
$$\hat{\pi}(e_i\otimes_A h) = \sum _j e_j \otimes_A \rho(a^j_i)h.$$ 
It is easy to check that this is in fact a self-adjoint projection, and since $\iota \pi$ is $A$-linear, then $\hat{\pi}$ also commutes with $\rho^{\oplus n}$, because
\begin{align*}
& \hat{\pi} \rho^{\oplus n}(a)(h_1, \ldots, h_{n}) = \sum_i \hat{\pi} (e_i\otimes_A \rho(a)h_i)  
=  \sum_{i,j} (e_j\otimes_A \rho(a^j_i) \rho(a)h_i) = \sum_{i,j}(e_j a^j_i)\otimes_A \rho(a)h_i \\
&= \sum_i \pi(e_i) \otimes_A \rho(a)h_i = \sum_i  \pi(e_i.a)  \otimes_A h_i 
= \sum_i  \pi(a.e_i) \otimes_A h_i =  \rho^{\oplus n}(a) \sum_i \pi(e_i) \otimes_A h_i \\
&= \rho^{\oplus n}(a) \sum_i \hat{\pi} (e_i\otimes_A h_i) = \rho^{\oplus n}(a) \hat{\pi} (h_1,\ldots,h_{n}).
\end{align*} 
Now let $V \subset H^{\oplus n}$ be the image of $\hat{\pi}$, and let $\hat{\iota}:V \to H^{\oplus n}$ be the inclusion, then consider the composition $A\xrightarrow{\rho^{\oplus n}} \mathfrak{B}(H^{\oplus n}) \to \mathfrak{B}(V)$, where the last map sends $T$ to $\hat{\pi} T \hat{\iota}$ (we are abusing the notation and denoting the composition of $\hat{\pi}$ with the orthogonal projection $H^{\oplus n} \to V$ by $\hat{\pi}$ as well.). %We just showed that this is a $*$-morphism and hence we have a new representation of $A$ which we denote by $\widehat{P\otimes \rho}$. 
It is easy to check that the compositions below are inverses to each other.
\begin{equation*}
\begin{array}{c}
V\xrightarrow{\hat{\iota}} H^{\oplus n} \cong A^n \otimes_A H \xrightarrow{\pi \otimes_A Id} P \otimes_A H  \\ 
V\xleftarrow{\hat{\pi}} H^{\oplus n} \cong A^n \otimes_A H \xleftarrow{\iota \otimes_A Id} P \otimes_A H  
 \end{array}
\end{equation*}
%P \otimes_A H\xrightarrow{\iota\otimes_A Id} A^n \otimes_A H \cong H^n \xrightarrow{\hat{\pi}} V 
Since all the maps above commute with multiplication by $A$ in $\mathfrak{D}_A$, this induces structure of a Hilbert space on $P\otimes _A H$.% It is clear that the choice of the integer $n$ (hence choice of $\iota$) does not affect this process.

\begin{prop} \label{pairing k theory proj mods with k theory pashcke cat-prop}
Let $A$ be a unital $C^*$-algebra. The tensor product introduced in definition \ref{pairing proj mod with rep def}, induces a biexact functor 
\begin{equation} \label{pairing pashcke cat with proj modules-equation}
\cap_A: \mathcal{P}_m^r(A) \times (\mathfrak{D}/ \mathfrak{C})_A \to (\mathfrak{D}/ \mathfrak{C})_A
\end{equation}
which we will call the \emph{cap product}. This induces a pairing on the level of K-theory spectra 
% which in turn induce a pairing on the level of K-theory
\begin{equation} 
\cap_A: K^{alg}(\mathcal{P}_m^r(A)) \wedge K((\mathfrak{D}/ \mathfrak{C})_A) \to K((\mathfrak{D}/ \mathfrak{C})_A).
\end{equation}

\end{prop}

\begin{proof}
First we need to check functoriality. Let $F:P_1\to P_2 $ be a morphism in $\mathcal{P}^r_m(A)$. Then we can consider the morphism $F\otimes_A Id: P_1 \otimes_A \rho \to P_2 \otimes_A \rho$ in the category $\mathfrak{D}_A$ where $p\otimes_A h \mapsto F(p) \otimes_A h$. Note that for $a\in A$, $F(p) \otimes_A \rho(a)h = F(p).a \otimes_A h = F(p.a) \otimes_A h$. Hence this is well defined and commutes with multiplication by $A$. It is clear that this process is functorial, i.e. $(F_2\otimes_A Id) \circ (F_1\otimes_A Id) = F_2 F_1 \otimes_A Id$ in $\mathfrak{D}_A$.  

Let $T:\rho_1\to \rho_2$ be a morphism in the Paschke category $(\mathfrak{D}/\mathfrak{C})_A$. Then we define $Id\otimes_A T : P\otimes_A \rho_1 \to P\otimes_A \rho_2$ as follows. Let $\pi:A^n\to P$ be a norm preserving surjective map of right $A$-modules, let $\iota:P\to A^n$ be the corresponding inclusion of right $A$-modules, and let $\hat{\pi}_i \in \mathfrak{B}(H_i^{\oplus n}) $ be the projection corresponding to $P \otimes_A H_i$ for $i=1,2$, and $\hat{\iota}_i$ be the inclusion of its image $V_i$ in the corresponding Hilbert space. Consider $T^{\oplus n}: A^{n} \otimes_A \rho_1 \cong \rho_1^{\oplus n} \to \rho_2^{\oplus n} \cong A^{n} \otimes_A \rho_2 $, and define $Id\otimes _A T= \hat{\pi}_2 T^{\oplus n} \hat{\iota}_1$. Note that this is a pseudo-local operator, as both $\hat{\pi}_2,\hat{\iota}_1$ commute with the representations and $T$ is pseudo-local. Also, since $T$ commutes with multiplication by $a_i^j \in A$ modulo compact operators, then $\hat{\pi}_2 T^{\oplus n}- T^{\oplus n}\hat{\pi}_1$ is compact. This shows that the definition for $Id\otimes_A T$ is independent of the choice of the projection $\pi$ up to locally compact operators. 
% the projections $\hat{\pi}_i$ commute with $Id_n \otimes_A T$ modulo compact operators as well, i.e. 
This also shows that for morphisms $T_1:\rho_1\to \rho_2$ and $T_2:\rho_2\to \rho_3$ in the Paschke category, the compositions $(Id\otimes_A T_2)\circ (Id \otimes_A T_1)= \hat{\pi}_3 T_2^{\oplus n} \hat{\iota}_2 \hat{\pi}_2 T_1^{\oplus n} \iota_1= \hat{\pi}_3 T_2^{\oplus n} T_1^{\oplus n} \hat{\pi}_1 \hat{\iota}_1 $ are equal to each other modulo locally compact operators. Hence this process is functorial. 

\begin{rmk}
The map $Id\otimes_A T: P\otimes _A \rho_1\to P\otimes_A \rho_2$ is \emph{not} well-defined in the category $\mathfrak{D}_A$.
\end{rmk}

% This construction is naturally functorial in both variables. 
Let $\rho_1,\rho_2$ be two objects of $(\mathfrak{D}/ \mathfrak{C})_A$, let $T:\rho_1\to \rho_2$ be a morphism, let $P_1, P_2,$ be two objects of $\mathcal{P}^r_m(A)$, and let $F:P_1\to P_2$ be a morphism of right $A$-modules. Choose the norm preserving inclusions $\iota_i:P_i\to A^{n_i}$ of right $A$-modules, so that there exist a map of right $A$-modules $F': A^{n_1}\to A^{n_2}$ that makes the corresponding diagram commute. Then the square on the left and the one on the right commutes in the category $\mathfrak{D}_A$. Consider $e_i\otimes_A h \in A^{n_1} \otimes_A H $, we have $(F'\otimes_A Id) T^{\oplus n_1}(e_i\otimes h)= (b_{1,i}T(h), \ldots, b_{n_2,i}T(h))$, where $b_{j,i}$ is the $j$'th term in $F'(e_i)\in A^{n_2}$, and $T^{\oplus n_2}(F' \otimes _A Id)(e_i)= \left(T(b_{1, i} h), \ldots, T(b_{n_2 ,i}h) \right)$. Since $T$ is pseudo-local, then the square in the center also commutes in the Paschke category $(\mathfrak{D}/ \mathfrak{C})_A$. Hence functoriality in two directions are compatible with each other.
 
% Then we have the commutative diagram below in the Paschke category $(\mathfrak{D}/\mathfrak{C})_A$ % \footnote{%Even though the two compositions applied on each element of the Hilbert space $P_1\otimes_A H_1$ seem to agree with each other, 
% The morphism $F\otimes_a Id$ can be considered as a morphism in $\mathfrak{D}_A$, however 
% the morphism $Id \otimes _A T$ is \emph{not} well defined as a map of $A$-modules, as $T$ does not commute with multiplication by $A$ (only commutes modulo compact operators), hence can be only considered as a morphism in the Paschke category $(\mathfrak{D}/\mathfrak{C})_A$.} 
%In \cite{atiyah1970global}, Atiyah is working with the pseudo-local operators, i.e. in the category $\mathfrak{D}_A$, where he can not work with the tensor product $Id\otimes_A T$. Therefore he keeps working with the Hilbert space $V$ defined above, which depends on the choices of $\iota,\pi,n$, and hence is not functorial.}.
%Let $n_{kl}, \pi_{kl} , \iota_{kl}, V_{kl}$ correspond to $P_k\otimes \rho_l$ for $k,l=1,2$. Then we have a morphism $\pi_{12} (Id\otimes T) \iota_{12} :P_1\otimes \rho_1\to P_1\otimes \rho_2$ Then the following diagram commutes in the Paschke category $(\mathfrak{D}/ \mathfrak{C})_A$.
\begin{center}
\begin{tikzcd}
P_1 \otimes \rho_1 \ar[rd,  "\hat{\iota}_1"] \ar[ddd, swap, "F\otimes_A Id"] \ar[rrr, "Id\otimes_A T"] & &
& P_1 \otimes_A \rho_2 \ar[ddd, "F\otimes_A Id"]\\
& \rho_1 ^{\oplus n_1} \ar[d, swap, "F' \otimes_A Id "] \ar[r, "T^{\oplus n_1}"] 
& \rho_2^{\oplus n_1} \ar[d, "F' \otimes_A Id "] \ar[ur, "\hat{\pi}_2"] & \\
& \rho_1 ^{\oplus n_2}  \ar[r, "T^{\oplus n_2}"] 
& \rho_2^{\oplus n_2} \ar[dr, "\hat{\pi}_2"] & \\
P_2 \otimes \rho_1 \ar[ru, "\hat{\iota}_1"] \ar[rrr, "Id \otimes_A T"] & & 
& P_2 \otimes_A \rho_2 
\end{tikzcd}
\end{center}

It is easy to check that if $T$ is invertible, then so is $Id\otimes_A T$, and if $F$ is an isomorphism of $A$-modules, then $F\otimes_A Id$ is a pseudo-local isomorphism of Hilbert spaces. This procedure is exact in both variables, because if $T_2\circ T_1=0$ then $(Id \otimes _A T_2)\circ (Id \otimes _A T_1)=0$ and similarly for $F$. Also, an exact sequence $(\rho_{\cdot}, T_{\cdot})$ in $(\mathfrak{D}/\mathfrak{C})_A$ has a contracting homotopy $S_{\cdot}$, which translates into a contracting homotopy $Id\otimes_A S_{\cdot}$ for the sequence $(P\otimes_A \rho_{\cdot}, Id\otimes_A T_{\cdot})$. Also a short exact sequence $(P_{\cdot}, F_{\cdot})$ of projective modules splits, i.e. has a contracting homotopy which again gives a contracting homotopy for the sequence $(P_{\cdot}\otimes_A \rho, F_{\cdot} \otimes_A Id)$. %Let us summarize the discussion above in the following corollary.

Let $F_1:P_1\to P_2$ be an admissible monomorphism in $\mathcal{P}_m^r(A)$ and consider an exact sequence
\begin{center}
\begin{tikzcd}
0 \ar[r ] & \rho_1 \ar[r, "T_1"] & \rho_2 \ar[r, "T_2"] \ar[l, bend left=20, "S_1"] & \rho_3 \ar[l, bend left=20, "S_2" ] \ar[r] & 0 
\end{tikzcd}
\end{center}  %$0\to \rho_1 \xrightarrow{T_1} \rho_2 \xrightarrow{T_2} \rho_3 \to 0$ 
% admissible monomorphism $T:\rho_1\to \rho_2$ 
with a choice of contracting homotopy in the Paschke category $(\mathfrak{D}/\mathfrak{C})_A$. Then there is a map $$P_1\otimes_A \rho_3 \oplus P_2\otimes_A \rho_1 \xrightarrow{(Id\otimes_A S_2) \oplus Id } P_1\otimes_A \rho_2 \oplus P_2\otimes_A \rho_1 \to (P_1\otimes_A \rho_2)\cup_{(P_1\otimes_A \rho_1)} (P_2\otimes_A \rho_1)$$ 
which induces an isomorphism between the first object and the last object in the Paschke category $(\mathfrak{D}/\mathfrak{C})_A$. The map  $P_1\otimes_A \rho_3 \oplus P_2\otimes_A \rho_1 \xrightarrow{ ( F\otimes_A S_2 , Id\otimes_A T_1)^{\text{t}} } P_2\otimes_A \rho_2 $ is an admissible monomorphism, whose cokernel is $P_2\otimes_A \rho_2 \xrightarrow{F_2\otimes_A T_2} P_3\otimes \rho_3$, where $F_2:P_2\to P_3$ is cokernel of $F_1$, and the contracting homotopies are the trivial ones induced by contracting homotopies of $\rho_{\cdot}$ and $ F_{\cdot}$. 
This proves that $\cap_A$ is biexact, hence by proposition \ref{biexact product waldhausen k theory prop} induces a map of K-theory spectra.
\end{proof}

Let $f:A\to B$ be a unital map between unital $C^*$-algebras. Recall there is an exact \emph{push-forward} functor $f_*: \mathcal{P}^r(A)\to \mathcal{P}^r(B)$ and $f_*:\mathcal{P}_m^r(A)\to \mathcal{P}_m^r(B)$ defined by $f_*(P)= P\otimes _A B$. There is also a pull-back map $f^*:(\mathfrak{D}/ \mathfrak{C})_B \to (\mathfrak{D}/ \mathfrak{C})_A$. One could ask about the relation between the pairing defined above %in proposition \ref{pairing k theory proj mods with k theory pashcke cat-prop} 
and these structures.

\begin{prop}\label{pairing naturality prop}
The pairing defined in proposition \ref{pairing k theory proj mods with k theory pashcke cat-prop} is natural in the sense that for a unital map $f:A\to B$ of unital $C^*$-algebras, the diagram below commutes up to homotopy
\begin{center}
\begin{tikzpicture}[commutative diagrams/every diagram]
\node (P0) at (90:2.8cm) {$K(\mathcal{P}_m^r(A))\wedge K((\mathfrak{D}/ \mathfrak{C})_B)$};
\node (P1) at (90+72:2.5cm) {$K(\mathcal{P}_m^r(B))\wedge K((\mathfrak{D}/ \mathfrak{C})_B)$} ;
\node (P2) at (90+2*72:2.5cm) {\makebox[5ex][r]{$K((\mathfrak{D}/ \mathfrak{C})_B)$ }};
\node (P3) at (90+3*72:2.5cm) {\makebox[5ex][l]{$ K((\mathfrak{D}/ \mathfrak{C})_A)$.}};
\node (P4) at (90+4*72:2.5cm) {$K(\mathcal{P}_m^r(A))\wedge K((\mathfrak{D}/ \mathfrak{C})_A)$};
\path[commutative diagrams/.cd, every arrow, every label]
(P0) edge node[swap] {$f_*\times Id$} (P1)
(P1) edge node[swap] {$\cap_B$} (P2)
(P2) edge node {$f^*$} (P3)
(P4) edge node {$\cap_A$} (P3)
(P0) edge node {$Id \times f^*$} (P4);
\end{tikzpicture}
\end{center}

\end{prop}

\begin{proof}
Consider the diagram below
\begin{center}
\begin{tikzpicture}[commutative diagrams/every diagram]
\node (P0) at (90:2.3cm) {$\mathcal{P}_m^r(A)\times (\mathfrak{D}/ \mathfrak{C})_B$};
\node (P1) at (90+72:2cm) {$\mathcal{P}_m^r(B)\times (\mathfrak{D}/ \mathfrak{C})_B$} ;
\node (P2) at (90+2*72:2cm) {\makebox[5ex][r]{$(\mathfrak{D}/ \mathfrak{C})_B$ }};
\node (P3) at (90+3*72:2cm) {\makebox[5ex][l]{$ (\mathfrak{D}/ \mathfrak{C})_A$.}};
\node (P4) at (90+4*72:2cm) {$\mathcal{P}_m^r(A)\times (\mathfrak{D}/ \mathfrak{C})_A$};
\path[commutative diagrams/.cd, every arrow, every label]
(P0) edge node[swap] {$f_*\times Id$} (P1)
(P1) edge node[swap] {$\cap_B$} (P2)
(P2) edge node {$f^*$} (P3)
(P4) edge node {$\cap_A$} (P3)
(P0) edge node {$Id \times f^*$} (P4);
\end{tikzpicture}
\end{center}
Let $\rho:B\to \mathfrak{B}(H)$  be a representation, and let $P$ be an object in $\mathcal{P}_r(A)$. We can consider $H$ as a left $A$-module through the representation $f^*\rho:A\to B\to \mathfrak{B}(H)$. It is straightforward to check that the natural map of Hilbert spaces $P\otimes_A H\to (P\otimes_A B)\otimes_B H$ defined by $p\otimes_A h\mapsto (p\otimes_A 1)\otimes_B h$ is well-defined, and has a two-sided inverse given by $(p\otimes_A b)\otimes_B h \mapsto p \otimes_A \rho(f(b))h$. This isomorphism is pseudo-local, hence induces a natural isomorphism between $f^*\left( (P\otimes_A B) \otimes_B \rho \right)$ and $P\otimes_A f^*\rho$ in the category $\mathfrak{D}_A$. Hence the diagram above commutes up to natural isomorphisms.%, and a similar diagram on the level of K-theory commutes up to homotopy.
\end{proof}

\begin{rmk}
One can replace the Paschke category $(\mathfrak{D}/\mathfrak{C})_A$ with the category $Ch^b(\mathfrak{D}/\mathfrak{C})_A$ (or $C(\mathfrak{D}/\mathfrak{C})_A,C^b(\mathfrak{D}/\mathfrak{C})_A,Ch(\mathfrak{D}/\mathfrak{C})_A$) in propositions \ref{pairing k theory proj mods with k theory pashcke cat-prop} and \ref{pairing naturality prop} and the same result would still hold. However, we can not necessarily replace $Ch'(\mathfrak{D}/\mathfrak{C})_A$, as morphisms in $Ch'(\mathfrak{D}/\mathfrak{C})_A$ come from $\mathfrak{D}_A$, but pairing a morphism with the identity on a projective module is only well-defined up to compact operators.
\end{rmk}

Fix an object $(\rho_{\cdot}, T_{\cdot})$ of $Ch'(\mathfrak{D}/ \mathfrak{C})_A$, since for a morphism $F:P_1\to P_2$ in $\mathcal{P}_m^r(A)$, the morphism $F\otimes_A Id : P_1 \otimes_A \rho_{\cdot} \to P_2 \otimes _A \rho_{\cdot}$ is well-defined in $\mathfrak{D}_A$, hence we obtain a functor 
\begin{equation}\label{map from k thry proj mods to k thry chn cxs - equation}
\mathcal{P}_m^r(A) \to  Ch'(\mathfrak{D}/\mathfrak{C})_A
\end{equation}  
which maps $P$ to $ P \cap_A (\rho_{\cdot},T_{\cdot})= (P\otimes_A \rho_{\cdot}, Id\otimes_A T_{\cdot})$.

\begin{Def}
Let $X$ be a compact complex manifold, let $g$ be a hermitian metric on $X$, let $X\times \mathbb{C}$ denote the trivial rank one bundle on $X$, and let $E$ be a \emph{topological} vector bundle on $X$. Then denote the map \ref{map from k thry proj mods to k thry chn cxs - equation} obtained through pairing with $\tau^D_{X,g}(X\times \mathbb{C}) \in Ch'(\mathfrak{D}/ \mathfrak{C})_{C(X)} $ defined in \ref{completed-Dolbeault-complex}, by $-\cap \tau^D [X]$.  
%\begin{equation*} \end{equation*}
\end{Def}

We have to emphasize that for a non-holomorphic vector bundle, the Dolbeault complex is not well-defined.
Let $X$ be a compact complex manifold, let $E$ be a holomorphic vector bundle on $X$, and let $g,h$ be hermitian metrics on $X,E$, respectively. 
Recall from \ref{completed Dolbeault complex def} that we have an exact sequence $\tau^D_{X,g}(E,h)$ in the Paschke category $(\mathfrak{D}/\mathfrak{C})_{C(X)}$ corresponding to the Dolbeault complex. %, where $X\times \mathbb{C}$ is the trivial rank $1$ bundle on $X$. 
% Therefore by corollary \ref{euler char for exact seq in pashcke cat-cor}, we can define the Euler characteristic of this sequence. 
\begin{prop} \label{pairing holomophic bundle agrees with original - prop}
Let $X$ be a compact complex manifold, and let $E$ be a holomorphic vector bundle on $X$. Choose hertmitian metrics $g,h$ on $X,E$ respectively. Then the chain complexes $\tau^D_{X,g}(E,h)$ and $E\cap \tau^D[X]$ are isomorphic to each other in the category $Ch'(\mathfrak{D}/ \mathfrak{C})_{C(X)}$.
\end{prop}

\begin{proof}
Let $\pi: X\times \mathbb{C}^m\to E$ be a smooth projection onto the bundle $E$, and let $\iota:E \to X\times \mathbb{C}^m$ be the inclusion. %Let $E'$ be the orthogonal complement of $\iota(E)$ in $X\times \mathbb{C}^m$. 
Denote the Dolbeault operator on the trivial bundle $X\times \mathbb{C}^k$ of rank $k$ by $D^k= \bar{\partial} + \bar{\partial}^*$, and let $D_E= \bar{\partial}_E + \bar{\partial}^*_E$ denote the Dolbeault operator on $E$.  Note that by definition, $Id\otimes \chi(D^1)\in \mathfrak{B}( E\otimes L^2(X,\wedge^{0,*}T^*X))$ is defined as $\pi \chi(D^1)^{\oplus m} \iota= \pi \chi(D^m) \iota$. By remark \ref{functoriality tau^D continuous is enough -rmk}, $\pi \chi(D^m) \iota - \pi \iota \chi(D_E)$ is locally compact in the Paschke category $(\mathfrak{D}/ \mathfrak{C})_{C(X)}$. Since $\pi \iota = Id_E$, this proves the assertion.
\end{proof}

\begin{cor}
Let $X$ be a compact complex manifold, and let $E$ be a topological vector bundle. Then $E\cap \tau^D_{X,g}$ is an exact sequence in the Paschke category $(\mathfrak{D}/ \mathfrak{C})_{C(X)}$, and by corollary \ref{euler char for exact seq in pashcke cat-cor} has a well-defined Euler characteristic. By propositions \ref{pairing holomophic bundle agrees with original - prop} and \ref{completed dolbeault cx cmpt mnfld q.i. to dolb cx prop}, this concept of Euler characteristic is equal to the classical concept of the Euler characteristic of the Dolbeault complex when $E$ is a holomorphic vector bundle.
\end{cor}

\begin{appendices}
\section{Complex Manifolds}

Let us give the basics and the notation used for complex manifolds in here. A good source for reading more on the topic is \cite{wells2007differential}. %A reader familiar with the subject can skip this subsection. 
% Let $X$ be an $\mathscr{S}$-manifold \footnote{See \cite[1.1.1, 2.1.5.]{wells2007differential}. }, and let $E$ be an $\mathscr{S}$-vector bundle on $X$. Then denote the sheaf of $\mathscr{S}$-sections of $E$ by $\mathscr{S}(E)$. Also, for an open subset $U$ of $X$, denote the space of sections of $E$ on $U$ by $\mathscr{S}(U,E)$. In the case when $E$ is the trivial line bundle $X\times \mathbb{C}$, then we will just denote $\mathscr{S}(U)$ instead of $\mathscr{S}(U,E)$ and also denote the structure sheaf by $\mathscr{S}_X$.

Let $X$ be a complex manifold and let $E$ be a holomorphic vector bundle, then denote the sheaf of holomorphic, real analytic, differentiable, and continuous sections of $E$ by $\mathscr{O}(E), \mathscr{C}^{\omega}(E), \mathscr{C}^{\infty}(E), \mathscr{C}(E)$, respectively. Notice that each of the four sheaves just mentioned, is a subsheaf of the next ones. Also if $X$ is only real analytic, then we can still consider the sheaves $\mathscr{C}^{\omega}(E), \mathscr{C}^{\infty}(E), \mathscr{C}(E)$, and similar statements can be repeated for differentiable or topological manifolds. Let $\mathscr{S}$ be one of the four sheaves above, then for an open subset $U$ of $X$, denote the space of sections of $E$ on $U$ by $\mathscr{S}(U,E)$. In the case when $E$ is the trivial line bundle $X\times \mathbb{C}$, then we will just denote $\mathscr{S}(U)$ instead of $\mathscr{S}(U,E)$ and also denote the structure sheaf by $\mathscr{S}_X$.

% Note that if $E,F$ are $\mathscr{S}$-bundles on the $\mathscr{S}$-manifold $X$, and if $\mathscr{S}$ is a subsheaf of $\mathscr{T}$, then $\mathscr{T}(F)\otimes_{\mathscr{S}} \mathscr{S}(E) \cong \mathscr{T}(F\otimes E)$.

Let $T^*X$ denote the cotangent bundle of the complex manifold $X$. Then the (almost) complex structure of $X$ induces the decomposition $T^*X\otimes_{\mathbb{R}} \mathbb{C} = T^*(X)^{1,0}\oplus T^*(X)^{0,1}$, which in turn induces the \emph{Dolbeault operator} $\bar{\partial}: \mathscr{C}^{\infty}(\wedge^{p,q}T^*X) \to \mathscr{C}^{\infty}(\wedge^{p,q+1} T^*X)$, that vanishes on the holomorphic sections. Hence for a holomorphic vector bundle $E$, we get an induced differential operator 
\begin{equation*}
\bar{\partial} \otimes 1: \mathscr{C}^{\infty}(\wedge^{p,q}T^*X)\otimes_{\mathscr{O}} \mathscr{O}(E) \to \mathscr{C}^{\infty}(\wedge^{p,q+1} T^*X) \otimes_{\mathscr{O}} \mathscr{O}(E),
\end{equation*}
which is also known as the Dolbeault operator. But we have $\mathscr{C}^{\infty}(\wedge^{p,q}T^*X)\otimes_{\mathscr{O}} \mathscr{O}(E) \cong \mathscr{O}(\wedge^{p,q} T^*X\otimes_{\mathbb{C}} E)$. From now on, we will abbreviate the latter to $\mathscr{A}_X^{p,q}(E)$ (or just $\mathscr{A}^{p,q}(E)$, if $X$ is clear from the context.), and we denote the Dolbeault operator by $\bar{\partial}_E: \mathscr{A}^{p,q}(E)\to \mathscr{A}^{p,q+1}(E)$. We will also call the following as the \emph{Dolbeault complex} with coefficients in $E$:
\begin{equation} \label{Dolbeault complex}
0 \to \mathscr{A}_X^{0,0}(E) \xrightarrow{\bar{\partial}_E} \mathscr{A}_X^{0,1}(E) \xrightarrow{\bar{\partial}_E} \mathscr{A}_X^{0,2}(E) \xrightarrow{\bar{\partial}_E} \ldots \xrightarrow{\bar{\partial}_E} \mathscr{A}_X^{0,n}(E) \to 0
\end{equation}
where in here, $n=dim_{\mathbb{C}}(X)$.

\begin{Def} 
We follow \cite[4.2.]{wells2007differential} to recall the definition of \emph{symbol} of a differential operator. First let $X$ be a differentiable manifold, and consider differentiable vector bundles\footnote{\emph{All} the vector bundles and vector spaces we are considering in this section are over the complex numbers. Some of the arguments still hold over the real numbers as well.} $E,F$ on $X$. A linear operator $D:\mathscr{C}^{\infty}(X,E)\to \mathscr{C}^{\infty}(X,F)$ is a \emph{differential operator of order $k$}, if no derivations of order $\geq k+1$ appear in its local representation. We denote the vector space of all such operators with $\text{Diff}_k(E,F)$. 

Let $T'X$  denote the cotangent bundle $T^*X$ of $X$ with the zero section deleted, and let $\pi:T'X\to X$ denote the projection. For $k\in \mathbb{Z}$ set
\begin{equation*}
\text{Smbl}_k(E,F) \coloneqq \{\sigma \in Hom(\pi^*E,\pi^*F): \sigma(x,\rho v) = \rho^k \sigma(x,v), \text{ where } (x,v)\in T'X, \rho > 0 \}.
\end{equation*} 
We now define the \emph{$k$-symbol of a differential operator} as a linear map $\sigma_k: \text{Diff}_k(E,F) \to \text{Smbl}_k(E,F)$ by 
\begin{equation*}
\sigma_k(D)(x,v)e= D(\frac{j^k}{k!} (g-g(x))^k f)(x)\in F_x 	
\end{equation*}
where in here $(x,v)\in T'X,  e\in E_x$ are given and $g\in \mathscr{C}^{\infty}(X), f\in \mathscr{C}^{\infty}(X,E)$ are chosen so that $f(x)=e, dg_x=v$. We can see that we have a linear mapping $\sigma_k(D)(x,v):E_x\to F_x$, and that the symbol does not depend on the choices made.
\end{Def}
One can also define \emph{pseudo-differential operator} of order $k$ for $k\in \mathbb{Z}$ (which we will denote by $\text{PDiff}_k$), and their symbol, but since definitions are somewhat more technical, and will not be used here, we refer the interested reader to \cite[4.3.]{wells2007differential}.

Symbols of (pseudo-) differential operators have the following important properties: 
\begin{equation*}
\begin{array}{c}
\sigma_{k+m}(D_2 D_1) = \sigma_m(D_2) \sigma_k(D_1) \text{ when } D_1\in \text{PDiff}_k(E_1,E_2) , D_2\in \text{PDiff}_m(E_2,E_3) \\
\sigma_k(D^*)= (-1)^k \sigma_k(D)^* \text{ if } D\in \text{PDiff}_k(E,F)
\end{array}
\end{equation*}
where in here $D^*\in \text{PDiff}_k(F,E)$ is the \emph{formal adjoint} of $D$ \cite[4.1.5.]{wells2007differential}, and $\sigma_k(D)^*$ is the adjoint of the linear map $\sigma_k(D)(x,v):E_x\to F_x$. Note that both $D^*$ and $\sigma(D^*)=\sigma(D)^*$ \emph{depend} on the choice of metric on $X$ and the bundles $E,F$.

\begin{Def}\cite[4.4.]{wells2007differential}
Let $E,F$ be differentiable vector bundles on the differentiable manifold $X$ and let $D\in \text{Diff}_k(E,F)$. Then we say that $D$ is an \emph{elliptic differential operator} if for all $(x,v)\in T'X$, the linear map $\sigma_k(D)(x,v):E_x\to F_x$ is an isomorphism. In particular both $E,F$ have the same fiber dimension. The same can be defined for pseudo-differential operators.

Let $E_0,E_1, \ldots , E_m$ be a sequence of differentiable vector bundles on $X$ and for some fixed $k$, let $D_i\in \text{Diff}_k(E_i,E_{i+1})$ for all $i=0,1,\ldots, m-1$. We say this is an \emph{elliptic complex} if $D_{i+1}\circ D_i = 0$ for all $i$, and also if the associated symbol sequence
\begin{equation*}
0 \to \pi^*E_0 \xrightarrow{\sigma_k(D_0)} \pi^*E_1 \xrightarrow{\sigma_k(D_1)} \pi^* E_2 \xrightarrow{\sigma_k(D_2)} \ldots \xrightarrow{\sigma_k(D_{m-1})} \pi^*E_m \to 0
\end{equation*}
is exact, where $\pi:T'X\to X$ is the projection. 
\end{Def}

\begin{rmk}
In the literature, elliptic complexes are usually defined for \emph{compact} differentiable manifolds, since \emph{Sobelov spaces} over non-compact spaces don't behave as well as they do on compact spaces (e.g. Rellich's lemma works for Sobelov spaces over a fixed compact subset of the manifold.), which makes elliptic complexes over non-compact manifolds not as easy to work with. For example, the Hodge decomposition theorem (mentioned later in this section) is no longer true for non-compact complex manifolds.
\end{rmk}

\begin{ex}\label{Dolbeault complex symbol seq example}
Let $E$ be a holomorphic vector bundle on the complex manifold $X$. The Dolbeault operator $\bar{\partial}_E:\mathscr{A}^{p,q}_X(E)\to \mathscr{A}^{p,q+1}_X(E)$ is a differential operator of order $1$, and for $(x,v)\in T'X$, and $f\otimes e\in \wedge ^{p,q}T_x^*X\otimes E$, 
\begin{equation*}
\sigma_1(\bar{\partial}_E)(x,v)f\otimes e = (iv^{0,1}\wedge f) \otimes e
\end{equation*}
where in here $v=v^{1,0} + v^{0,1}\in T^*_x(X)^{1,0} + T_x^*(X)^{0,1}$. It is easy to check that the symbol sequence is exact, and hence the Dolbeault complex \ref{Dolbeault complex} is an elliptic complex. 

% Now, choose a hermitian metric $g$ on $X$ and a hermitian metric on $E$. These induce a hermitian metric on $\wedge^{p,q}T^*X\otimes E$. If $\bar{\partial}^*_E$ is the formal adjoint (with respect to $g,h$) of the Dolbeault operator $\bar{\partial}_E$, then 
% \begin{equation*}
% \sigma_1(\bar{\partial}^*_E)(x,v)f\otimes e = \left(\sigma_1(\bar{\partial}_E)(x,v)f\otimes e\right)^*=  (-i (Re(g)^{-1} v) \lrcorner f)\otimes e 
% \end{equation*}
% where in here, $Re(g)=\frac{g+\bar{g}}{2}$ is the Riemannian metric corresponding to $g$, considered as a (real) linear map $Re(g): TX\to T^*X$, and $(Re(g)^{-1} v) \lrcorner f$ is interior product of $f$ with the tangent vector $a$.\footnote{For more explanation on the computation, see ...... FIND THE ORIGINAL REFERENCE OR SOMETHING!} We can see that this symbol depends on the choice of the metric $g$ on $X$, but it does not depend on the choice of the metric $h$ on $E$.  
\end{ex}

\begin{thm}[The Hodge decomposition]\label{hodge-decomposition theorem}
Let $X$ be a \emph{compact} complex manifold, and let $E$ be a holomorphic vector bundle on $X$. Choose hermitian metrics on $X$ and on $E$ and let $\bar{\partial}_i^*: \mathscr{A}_X^{0,i+1}(E)\to \mathscr{A}_X^{0,i}(E)$ be the formal adjoint of $\bar{\partial}_i: \mathscr{A}_X^{0,i}(E) \to \mathscr{A}_X^{0,i+1}(E)$ (with respect to the metrics chosen). Let $\Delta_i= \bar{\partial}_{i-1} \bar{\partial}_{i-1}^* + \bar{\partial}_i^* \bar{\partial}_i$ and let $\mathscr{H}^{0,i}(X,E) = \ker \Delta_i \subset \mathscr{A}^{0,i}_X(E)$ denote the \emph{harmonic} $(0,i)$-forms. Then we have the orthogonal decomposition
\begin{equation}
\mathscr{A}^{0,i}_X(E) \cong \mathscr{H}^{0,i}(X,E) \oplus \text{im}(\bar{\partial}_{i-1}) \oplus \text{im}(\bar{\partial}^*_i),  
\end{equation}
and also there is an isomorphism $\mathscr{H}^{0,i}(X,E) \cong H^{i}(X,E)$, where the latter, is the cohomology of $X$ with coefficients in $E$.
\end{thm}

\begin{Def}[Hodge Star operator]\cite[5.1.]{wells2007differential}\label{hodge-star operator definition}
Let $V$ be a (complex) vector space of dimension $n$. Choose an inner product on $V$ and then choose an orthonormal basis $e_1, \ldots , e_n$ for $V$. Then define the \emph{Hodge $*$-operator}
\begin{equation*}
\star : \wedge^k V \to \wedge ^{n-k} V
\end{equation*}
defined by $\star(e_{i_1}\wedge \ldots \wedge e_{i_k}) = \pm (e_{j_1}\wedge \ldots \wedge e_{j_{n-k}})$, where $\{j_1, \ldots , j_{n-k}\}$ is complement of $\{i_1,\ldots, i_k\}$ in $\{1,\ldots,n\}$, and we assign the plus sign if $\{i_1, \ldots , i_k, j_1, \ldots ,j_{n-k}\}$ is an even permutation of $\{1,\ldots,n\}$, and assign the minus sign if it is an odd permutation.

It is easy to extend $\star$ by linearity, and also to observe that $\star$ does not depend on the choice of the orthonormal basis, and depends only on the inner-product structure.

Let $X$ be a complex manifold of dimension $n$, and choose a hermitian metric $g$ on $X$. Then similar to above, we can define the Hodge $\star$-operator 
\begin{equation*}
\star:\wedge ^k T^*X \to \wedge^{n-k} T^*X
\end{equation*} 
and it is easy to see that there is an induced $\star$-operator
\begin{equation*}
\star: \wedge^{p,q} T^*X\to \wedge ^{n-p, n-q} T^*X.
\end{equation*}
Let $E$ be a holomorphic vector bundle on $X$, and choose a hermitian metric $h$ on $E$. We can consider the metric as a linear map $h:E\to E^*$, where $E^*$ is the dual vector bundle to $E$, and we also have the dual linear map $h^*:E^*\to E$, and these satisfy $h^*h=Id_E, h h^* = Id_{E^*}$. Let $\bar{\star}(f) \coloneqq  \star(\bar{f})$ for a section $f$ of $\wedge^{*,*}T^*X$. Define
\begin{equation*}
\bar{ \star}_E = \bar{\star} \otimes h: \wedge^{p,q}T^*X \otimes E \to \wedge^{n-q,n-p}T^*X\otimes E^*.
\end{equation*}
Then one can show \cite[5.2.4.a.]{wells2007differential} the following relation between the adjoint $\bar{\partial}^*$ of the Dolbeault operator $\bar{\partial}$ and the Hodge $\star$-operator.
\begin{equation} \label{hodge-star operator and adjoint dolbeault operator - equation}
\bar{\partial}^* = - \bar{\star}_{E^*} \bar{\partial} \bar{\star}_E : \mathscr{A}_X^{p,q}( E) \to \mathscr{A}_X^{p,q-1}(E).
\end{equation} 
\end{Def}

\section{Functional Calculus} \label{functional calc appendix}

Let us follow \cite{higson2000analytic} to give a quick introduction to functional calculus. % A reader familiar with the subject can skip the following series of definitions and lemmas.% and start at example \ref{Dolbeault complex symbol seq example}. 

\begin{Def}
Let $T\in \mathfrak{B}(H)$. Then let $C^*(T)$ temporarily denote the Banach subalgebra of $\mathfrak{B}(H)$, generated by $T$, its adjoint $T^*$, and the identity operator. 

We say that the operator $T$ is \emph{normal} if $TT^*=T^*T$. If $T$ is normal then $C^*(T)$ is a commutative Banach algebra.
\end{Def}

Let $A$ be a unital Banach algebra. Then for $a\in A$, we define
\begin{equation*}
\text{Spectrum}_A(a)= \{\lambda\in \mathbb{C}: a-\lambda.1 \text{ is not invertible in } A\}.
\end{equation*}    

\begin{prop}[Spectral Theorem]\cite[1.1.11.]{higson2000analytic}
Let $T\in \mathfrak{B}(H)$ be a bounded normal operator acting on the Hilbert space $H$, then the map $\alpha \mapsto \alpha(T)$ is a homomorphism from dual of $C^*(T)$ onto $\text{Spectrum}_{\mathfrak{B}(H)}(T)$, and the induced \emph{Gelfand transform} $C^*(T)\to \mathscr{C}(\text{Spectrum}_{\mathfrak{B}(H)}(T))$ is an isometric $*$-isomorphism.
\end{prop}

\begin{Def}
Let $T\in \mathfrak{B}(H)$ be a bounded normal operator acting on the Hilbert space $H$, and let $f\in \mathscr{C}(\text{Spectrum}_{\mathfrak{B}(H)}(T))$. Denote the corresponding element in $C^*(T)$ by $f(T)$. The $*$-homomorphism (inverse of the Gelfand transform) $\mathscr{C}(\text{Spectrum}_{\mathfrak{B}(H)}(T)) \to \mathfrak{B}(H)$ defined by $f\mapsto f(T)$ is called \emph{functional calculus} for $T$.   
\end{Def}

\begin{Def}\cite[1.8.]{higson2000analytic}
Let $T$ be an unbounded operator, defined over a dense subset of the Hilbert space $H$. Then we say $T$ is \emph{symmetric} if for each $x,y\in H$ which are in domain of $T$, we have $\langle Tx, y \rangle = \langle x, Ty \rangle $.

We say $T$ is \emph{essentially self-adjoint} if domain of $T$ is a subset of domain of $T^*$, and for any $x$ in domain of $T$, $Tx=T^*x$, and also $x$ is in domain of $T^*$ if there is a sequence of points $\{x_i\}_{i=1}^{\infty}$ in domain of $T$ so that $x_i$'s converge to $x$ and $\Vert T(x_i)\Vert$ remains bounded. 

Note that the first two conditions are equivalent to $T$ being symmetric. In other words, every essentially self-adjoint unbounded operator is symmetric.
\end{Def}

\begin{lem} \cite[10.2.6.]{higson2000analytic} \label{compactly supported symm diff op is self adj lem}
Every symmetric differential operator on a compact manifold is essentially self-adjoint. More generally, every compactly supported symmetric differential operator on a (non-compact) manifold is essentially self-adjoint.
\end{lem}

\end{appendices}

\bibliographystyle{alpha}
\bibliography{bibfile}

\end{document}